\documentclass[reqno,centertags, 12pt]{amsart}
\usepackage{amsmath,amsthm,amscd,amssymb}
\usepackage{latexsym}
\newcommand{\bbC}{{\mathbb{C}}}
\newcommand{\bbD}{{\mathbb{D}}}
\newcommand{\bbE}{{\mathbb{E}}}

\newcommand{\bbR}{{\mathbb{R}}}

\newcommand{\bbZ}{{\mathbb{Z}}}

\newcommand{\calB}{{\mathcal{B}}}

\newcommand{\calE}{{\mathcal{E}}}

\newcommand{\calM}{{\mathcal M}}
\newcommand{\calP}{{\mathcal P}}

\newcommand{\dott}{\,\cdot\,}

\newcommand{\lb}{\label}
\newcommand{\f}{\frac}
\newcommand{\ul}{\underline}
\newcommand{\ol}{\overline}
\newcommand{\ti}{\tilde  }
\newcommand{\wti}{\widetilde  }

\newcommand{\cvh}{\text{\rm{cvh}}}

\newcommand{\dist}{\text{\rm{dist}}}

\newcommand{\capa}{\text{\rm{cap}}}
\newcommand{\disc}{\text{\rm{disc}}}
\newcommand{\spec}{\text{\rm{spec}}}

\newcommand{\ran}{\text{\rm{ran}}}

\newcommand{\ess}{\text{\rm{ess}}}
\newcommand{\ac}{\text{\rm{ac}}}

\newcommand{\s}{\text{\rm{s}}}
\newcommand{\dc}{\text{\rm{d}}}

\newcommand{\intt}{\text{\rm{int}}}

\newcommand{\supp}{\text{\rm{supp}}}

\newcommand{\bi}{\bibitem}

\newcommand{\beq}{\begin{equation}}
\newcommand{\eeq}{\end{equation}}
\newcommand{\ba}{\begin{align}}
\newcommand{\ea}{\end{align}}
\newcommand{\veps}{\varepsilon}

\newcounter{smalllist}
\newenvironment{SL}{\begin{list}{{\rm\roman{smalllist})}}{%
\setlength{\topsep}{0mm}\setlength{\parsep}{0mm}\setlength{\itemsep}{0mm}%
\setlength{\labelwidth}{2em}\setlength{\leftmargin}{2em}\usecounter{smalllist}%
}}{\end{list}}

\newcommand{\bigtimes}{\mathop{\mathchoice%
{\smash{\vcenter{\hbox{\LARGE$\times$}}}\vphantom{\prod}}%
{\smash{\vcenter{\hbox{\Large$\times$}}}\vphantom{\prod}}%
{\times}%
{\times}%
}\displaylimits}

\DeclareMathOperator{\Real}{Re}
\DeclareMathOperator{\Ima}{Im}
\DeclareMathOperator{\diam}{diam}

\DeclareMathOperator*{\wlim}{w-lim}

\allowdisplaybreaks
\numberwithin{equation}{section}

\newtheorem{theorem}{Theorem}[section]

\newtheorem*{p2.1}{Proposition 2.1}
\newtheorem{proposition}[theorem]{Proposition}
\newtheorem{mt}[theorem]{Metatheorem}
\newtheorem{lemma}[theorem]{Lemma}
\newtheorem{corollary}[theorem]{Corollary}
\theoremstyle{definition}
\newtheorem{example}[theorem]{Example}
\newtheorem{conjecture}[theorem]{Conjecture}
\newtheorem{oq}[theorem]{Open Question}
\newtheorem{op}[theorem]{Open Project}

\theoremstyle{remark}
\newtheorem*{definition}{Definition}

\theoremstyle{definition}
\newtheorem*{remark}{Remark}
\newtheorem*{remarks}{Remarks}

\newcommand{\abs}[1]{\lvert#1\rvert}

\begin{document}
\title{Equilibrium Measures and Capacities in Spectral Theory}
\author[B.~Simon]{Barry Simon$^*$}

\thanks{$^*$ Mathematics 253-37, California Institute of Technology, Pasadena, CA 91125.
E-mail: bsimon@caltech.edu. Supported in part by NSF grants DMS-0140592 and DMS-0652919 and
U.S.--Israel Binational Science Foundation (BSF) Grant No.\ 2002068}

\date{August 23, 2007}

\begin{abstract}  This is a comprehensive review of the uses of potential theory in
studying the spectral theory of orthogonal polynomials.  Much of the article
focuses on the Stahl--Totik theory of regular measures, especially the case
of OPRL and OPUC.  Links are made to the study of ergodic Schr\"odinger
operators where one of our new results implies that, in complete generality,
the spectral measure is supported on a set of zero Hausdorff dimension
(indeed, of capacity zero) in the region of strictly positive Lyapunov
exponent.  There are many examples and some new conjectures and indications
of new research directions.  Included are appendices on potential theory and
on Fekete--Szeg\H{o} theory.

\end{abstract}

\maketitle

\section{Introduction} \lb{s1}

This paper deals with applications of potential theory to spectral and inverse spectral theory,
mainly to orthogonal polynomials especially on the real line (OPRL) and unit circle (OPUC).
This is an area that has traditionally impacted both the orthogonal polynomial community and
the spectral theory community with insufficient interrelation. The OP approach emphasizes the
procedure of going from measure to recursion parameters, that is, the inverse spectral problem,
while spectral theorists tend to start with recursion parameters and so work with direct
spectral theory.

Potential theory ideas in the orthogonal polynomial community go back at least to a deep
1919 paper of Faber \cite{Faber} and a seminal 1924 paper of Szeg\H{o} \cite{Sz24} with
critical later developments of Kalm\'ar \cite{Kal} and Erd\"os--Tur\'an \cite{ET}. The modern
theory was initiated by Ullman \cite{Ull} (see also \cite{Ull84,Ull85,Ull89,UW,TU,UWZ,Wyn}
and earlier related work of Korovkin \cite{Kor} and Widom \cite{Wid}), followed by an often
overlooked paper of Van Assche \cite{vA}, and culminating in the comprehensive and deep
monograph of Stahl--Totik \cite{StT}. (We are ignoring the important developments connected to
variable weights and external potentials, which are marginal to the themes we study;
see \cite{ST} and references therein.)

On the spectral theory community side, theoretical physicists rediscovered Szeg\H{o}'s potential
theory connection of the growth of polynomials and the density of zeros---this is called the Thouless
formula after \cite{Thou}, although discovered slightly earlier by Herbert and Jones \cite{HJ}.
The new elements involve ergodic classes of potentials, especially Kotani theory (see
\cite{Kot,S168,Kot89,Kot97,S169,Dam07}).

One purpose of this paper is to make propaganda on both sides: to explain some of the main aspects
of the Stahl--Totik results to spectral theorists and the relevant parts of Kotani theory to
the OP community. But this article looks forward even more than it looks back. In thinking through
the issues, I realized there were many interesting questions to examine. Motivated in part by
the remark that one can often learn from wrong conjectures \cite{Jit07}, I make several conjectures
which, depending on your point of view, can be regarded as either bold or foolhardy (I especially
have Conjectures~\ref{Con8.7} and \ref{Con8.10} in mind).

The potential of a measure $\mu$ on $\bbC$ is defined by
\begin{equation}\lb{1.1}
\Phi_\mu(x) =\int \log\abs{x-y}^{-1}\, d\mu(y)
\end{equation}
which, for each $x$, is well defined (although perhaps $\infty$) if $\mu$ has compact support.
The relevance of this to polynomials comes from noting that if $P_n$ is a monic polynomial,
\begin{equation}\lb{1.2}
P_n(z)=\prod_{j=1}^n (z-z_j)
\end{equation}
and $d\nu_n$ its zero counting measure, that is,
\begin{equation}\lb{1.3}
\nu_n =\f{1}{n} \sum_{j=1}^n \delta_{z_j}
\end{equation}
the point measure with $n\nu_n(\{w\})=$ multiplicity of $w$ as a root of $P_n$, then
\begin{equation}\lb{1.4}
\abs{P_n(z)}^{1/n} =\exp (-\Phi_{\nu_n}(z))
\end{equation}

If now $d\mu$ is a measure of compact support on $\bbC$, let $X_n(z)$ and $x_n(z)$ be the monic
orthogonal and orthonormal polynomials for $d\mu$, that is,
\begin{equation}\lb{1.5}
X_n(z)=z^n + \text{lower order}
\end{equation}
with
\begin{equation}\lb{1.6}
\int \ol{X_n(z)}\, X_m (z)\, d\mu(z) =\|X_n\|_{L^2}^2 \delta_{nm}
\end{equation}
and
\begin{equation}\lb{1.7}
x_n(z) =\f{X_n(z)}{\|X_n\|_{L^2}}
\end{equation}
Here and elsewhere $\|\cdot\|$ without a subscript means the $L^2$ norm for the
measure currently under consideration.

When $\supp(d\mu)\subset\bbR$, we use $P_n, p_n$ and note (see \cite{Szb,FrB}) there are
Jacobi parameters $\{a_n,b_n\}_{n=1}^\infty \in [(0,\infty)\times\bbR]^\infty$, so
\begin{gather}
xp_n(x) = a_{n+1} p_{n+1}(x) + b_{n+1} p_n(x) + a_n p_{n-1}(x) \lb{1.8} \\
\|P_n\| =a_1 \dots a_n \mu(\bbR)^{1/2} \notag
\end{gather}
and if $\supp(d\mu)\subset\partial\bbD$, the unit circle, we use $\Phi_n,\varphi_n$ and
note (see \cite{Szb,GBk,OPUC1,OPUC2}) there are Verblunsky coefficients
$\{\alpha_n\}_{n=0}^\infty \in\bbD^\infty$, so
\begin{gather}
\Phi_{n+1}(z) =z\Phi_n - \bar\alpha_n \Phi_n^*(z) \lb{1.9} \\
\|\Phi_n(z)\| = \rho_0 \dots \rho_{n-1} \mu(\partial\bbD)^{1/2} \lb{1.10}
\end{gather}
where
\begin{equation}\lb{1.11}
\Phi_n^*(z) =z^n \, \ol{\Phi_n (1/\bar z)} \qquad
\rho_j =(1-\abs{\alpha_j}^2)^{1/2}
\end{equation}
As usual, we will use $J$ for the Jacobi matrix formed from the parameters in the OPRL case,
that is, $J$ is tridiagonal with $b_j$ on diagonal and $a_j$ off-diagonal.

The $X_n$ minimize $L^2$ norms, that is,
\begin{equation}\lb{1.12}
\|X_n\|_{L^2 (d\mu)} = \min\{\|Q_n\|_{L^2} \mid Q_n(z)=z^n + \text{lower order}\}
\end{equation}
Given a compact $E\subset\bbC$, the Chebyshev polynomials are defined by ($L^\infty$ is the
sup norm over $E$)
\begin{equation}\lb{1.13}
\|T_n\|_{L^\infty (E)} =\min \{\|Q_n\|_{L^\infty} \mid Q_n(z) =z^n + \text{lower order}\}
\end{equation}
These minimum conditions suggest that extreme objects in potential theory, namely, the capacity,
$C(E)$, and equilibrium measure, $d\rho_E$, discussed in \cite{Helms,Land,M-F,Ran,ST} and Appendix~A
will play a role (terminology from Appendix~A is used extensively below). In fact, going back
to Szeg\H{o} \cite{Sz24} (we sketch a proof in Appendix~B), one knows

\begin{theorem}[Szeg\H{o} \cite{Sz24}]\lb{T1.1} For any compact $E\subset\bbC$ with
Chebyshev polynomials, $T_n$, one has
\begin{equation}\lb{1.14}
\lim_{n\to\infty}\, \|T_n\|_{L^\infty (E)}^{1/n} =C(E)
\end{equation}
\end{theorem}

This has an immediate corollary (it appears with this argument in Widom \cite{Wid} and may well
be earlier):

\begin{corollary}\lb{C1.2} Let $\mu$ be a measure of compact support, $E$, in $\bbC$. Let $X_n
(z;d\mu)$ be its monic OPs. Then
\begin{equation}\lb{1.15}
\limsup_{n\to\infty}\, \|X_n\|_{L^2 (\bbC,d\mu)}^{1/n} \leq C(E)
\end{equation}
\end{corollary}

\begin{proof} By \eqref{1.12},
\begin{equation}\lb{1.16}
\|X_n\|_{L^2 (\bbC,d\mu)}^{1/n} \leq \|T_n\|_{L^2 (\bbC,d\mu)}^{1/n}
\end{equation}
where $T_n$ are the Chebyshev polynomials for $E$. On $E$, $\abs{T_n(z)} \leq \|T_n\|_{L^\infty (E)}$
so, since $\supp(d\mu)=E$,
\begin{equation}\lb{1.17}
\|X_n\|_{L^2(\bbC,d\mu)}^{1/n} \leq \|T_n\|_{L^\infty (E)}^{1/n} \mu(E)^{1/2n}
\end{equation}
and \eqref{1.15} follows from \eqref{1.14}.
\end{proof}

For OPRL and OPUC, the relation \eqref{1.15} says
\begin{alignat}{2}
\limsup (a_1 \dots a_n)^{1/n} &\leq C(E) \qquad &&\text{(OPRL)} \lb{1.18x} \\
\limsup (\rho_1 \dots \rho_n)^{1/n} & \leq C(E) \qquad &&\text{(OPUC)} \lb{1.19x}
\end{alignat}
\eqref{1.18x} is a kind of thickness indication of the spectrum of discrete Schr\"odinger
operators (with $a_j\equiv 1$) where it is not widely appreciated that $C(E)\geq 1$.

In many cases that occur in spectral theory, one considers discrete and essential spectrum.
In this context, $\sigma_\ess(d\mu)$ is the nonisolated points of $\supp(d\mu)$. $\sigma_\dc
(d\mu)=\supp(d\mu)\setminus\sigma_\ess(d\mu)$ is a countable discrete set. If $d\nu$ is
any measure with finite Coulomb energy $\nu (\sigma_\dc (d\mu))=0$, thus $C(\supp(d\mu)) =
C(\sigma_\ess(d\mu))$; so we will often consider $E=\sigma_\ess(d\mu)$ in \eqref{1.15}.
In fact, as discussed in Appendix~A after Theorem~\ref{TA.8}, we should take
$E=\sigma_\capa (d\mu)$.

The inequality \eqref{1.16} suggests singling out a special case. A measure $d\mu$ of compact
support, $E$, on $\bbC$ is called {\it regular\/} if and only if
\begin{equation}\lb{1.13x}
\lim_{n\to\infty}\, \|X_n\|_{L^2(d\mu)}^{1/n} =C(E)
\end{equation}
For $E=[-1,1]$, this class was singled out by Ullman \cite{Ull}; the general case is due to
Stahl--Totik \cite{StT}.

\begin{example}\lb{E1.3} The {\it Nevai class\/}, $N(a,b)$, with $a>0$, $b\in\bbR$, is the set
of probability measures on $\bbR$ whose Jacobi parameters obey
\begin{equation}\lb{1.14x}
a_n\to a \qquad b_n\to b
\end{equation}
The Jacobi matrix with $a_n\equiv a$, $b_n\equiv b$ is easily seen to have spectrum
\begin{equation}\lb{1.15x}
E(a,b) =[b-2a, b+2a]
\end{equation}
so, by \eqref{A.5a},
\begin{equation}\lb{1.16x}
C(E(a,b))=a
\end{equation}
By Weyl's theorem, if $\mu\in N(a,b)$, $\sigma_\ess(\mu)=E(a,b)$, so $\mu$ is regular.
\qed
\end{example}

\begin{example}\lb{E1.4} Here is an example of a regular measure on $\bbR$ not in a Nevai class.
Let $d\mu$ be the measure with Jacobi parameters $b_n\equiv 0$ and
\begin{equation}\lb{1.17x}
a_n=\begin{cases}
1 & n\neq k^2\text{ for all } k \\
\f12 & n=k^2\text{ for some } k
\end{cases}
\end{equation}
Clearly, $\lim(a_1\dots a_n)^{1/n} =1$, so if $\supp(d\mu)=[-2,2]$ (which has capacity $1$ by
\eqref{1.16x}), we will have a regular measure not in a Nevai class. Since $\left(
\begin{smallmatrix} 0&c\\c&0\end{smallmatrix}\right) \geq\left(\begin{smallmatrix}
-c&0\\ 0&-c\end{smallmatrix}\right)$ for any $c>0$, the Jacobi matrix $J$ associated to
$d\mu$ is bounded below by a diagonal matrix with elements either $-\f12$, $-1$, $-\f32$,
or $-2$ (for $n=1,2,k^2$ or $k^2+1$ and otherwise), so $J\geq -2$. Similarly, $J\leq 2$.
Thus, $\sigma(J)=\supp(d\mu)\subset [-2,2]$. On the other hand, since $\lim(a_1 \dots a_n)^{1/n}
=1$, \eqref{1.18x} implies $C(E)\geq 1$ where $E=\supp(d\mu)$. If $E\subsetneqq [-2,2]$, it is
missing an open subset and so $\abs{E}<4$ and $C(E) <1$ (by \eqref{A.35}). Thus $C(E)\geq 1$
implies $E=[-2,2]$. Alternatively, using plane wave trial functions cut off to live in
$[k^2+2, (k+1)^2-1]$, we easily see directly that $[-2,2]\subset\sigma(J)$.

This example has no a.c.\ spectrum by results of Remling \cite{Remppt}. In Section~\ref{s8}
(see Example~\ref{E8.11}), we have models which are regular, not in Nevai class with nonempty
a.c.\ spectrum.
\qed
\end{example}

\begin{example}\lb{E1.4A} The CN (for Ces\`aro--Nevai) class was introduced by Golinskii-Khrushchev
\cite{KhGo} for OPUC and it has an OPRL analog. For OPRL, it says
\begin{equation}\lb{1.24a}
\f{1}{n}\, \biggl[\, \sum_{j=1}^n \, \abs{a_j-1} + \abs{b_j}\biggr] \to 0
\end{equation}

Example~\ref{E1.4} is in this class and is regular, but it is not true that every element of CN for
OPRL is regular; for example, if $a_n\equiv 1$ and each $b_j$ is $0$ or $1$ so \eqref{1.24a} holds
but with arbitrarily long strings of only $0$'s and also of only $1$'s, e.g., $b_j=1$ if $n^2
\leq j\leq n^2+n$ and $b_j=0$ otherwise. Then $\sigma (J)=[-2,3]$ with $C(\sigma(J))=\f54$, but
$(a_1 \dots a_n)^{1/n}\to 1$. (To see that $\sigma (J)=[-2,3]$, let $J_+, J_-$ be the matrices
with $a_n\equiv 1$ and $b_n\equiv 1$ (for $J_+$) and $b_n\equiv 0$ (for $J_-$). Then $-2\leq
J_-\leq J\leq J_+\leq 3$, so $\sigma(J)\subset [-2,3]$. On the other hand, because of the long
strings, a variational argument shows $\sigma(J_+)\cup\sigma (J_-)\subset \sigma(J)$.) However,
for OPUC, the subset of CN class with $\sup_n\abs{\alpha_n}<1$ consists of only regular measures.
For by Theorem~4.3.17 of \cite{OPUC1}, $\sigma (d\mu)=\partial\bbD$ with capacity $1$. On the
other hand, if $A=\sup_n\abs{\alpha_n}$ and
\[
L(A)=-\f{\log (1-A)}{2A}
\]
then (since $\log \abs{1-x}$ is convex)
\[
-\log \rho_j \leq L(A) \abs{\alpha_j}^2
\]
so
\[
\f{1}{n}\, \sum_{j=0}^{n-1} \log\rho_j \leq L(A) \, \f{1}{n}\, \sum_{j=0}^{n-1} \,
\abs{\alpha_j}^2 \leq L(A)\, \f{1}{n} \, \sum_{j=0}^{n-1} \, \abs{\alpha_j}
\]
goes to zero so $(\rho_0\dots \rho_{n-1})^{1/n}\to 1$. It is easy to see that if $\sup_n
\abs{\alpha_n}<1$ is dropped, regularity can be lost.
\qed
\end{example}

\begin{example}\lb{E1.5} Let $a_n\equiv\f12$ and let $b_n=\pm 1$ chosen as identically
distributed random variables. As above, all these random $J$'s have $-2\leq J\leq 2$, so
$\supp(d\mu)\subset [-2,2]$. Since there will be, with probability $1$, long stretches of
$b_n\equiv 1$ or $b_n\equiv -1$, it is easy to see $\supp(d\mu)\supset ([-1,1]+1)\cup
([-1,1]-1)=[-2,2]$. Thus, a typical random $d\mu$ has support $[-2,2]$ with capacity $1$,
but obviously $\lim (a_1 \dots a_n)^{1/n}=\f12$. This shows there are measures which are
not regular. By Example~\ref{E1.3}, random slowly decaying Jacobi matrices are regular, so
neither randomness nor pure point measures necessarily destroy regularity. We return to
pure point measures in Theorem~\ref{T5.5} and Corollary~\ref{C5.6}.
\qed
\end{example}

Section~\ref{s8} has many more examples of regular measures. Regularity is important because
of its connections to zero distributions and to root asymptotics. Let $d\nu_n$ denote the
zero distribution for $X_n(z)$ defined by \eqref{1.3}. Then

\begin{theorem}\lb{T1.6} Let $d\mu$ be a measure on $\bbR$ with $\sigma_\ess(d\mu)=E$ compact
and $C(E)>0$. If $\mu$ is regular, then $d\nu_n$ converges weakly to $d\rho_E$, the equilibrium
measure for $E$.
\end{theorem}

\begin{remarks} 1. For $E=[0,1]$, ideas close to this occur in Erd\"os--Tur\'an \cite{ET}. The
full result is in Stahl--Totik \cite{StT} who prove a stronger result. Rather than $E\subset\bbR$,
they need that the unbounded component, $\Omega$, of $\bbC\setminus E$ is dense in $\bbC$.
We will prove Theorem~\ref{T1.6} in Section~\ref{s2}.

\smallskip
2. The result is false for measures on $\partial\bbD$. Indeed, it fails for $d\mu=\f{d\theta}
{2\pi}$, that is, $\alpha_n\equiv 0$ and $\Phi_n(z)=z^n$. However, there is a result for
paraorthogonal polynomials and for the balayage of $d\nu_n$. Theorem~\ref{T1.6} is true if
$\supp(d\mu)\subset\partial\bbD$ but is not all of $\partial\bbD$. We will discuss
this further in Section~\ref{s3} where we also prove a version of Theorem~\ref{T1.6}
for OPUC.
\end{remarks}

For later purposes, we note

\begin{proposition}\lb{P1.6A} Let $d\mu$ be a measure on $\bbR$ with $\sigma_\ess (d\mu)=E$
compact. Then any limit of $d\nu_n$ is supported on $E$.
\end{proposition}

\begin{proof} It is known (see \cite[Sect.~1.2]{OPUC1}) that if $(a,b)\cap\supp(d\mu)=
\emptyset$, then $P_n(x)$ has at most one zero in $(a,b)$. It follows that if $e$ is an
isolated point of $\supp(d\mu)$, then $(e-\delta, e+\delta)$ has at most three zeros for
$\delta$ small (with more argument, one can get two). Thus, points not in $E$ have
neighborhoods, $N$, with $\nu_n(N)\leq\f{3}{n}$.
\end{proof}

Stahl--Totik \cite{StT} also have the following almost converse (their Sect.~2.2)---for $E\subset
\bbR$, we prove a slightly stronger result; see Theorem~\ref{T2.4A}.

\begin{theorem}\lb{T1.7} Let $d\mu$ be a measure on $\bbR$ with $\sigma_\ess(d\mu)=E$ compact
and $C(E)>0$. Suppose that $d\nu_n\to d\rho_E$, the equilibrium measure. Then either $d\mu$
is regular or there exists a Borel set, $X$, with $d\mu(\bbR\setminus X)=0$ and $C(X)=0$.
\end{theorem}

\begin{remarks} 1. As an example where such an $X$ exists even though $C(\supp(\mu))>0$, consider
a $\mu$ which is dense pure point on $[-2,2]$.

\smallskip
2. In Section~\ref{s8}, we will see explicit examples where $d\nu_n\to d\rho_E$ but
$d\mu$ is not regular.
\end{remarks}

The other connection is to root asymptotics of the OPs. Recall the Green's function, $G_E(z)$,
is defined by \eqref{A.26a}; it vanishes q.e.\ (quasi-everywhere, defined in Appendix~A) on $E$,
is harmonic on $\Omega$, and asymptotic to $\log\abs{z}-\log C(E) +o(1)$ as $\abs{z}\to\infty$.
The main theorem on root asymptotics is:

\begin{theorem}\lb{T1.8} Let $E\subset\bbC$ be compact and let $\mu$ be a measure of compact
support with $\sigma_\ess(\mu)=E$. Then the following are equivalent:
\begin{SL}
\item[{\rm{(i)}}] $\mu$ is regular, that is, $\lim_{n\to\infty} \|X_n\|_{L^2(d\mu)}^{1/n}=C(E)$.
\item[{\rm{(ii)}}] For all $z$ in $\bbC$, uniformly on compacts,
\begin{equation}\lb{1.18}
\limsup\abs{x_n(z)}^{1/n} \leq e^{G_E(z)}
\end{equation}
\item[{\rm{(iii)}}] For q.e.\ $z$ in $\partial\Omega$ {\rm{(}}with $\Omega$ the unbounded connected
component of $\bbC\setminus E${\rm{)}}, we have
\begin{equation}\lb{1.19}
\limsup \abs{x_n(z)}^{1/n}\leq 1
\end{equation}
\end{SL}
Moreover, if {\rm{(i)--(iii)}} hold, then {\rm{(}}here $\cvh =$ closed convex hull{\rm{)}}
\begin{SL}
\item[{\rm{(iv)}}] For every $z\in\bbC\setminus\cvh(\supp(d\mu))$, we have
\begin{equation}\lb{1.20}
\lim_{n\to\infty}\, \abs{x_n(z)}^{1/n} = e^{G_E(z)}
\end{equation}
\item[{\rm{(v)}}] For q.e.\ $z\in\partial\Omega$,
\begin{equation}\lb{1.21}
\limsup_{n\to\infty} \, \abs{x_n(z)}^{1/n} =1
\end{equation}
\item[{\rm{(vi)}}] For any sequence $Q_n(z)$ of polynomials of degree $n$, for all $z\in\bbC$,
\begin{equation}\lb{1.22}
\limsup \biggl| \f{Q_n(z)}{\|Q_n\|_{L^2(d\mu)}}\biggr|^{1/n} \leq e^{G_E(z)}
\end{equation}
\end{SL}
\end{theorem}

\begin{remark} It is easy to see that (iv), (v) or (vi) are equivalent to (i)--(iii).
\end{remark}

For $E\subset\bbR$, we will prove this in Section~\ref{s2}. For $E\subset\partial\bbD$,
we prove it in Section~\ref{s3}.

The original result asserting cases where regularity holds was proven in 1940!

\begin{theorem}[Erd\"os--Tur\'an \cite{ET}]\lb{T1.9} Let $d\mu$ be supported on
$[-2,2]$ and suppose
\begin{equation}\lb{1.23}
d\mu(x) =w(x)\, dx + d\mu_\s
\end{equation}
with $d\mu_\s$ singular. Suppose $w(x) >0$ for a.e.\ $x$ in $[-2,2]$. Then $\mu$ is regular.
\end{theorem}

\begin{remarks} 1. Erd\"os--Tur\'an \cite{ET} worked on $[-1,1]$ and had $d\mu_\s =0$.

\smallskip
2. We now have a stronger result than this---namely, Rakhmanov's theorem (see \cite[Ch.~9]{OPUC2}).
If $w(x)>0$ for a.e.\ $x$, one knows $a_n\to 1$ (and $b_n\to 0$) much more than $(a_1 \dots a_n)^{1/n}
\to 1$ (equivalently, we have ratio asymptotics on the $p$'s and not just root asymptotics).
Regularity is a ``poor man's" Rakhmanov's theorem. But unlike Rakhmanov's theorem which is
only known for a few other $E$'s (see \cite{DKS07,Remppt} and the discussion in Section~\ref{s8}),
this weaker version holds very generally.

\smallskip
3. In this case, $d\rho_E$ is equivalent to $dx$, so \eqref{1.23} and $W(x) >0$ for a.e.\ $x$ is
equivalent to saying that $\rho_E$ is $\mu$-a.c.
\end{remarks}

In Section~\ref{s4}, we will prove the following vast generalization of the
Erd\"os--Tur\'an result:

\begin{theorem}[Widom \cite{Wid}] \lb{T1.10} Let $\mu$ be a measure on $\bbR$ with compact
support and $E=\sigma_\ess (d\mu)$ and $C(E) >0$. Suppose $d\rho_E$ is the equilibrium measure
for $E$ and
\begin{equation}\lb{1.24}
d\mu = w(x)\, d\rho_E(x) + d\mu_\s
\end{equation}
where $d\mu_\s$ is $d\rho_E$-singular. Suppose $w(x)>0$ for $d\rho_E$-a.e.\ $x$. Then $\mu$ is
regular.
\end{theorem}

\begin{remarks} 1. As above, $\text{\eqref{1.24}} + w(x) >0$ is equivalent to saying that $d\rho_E$
is absolutely continuous with respect to $d\mu$.

\smallskip
2. Widom's result is much more general than what we have in this theorem. His $E$ is a general
compact set in $\bbC$. His polynomials are defined by general families of minimum conditions,
for example, $L^p$ minimizers. Most importantly, he has a general family of support conditions
that, as he notes in a one-sentence remark, include the case where $d\rho_E$ is a.c.\ with
respect to $d\mu$. Because of its spectral theory connection, we have focused on the $L^2$ minimizers,
although it is not hard to accommodate more general ones.  We focus on the $w>0$ case because
if one goes beyond that, it is better to look at conditions that depend on weights and not
just supports of the measure as Widom (and Ullman \cite{Ull}) do (see Theorem~\ref{T1.11} below).

\smallskip
3. In Section~\ref{s4}, we will give a proof of this theorem due to Van Assche \cite{vA} who
mentions Widom's paper but says it is not clear his hypotheses apply despite an explicit (albeit
terse) aside in Widom's paper. Stahl--Totik \cite{StT} state Theorem~\ref{T1.10} explicitly.
They seem to be unaware of Van Assche's paper or Widom's aside.
\end{remarks}

It was Geronimus \cite{Ger} who seems to have first noted that there are non-a.c.\ measures which are
regular (and later Widom \cite{Wid} and Ullman \cite{Ull}). Of course,  with the discovery of Nevai
class measures which are not a.c. \cite{Sim157,Sim178,Totik,Lub88,Lub91,MagVan,S234,S265}, there are
many examples, but given a measure, one would like to know effective criteria. Stahl--Totik
\cite[Ch.~4]{StT} have many, of which we single out:

\begin{theorem}[Stahl--Totik \cite{StT}]\lb{T1.11} Let $E$ be a finite union of disjoint closed
intervals in $\bbR$. Suppose $\mu$ is a measure on $\bbR$ with $\sigma_\ess (d\mu)=E$, and for any
$\eta >0$ {\rm{(}}$\abs{\cdot}$ is Lebesgue measure{\rm{)}},
\begin{equation}\lb{1.25}
\lim_{m\to\infty}\, \bigl| \bigl\{x\bigm| \mu ([x-\tfrac{1}{m}\, , x+\tfrac{1}{m}])
\leq e^{-m\eta}\bigr\}\bigr| =0
\end{equation}
Then $\mu$ is regular.
\end{theorem}

\begin{theorem}[Stahl--Totik \cite{StT}]\lb{T1.12}  Let $E$ be a finite union of disjoint closed
intervals in $\bbR$. Suppose $\mu$ is a measure on $\bbR$ and that $\mu$ is regular. Then for any
$\eta >0$,
\begin{equation}\lb{1.26}
\lim_{m\to\infty} \, C \bigl( \bigl\{x\in E \bigm| \mu ([x-\tfrac{1}{m}\, , x+\tfrac{1}{m}])
\leq e^{-m\eta}\bigr\}\bigr) =0
\end{equation}
\end{theorem}

\begin{remarks} 1. We will prove Theorem~\ref{T1.11} in Section~\ref{s5}.

\smallskip
2. Stahl--Totik \cite{StT} state these results for $E=[-1,1]$, but it is easy to accommodate
finite unions of disjoint closed intervals; see Corollary~\ref{C6.6}.
\end{remarks}

In Section~\ref{s6}, we turn to structural results (all due to Stahl--Totik \cite{StT})
connected to inheritance of regularity when measures have a relation, for example, when
restrictions of regular measures are regular.

Section~\ref{s7} discusses relations of potential theory and ergodic Jacobi matrices. This theory
concerns OPRL (or OPUC) whose recursion coefficients are samples of an ergodic process---as examples,
totally random or almost periodic cases. In that case, various ergodic theorems guarantee the existence
of $\lim (a_1 \dots a_n)^{1/n}$, of $d\nu_\infty \equiv \lim d\nu_n$, and of a natural Lyapunov
exponent, $\gamma(z)$, which off of $\supp(d\mu)$ is $\lim \abs{p_n(z;d\mu)}^{1/n}$ and is
subharmonic on $\bbC$. In that section, we will prove some of the few new results of this paper:

\begin{theorem}\lb{T1.13} Let $d\mu_\omega$ be the measures associated to an ergodic family of OPRL,
$d\nu_\infty$ and $\gamma$ its density of states and Lyapunov exponent. Let $E=\supp(d\nu_\infty)$.
Then the following are equivalent:
\begin{alignat}{2}
&\text{\rm{(a)}} \qquad && \gamma(x)=0 \text{ for } d\rho_E\text{-a.e. } x \notag \\
&\text{\rm{(b)}}  \qquad && \lim_{n\to\infty}\, (a_1 (\omega)\dots a_n(\omega))^{1/n} =C(E) \lb{1.27}
\end{alignat}
for a.e.\ $\omega$. Moreover, if \text{\rm{(a)}} and \text{\rm{(b)}} hold, then
\begin{equation}\lb{1.34a}
d\nu_\infty (x) =d\rho_E (x)
\end{equation}
with $d\rho_E$ the equilibrium measure for $E$. Conversely, if \eqref{1.34a} holds,
either \text{\rm{(a)}} and \text{\rm{(b)}} hold, or else for a.e.\ $\omega$, $d\mu_\omega$ is
supported on a set of capacity zero.
\end{theorem}

\begin{remarks} 1. We will prove that for a.e.\ $\omega$,
\begin{equation}\lb{1.34b}
\lim_{n\to\infty}\, (a_1 (\omega) \dots a_n (\omega))^{1/n} =C(E) \exp \biggl( -\int \gamma(x)\,
d\rho_E(x)\biggr)
\end{equation}

\smallskip
2. In Section~\ref{s8}, we will see examples where $d\nu_\infty =d\rho_E$
but \eqref{1.27} fails. Of course, (a) also fails.
\end{remarks}

The following is an ultimate version of what is sometimes called the Pastur--Ishii theorem
(see Section~\ref{s7}).

\begin{theorem}\lb{T1.15A}  Let $d\mu_\omega$ be a family of measures associated to an ergodic
family of OPRL and let $\gamma$ be its Lyapunov exponent. Let $S\subset\bbR$ be the Borel
set of $x\in\bbR$ with $\gamma(x) >0$. Then for a.e.\ $\omega$, there exists $Q_\omega$
of capacity zero so $d\mu_\omega (S\setminus Q_\omega) =0$. In particular, $d\mu_\omega
\restriction S$ is of local Hausdorff dimension zero.
\end{theorem}

We should explain what is really new in this theorem. It has been known since Pastur \cite{Pas80}
and Ishii \cite{Ishii} that for ergodic Schr\"odinger operators, the spectral measures are
supported on the eigenvalues union the bad set where Lyapunov behavior fails (this bad set
actually occurs, e.g., \cite{AvS83,S236}). The classic result is that the bad set has
Lebesgue measure zero. The new result here (elementary given a potential theoretic
point of view!) is that the bad set has capacity zero.

Section~\ref{s8} describes examples, open questions, and conjectures. Section~\ref{s9}
has some remarks on the possible extensions of these ideas to continuum Schr\"odinger
operators. Appendix~A is a primer of potential theory and Appendix~B proves
Theorem~\ref{T1.1} on Chebyshev polynomials.

\medskip
It is a pleasure to thank Jonathan Breuer, Jacob Christiansen, David Damanik, Svetlana
Jitomirskaya, Yoram Last, Christian Remling, Vilmos Totik, and Maxim Zinchenko for useful
discussions. I would also like to thank Ehud de Shalit and Yoram Last for the hospitality
of the Einstein Institute of Mathematics of the Hebrew University where part of this paper
was written.

\section{Regular Measures for OPRL} \lb{s2}

In this section, our main goal is to prove Theorems~\ref{T1.6} and \ref{T1.8} for OPRL. The key
will be a series of arguments familiar to spectral theorists as the Thouless formula, albeit
in a different (nonergodic) guise. The key will be an analog of positivity of the Lyapunov
exponent off the spectrum.

\begin{lemma} \lb{L2.1}
\begin{SL}
\item[{\rm{(a)}}] Let $J$ be a bounded Jacobi matrix and let $H$ be the convex hull of the
spectrum of $J$. For any $\varphi\in L^2(\bbR,d\mu)$,
\begin{equation} \lb{2.1}
\abs{\langle\varphi, (J-z)\varphi\rangle} \geq \dist(z,H) \|\varphi\|^2
\end{equation}
\item[{\rm{(b)}}] The Jacobi parameters $a_n$ obey {\rm{(}}recall $a_n >0${\rm{)}}
\begin{equation} \lb{2.2}
a_n\leq \tfrac12\, \diam(H)
\end{equation}
\end{SL}
\end{lemma}

\begin{proof} In a spectral representation, $J$ is multiplication by $x$. If $d=\dist(z,H)$,
there is $\omega\in\partial\bbD$ with $\Real[(x-z)\omega] \geq d$ for all $x\in H$. Thus
$\Real (\langle \varphi, (J-z)\varphi\rangle\omega)\geq d\|\varphi\|^2$, which yields \eqref{2.1}.

Let $D=\f12\diam(H)$ and $c=$ center of $H$ so $H=[c-D, c+D]$. Then
\begin{align*}
a_n &= \int x p_n p_{n-1}\, d\mu \\
&= \int (x-c) p_n p_{n-1}\, d\mu \\
&\leq \sup_{\sigma (J)} \, \abs{x-c} = D
\end{align*}
proving \eqref{2.2}.
\end{proof}

The following is related to the proof of Theorem~4.3.15 in \cite{OPUC1}:

\begin{proposition}\lb{P2.2} Let $d\mu$ be a measure on $\bbR$ of compact support, $p_n(x,d\mu)$
the normalized OPRL, and $H$ the convex hull of the support of $d\mu$. For $z\notin H$, let $d(z) =
\dist(z,H)$. Let $D=\f12 \diam(H)$. Then for such $z$,
\begin{equation} \lb{2.3}
\abs{p_n(z,d\mu)}^2 \geq\biggl( \f{d}{D}\biggr)^2 \biggl( 1+ \biggl( \f{d}{D}\biggr)^2\biggr)^{n-1}
\end{equation}
In particular, $p_n(z)\neq 0$ for all $n$ and
\begin{equation} \lb{2.4x}
\liminf \abs{p_n (z,d\mu)}^{1/n} \geq\biggl( 1+\biggl(\f{d}{D}\biggr)^2\biggr)^{1/2} >1
\end{equation}
\end{proposition}

\begin{remark} Of course, it is well known that $p_n$ has all its zeros on $H$.
\end{remark}

\begin{proof} Let $\varphi_n(x)$ be the function
\begin{equation} \lb{2.4}
\varphi_n(x) =\sum_{j=0}^n p_n(z) p_n(x)
\end{equation}
which has components $\varphi_n=\langle p_0(z), \dots, p_n(z), 0, 0, \dots\rangle$ in $p_n(x)$ basis.
Then, by the recursion relation,
\begin{equation} \lb{2.5}
[(J-z) \varphi_n]_j =\begin{cases}
0 & j\neq n,\, n+1 \\
-a_{n+1} p_{n+1}(z) & j=n \\
a_{n+1} p_n (z) & j=n+1
\end{cases}
\end{equation}
(a version of the CD formula!). Thus,
\begin{equation} \lb{2.6}
\langle\varphi_n, (J-z)\varphi_n\rangle =-a_{n+1} p_{n+1}(z)\, \ol{p_n(z)}
\end{equation}
and \eqref{2.1} becomes
\begin{equation} \lb{2.7}
a_{n+1} \abs{p_{n+1}(z) p_n(z)} \geq d\sum_{j=0}^n \, \abs{p_j(z)}^2
\end{equation}
By \eqref{2.2},
\begin{equation} \lb{2.8}
\abs{p_{n+1}(z) p_n(z)} \geq \f{d}{D}\, \sum_{j=0}^n \, \abs{p_j(z)}^2
\end{equation}

Next use $2\abs{xy}\leq\alpha x^2 + \alpha^{-1} y^2$ for any $\alpha$ to see
\begin{equation} \lb{2.9}
\abs{p_{n+1}(z) p_n(z)} \leq \f12\, \f{d}{D}\, \abs{p_n(z)}^2 + \f12\, \f{D}{d}\,
\abs{p_{n+1}(z)}^2
\end{equation}
which, with \eqref{2.8}, implies
\begin{equation} \lb{2.10}
\abs{p_{n+1}(z)}^2 \geq \biggl( \f{d}{D}\biggr)^2 \sum_{j=0}^n\, \abs{p_j(z)}^2
\end{equation}
This implies that
\begin{equation} \lb{2.11}
\sum_{j=0}^{n+1}\, \abs{p_j(z)}^2 \geq \biggl[ 1+\biggl(\f{d}{D}\biggr)^2\biggr] \sum_{j=0}^n\,
\abs{p_j(z)}^2
\end{equation}
so, since $p_0(z)=1$, we obtain
\begin{equation} \lb{2.12}
\sum_{j=0}^n \, \abs{p_j(z)}^2 \geq \biggl[ 1+\biggl( \f{d}{D}\biggr)^2\biggr]^n
\end{equation}
\eqref{2.12} plus \eqref{2.10} imply \eqref{2.3}, and that implies \eqref{2.4x}.
\end{proof}

\begin{remark} \eqref{2.4x} is also related to Schnol's theorem (see \cite{SL99,Sh} and
\cite[Lemma~4.3.13]{OPUC1}) and to Combes--Thomas estimates \cite{CGes03,Agmon}.
\end{remark}

This yields the key estimate, given the following equality:

\begin{theorem}\lb{T2.3} Let $d\mu$ be a measure of compact support on $\bbR$ with $H$ the
convex hull of $\supp(d\mu)$. Let $n(j)$ be a subsequence {\rm{(}}i.e., $n(1)<n(2) <n(3) < \dots$
in $\{0,1,2,\dots\}${\rm{)}} so that the zero counting measures $d\nu_{n(j)}$ have a weak limit
$d\nu_\infty$ and so that $(a_1\dots a_{n(j)})^{1/n(j)}$ has a nonzero limit $A$. Then, for
any $z\notin H$,
\begin{equation} \lb{2.13}
\lim_{j\to\infty}\, \abs{p_{n(j)}(z)}^{1/n(j)} =A^{-1} \exp(-\Phi_{\nu_\infty}(z))
\end{equation}
where $\Phi_\nu$ is the potential of $\nu$. In particular,
\begin{equation} \lb{2.14}
\exp(-\Phi_{\nu_\infty}(z)) > A
\end{equation}
\end{theorem}

\begin{proof} \eqref{1.4} says that
\begin{equation} \lb{2.15}
\abs{p_{n(j)}(z)}^{1/n(j)}(a_1 \dots a_{n(j)})^{1/n(j)} = \exp (-\Phi_{\nu_{n(j)}}(z))
\end{equation}
For $z\notin H$, $\log\abs{z-y}^{-1}$ is continuous on $H$ so since $\nu_n$ and so $\nu_\infty$
are supported on $H$ (indeed, $\nu_\infty$ is supported on $\sigma_\ess(d\mu)$), $\Phi_{\nu_{n(j)}}(z)
\to \Phi_{\nu_\infty}(z)$ and \eqref{2.15} implies \eqref{2.13}. By \eqref{2.4x},
$\text{LHS of \eqref{2.13}} > 1$, which implies \eqref{2.14}.
\end{proof}

{\it Note}: \eqref{2.14} implies that $\Phi_{\nu_\infty}$ is bounded above which, by arguments of
Craig--Simon \cite{CrS83}, implies $\nu_\infty ((-\infty, E])$ is $\log$-H\"older continuous. This is
a new result, although the fact for ergodic Jacobi matrices is due to Craig--Simon \cite{CrS83}.

\smallskip

This yields an independent proof of Corollary~\ref{C1.2} for OPRL, and more:

\begin{theorem}\lb{T2.4} Under the hypotheses of Theorem~\ref{T2.3}, if $E=\sigma_\ess(d\mu)$,
then $A\leq C(E)$, and if $A=C(E)$, then $d\nu_\infty =d\rho_E$, the equilibrium measure for $E$.
In particular, if $\mu$ is regular {\rm{(}}i.e., $\lim (a_1 \dots a_n)^{1/n} =C(E)${\rm{)}}, then
$d\nu_n \to d\rho_E$ and \eqref{1.20} holds for $z\notin H$.
\end{theorem}

\begin{remark} Thus, we have proven Corollary~\ref{C1.2} again, Theorem~\ref{T1.6}, and one
part of Theorem~\ref{T1.8}.
\end{remark}

\begin{proof} By \eqref{2.14} for $z\notin H$,
\begin{equation} \lb{2.16}
\Phi_{\nu_\infty}(z) \leq \log(A^{-1})
\end{equation}
By lower semicontinuity, this also holds on $H$. Integrating $d\nu_\infty$ using \eqref{A.4}, we obtain
\begin{equation} \lb{2.17}
\calE(\nu_\infty) \leq \log(A^{-1})
\end{equation}
Since $\inf_\nu (\calE(\nu)) =\log (C(E)^{-1})$, we obtain $\log (C(E)^{-1})\leq \log(A^{-1})$, that is,
$A\leq C(E)$. By uniqueness of minimizers, if $A=C(E)$, $d\nu_\infty =d\rho_E$ and regularity implies
$d\rho_E$ is the only limit point. By compactness, $d\nu_n\to d\rho_E$, and then, by \eqref{2.13},
we obtain \eqref{1.20} for $z\notin H$.
\end{proof}

\begin{proof}[Completion of the Proof of Theorem~\ref{T1.8} for OPRL] \ We proved above that (i)
$\Rightarrow$ (iv) and so, by the submean property of $\abs{f(z)}^{1/n}$ (alternatively, by the
subharmonicity of $\log\abs{f(z)}$) for analytic functions, we get (ii) also on $H$ and thus
have (i) $\Rightarrow$ (ii) in full. (ii) $\Rightarrow$ (iii) is trivial.

\smallskip
\noindent \ul{(iii) $\Rightarrow$ (i).} Pick a subsequence $n(j)$ so that $(a_1 \dots a_{n(j)})^{1/n(j)}
\to \liminf (a_1\dots a_{n(j)})^{1/n}=A$,  and so $\nu_{n(j)}\to \nu_\infty$. By \eqref{1.4}
and Theorem~\ref{TA.4}, (iii) implies for q.e.\ $x$ in $E$ we have that $\Phi_{\nu_\infty} (x)
\geq\log(A^{-1})$. Thus, since $d\rho_E$ gives zero weight to zero capacity sets (see
Proposition~\ref{PA.3}) and \eqref{A.2},
\begin{align}
\log(A^{-1}) &\leq\int \log\Phi_{\nu_\infty}(x)\, d\rho_E(x) \notag \\
&= \int \Phi_{\rho_E}(x)\, d\nu_\infty (x) \notag \\
&\leq \log (C(E)^{-1}) \lb{2.18}
\end{align}
by \eqref{A.17}. Thus $A^{-1}\leq C(E)^{-1}$, so $C(E)\leq A$. By Theorem~\ref{T2.4}
(or Corollary~\ref{C1.2}), we see $A=C(E)$, that is, $\mu$ is regular.

The reader may be concerned about this argument if $A=0$. But in that case, Theorem~\ref{TA.4}
and \eqref{1.19} imply that q.e.\ on $E$, $\Phi_\nu(x)=\infty$ which is inconsistent with
Theorem~\ref{TA.7A} (or with Corollary~\ref{CA.2A}). Thus \eqref{1.19} implies $A>0$.

\smallskip
\noindent \ul{(i) $\Rightarrow$ (v).} This is immediate from Theorem~\ref{T2.4}, the fact that
equality holds q.e.\ in \eqref{A.9} and in \eqref{A.17}.

\smallskip
\noindent \ul{(ii) $\Rightarrow$ (vi).} Without loss, we can redefine $Q_n$ so $\|Q_n\|_{L^2
(d\mu)}=1$. Then
\begin{equation} \lb{2.19}
Q_n(z) =\sum_{j=0}^n c_{j,n} p_j(z,d\mu)
\end{equation}
where
\begin{equation} \lb{2.20}
\sum_{j=0}^n \, \abs{c_{j,n}}^2 =1 \Rightarrow \abs{c_{j,n}} \leq 1
\end{equation}
Thus
\begin{equation} \lb{2.21}
\abs{Q_n(z)} \leq n \, \sup_{0\leq j\leq n}\, \abs{p_j(z,d\mu)}
\end{equation}
and \eqref{1.18} $\Rightarrow$ \eqref{1.22}.
\end{proof}

This completes the proof of Theorem~\ref{T1.8} for OPRL and our presentation of the key properties
of regular measures for OPRL. We turn to relations between the support of $d\mu$ and regularity of the
density of zeros that will include Theorem~\ref{T1.7}.

\begin{theorem}\lb{T2.4A} Let $d\mu$ be a measure of compact support, $E$, with Jacobi parameters,
$\{a_n,b_n\}_{n=1}^\infty$. Let $n(j)$ be a subsequence so that $d\nu_{n(j)}$ has a limit,
$d\rho_E$, the equilibrium measure for $E$. Then either
\begin{SL}
\item[{\rm{(a)}}]
\begin{equation}\lb{2.22a}
\lim_{j\to\infty}\, (a_1 \dots a_{n(j)})^{1/n(j)} =C(E)
\end{equation}
or
\item[{\rm{(b)}}] $\mu$ is carried by a set of capacity zero, that is, there is $X\subset E$
of capacity zero so $\mu(\bbR\setminus X)=0$.
\end{SL}
\end{theorem}

\begin{proof} Let $A$ be a limit point of $(a_1 \dots a_{n(j)})^{1/n}$. If $A=0$, interpret $A^{-1}$
as $\infty$. By \eqref{2.15} and the upper envelope theorem (Theorem~\ref{TA.4}), we see for some
subsubsequence $\ti n(j)$,
\begin{equation}\lb{2.22b}
\lim_{j\to\infty}\, \abs{p_{\ti n(j)}(x)}^{1/\ti n(j)} =A^{-1} \exp (-\Phi_{\rho_E}(x))
\end{equation}
for q.e.\ $x$. By Theorem~\ref{TA.7}, $\Phi_{\rho_E}(x)=\log(C(E)^{-1})$ for q.e.\ $x$. So for
q.e.\ $x\in E$,
\begin{equation}\lb{2.22c}
\lim_{j\to\infty}\, \abs{p_{\ti n(j)}(x)}^{1/\ti n(j)} =A^{-1} C(E)
\end{equation}

On the other hand (see \eqref{4.10} below), for $\mu$-a.e.\ $x$, we have
\begin{equation}\lb{2.22d}
\abs{p_n(x)} \leq C(x)(n+1)
\end{equation}
so for such $x$,
\begin{equation}\lb{2.22e}
\limsup \abs{p_n(x)}^{1/n} \leq 1
\end{equation}

If $A<C(E)$, then $\f{C(E)}{A} >1$, so \eqref{2.22e} can only hold on the set of capacity zero where
\eqref{2.22c} fails, that is, either $A=C(E)$ (since it is always true that $A\leq C(E)$) or $\mu$ is
carried by a set of capacity zero.
\end{proof}

Before leaving the subject of OPRL, we want to say something about nonregular situations:

\begin{theorem}\lb{T2.5} Let $\mu$ be a fixed measure of compact support on $\bbR$.
\begin{SL}
\item[{\rm{(a)}}] The set of limit points of $(a_1 \dots a_n)^{1/n}$ is always a closed interval.
\item[{\rm{(b)}}] The set of limits of zero counting measures $d\nu_n$ is always a closed compact set.
\end{SL}
\end{theorem}

\begin{remarks} 1. As quoted in \cite{UWZ}, where the first proof of (a) appeared, (a) is a theorem of
Freud and Ziegler.

\smallskip
2. Part (b) was conjectured in Ullman \cite{Ull89} and is proven in Stahl--Totik \cite{StT} (see
Theorem~2.1.4 of \cite{StT}).

\smallskip
3. Stahl--Totik \cite{StT} also prove (their Theorem~2.2.1) that so long as no carrier of $\mu$ has
capacity zero, the existence of a limit for $d\nu_{n(j)}$ implies the existence of a limit for $(a_1
\dots a_{n(j)})^{1/n(j)}$. However, as we will see (Example~\ref{E2.7}), the converse is false.
\end{remarks}

\begin{proof} We sketch the proof of (a); the proof of (b) can be found in \cite{StT} and is similar
in spirit. The set of limit points is a closed subset of $[0,C(E)]$. If it is not connected,
we can find limit points $A<B$ and $c\in (A,B)$ which is not a limit point.

Thus, there are $N$ and $\veps$ so for $n>N$,
\begin{equation} \lb{2.22}
\Gamma_n\equiv (a_1\dots a_n)^{1/n}\notin (c-\veps, c+\veps)
\end{equation}
Suppose $\Gamma_n <c-\veps$ and let $D=\f12 \diam(\cvh(\supp(\mu)))\geq a_n$ by \eqref{2.2}. Then
\begin{equation} \lb{2.23}
\Gamma_{n+1} =a_{n+1}^{1/n+1} \Gamma_n^{n/n+1} \leq D^{1/n+1} (c-\veps)^{n/n+1}
\end{equation}

Since RHS of \eqref{2.23} converges to $c-\veps$, we can find $N_1$ so
\begin{equation} \lb{2.24}
n\geq N_1 \Rightarrow \text{RHS of \eqref{2.23}} \leq c
\end{equation}
Thus $n\geq N$, $n\geq N_1$, and $\Gamma_n \leq c-\veps$ implies $\Gamma_{n+1}\leq c-\veps$
(by \eqref{2.22}). It follows that $\Gamma_n$ cannot have both $A$ and $B$ as limit points.

This contradiction proves the set of limit points is an interval.
\end{proof}

\begin{example}\lb{E2.7} This example shows that $(a_1\dots a_n)^{1/n}$ may have a limit
(necessarily strictly less than $C(E)$) but $d\nu_n$ does not. A more complicated example
appears as Example~2.2.7 in \cite{StT}. Let $a_n\equiv 1$ (so $(a_1\dots a_n)^{1/n}\to 1$) and
\[
b_n=\begin{cases}
1 & N_{2\ell}\leq n < N_{2\ell+1} \\
-1 & N_{2\ell+1}\leq n < N_{2\ell+2}
\end{cases}
\]
where $N_\ell=2^{3^\ell}$. It is easy to see by looking at traces of powers of the cutoff
Jacobi matrix that $d\nu_{N_{2\ell}^2}\to d\rho_{[-1,3]}$ and $d\nu_{N_{2\ell+1}^2}\to
d\rho_{[-3,1]}$.
\qed
\end{example}

There is another result about the set of limit points that should be mentioned in connection
with work of Ullman and collaborators. Define $c_\mu$ to be $\inf$ of the capacity of Borel
sets, $S$, which are carriers of $\mu$ in the sense that $\mu(\bbR\setminus S) =0$. For example,
if $\mu$ is a dense pure point measure with support $E=[-2,2]$, $\mu$ is supported on a countable
set, so $c_\mu =0$ even though $C(E)=1$. Then, in general, Ullman shows that any limit point of
$(a_1 \dots a_n)^{1/n}$ lies in $[c_\mu, C(\supp(d\mu))]$, and Wyneken \cite{Wyn} proved that
given any $\mu$ and any $[A,B]\subset [c_\mu, C(\supp(d\mu))]$, there is $\eta$ mutually
equivalent to $\mu$ so the set of limit points of $\Gamma_n(\eta)$ is $[A,B]$ (see also
Theorem~\ref{T5.4} below).

In particular, these results show that if $c_\mu =C(\supp(d\mu))$, then $\mu$ is regular---a
theorem of Ullman \cite{Ull}, although Widom \cite{Wid} essentially had the same theorem
(this oversimplifies the relation between Widom \cite{Wid} and Ullman \cite{Ull}; see
\cite[Ch.~4]{StT}). We have not discussed this result in detail because the
Stahl--Totik criterion of Theorem~\ref{T1.11} essentially subsumes these earlier works
(at least for $E$ a finite union of closed intervals)
and we will prove that in Section~\ref{s5}.

\section{Regular Measures for OPUC} \lb{s3}

In this section, we will prove Theorem~\ref{T1.8} for OPUC and an analog of Theorem~\ref{T1.6}.
Here one issue will be that if $E=\partial\bbD$, the zero density may not converge to a measure
on $\partial\bbD$. The key step concerns Proposition~\ref{P2.2}, which essentially depended
on the CD formula which is only known for OPRL and OPUC, and where the OPUC version is not
obviously relevant. Instead, we will see, using operator theoretic methods \cite{Sim-cd},
that there is a kind of ``half CD formula" that suffices. We begin with an analog of
Lemma~\ref{L2.1}:

\begin{lemma}\lb{L3.1}
\begin{SL}
\item[{\rm{(a)}}] Let $\mu$ be a measure of compact support on $\bbC$ and $H$ the convex
hull of the support of $\mu$. Let $M_z$ be multiplication by $z$ on $L^2 (\bbC,d\mu)$.
Then for any $z_0\in\bbC$ and $\varphi \in L^2 (\bbC,d\mu)$, we have
\begin{equation}\lb{3.1}
\abs{\langle \varphi, (M_z-z_0)\varphi\rangle} \geq\dist (z_0, H) \|\varphi\|^2
\end{equation}
\item[{\rm{(b)}}] Let $D$ be defined by
\begin{equation}\lb{3.2}
D=\min_w\, \bigl[\, \max_{z\in H}\, \abs{z-w}\bigr]
\end{equation}
{\rm{(}}which lies between $\f12\diam(H)$ and $\diam(H)${\rm{)}}. Then
\begin{equation}\lb{3.3}
\|X_{n+1}\|_{L^2 (d\mu)} \leq D \|X_n\|_{L^2 (d\mu)}
\end{equation}
\end{SL}
\end{lemma}

\begin{proof} (a) Let $\omega\in\partial\bbD$. Then
\begin{align*}
\abs{\langle\varphi, (M_z-z_0)\varphi\rangle}
& \geq \int \Real ((z-z_0)\bar\omega)\abs{\varphi(z)}^2 \, d\mu (z) \\
&\geq \min_{z\in H}\, \Real ((z-z_0)\bar\omega) \|\varphi\|^2
\end{align*}
Maximizing over $\omega$ yields \eqref{3.1}.

\smallskip
(b) Since $(z-w)X_n$ is a monic polynomial of degree $n+1$,
\[
\|X_{n+1}\| \leq \|(z-w)X_n\| \leq \max_{z\in H}\, \abs{z-w} \|X_n\|
\]
Minimizing over $w$ yields \eqref{3.3}.
\end{proof}

To get the analog of \eqref{2.6}, we need

\begin{proposition}\lb{P3.2} Let $d\mu$ be a measure of compact support on $\bbC$ and let
$M_z$ be multiplication by $z$ on $L^2 (\bbC,d\mu)$. Let $K$ be the orthogonal projection
in $L^2 (\bbC, d\mu)$ onto the $n+1$-dimensional subspace polynomials of degree at most
$n$. Then
\begin{equation}\lb{3.4}
[M_z, K] K = \f{\|X_{n+1}\|}{\|X_n\|}\, [\langle x_n, \dott\rangle x_{n+1}]
\end{equation}
\end{proposition}

\begin{remark} This is essentially ``half" the CD formula; operator theoretic approaches
to the CD formula are discussed in \cite{Sim-cd}.
\end{remark}

\begin{proof} For any $\varphi$,
\begin{equation}\lb{3.5}
[M_z,K] K\varphi = (1-K) z(K\varphi)
\end{equation}
This clearly vanishes if $K\varphi =0$ or if $\varphi\in\ran K_{n-1}$. Thus, it is a rank one
operator. Moreover, since $(1-K)z X_n=X_{n+1}$, we see
\[
[M_z, K] K X_n = X_{n+1}
\]
Since $X_{n+1}=\|X_{n+1}\| x_{n+1}$ and $X_n =\|X_n\| x_n$, we see that \eqref{3.4} holds.
\end{proof}

\begin{proposition}\lb{P3.3} Let $d\mu$ be a measure of compact support on $\bbC$, $x_n(z;d\mu)$
the normalized OPs, and $H$ the convex hull of the support of $d\mu$. For $z_0\notin H$, let
$d(z_0)= \dist(z_0, H)$ and let $D$ be given by \eqref{3.2}. Then for such $z_0$,
\begin{equation}\lb{3.6}
\abs{x_n(z_0; d\mu)}^2 \geq \biggl( \f{d}{D}\biggr)^2
\biggl(1+\biggl( \f{d}{D}\biggr)^2\biggr)^{n-1}
\end{equation}
In particular, $x_n(z_0)\neq 0$ for all $n$ and
\begin{equation}\lb{3.6a}
\liminf \abs{x_n(z_0; d\mu)}^{1/n} \geq \biggl(1+\biggl( \f{d}{D}\biggr)^2\biggr)^{1/2} >1
\end{equation}
\end{proposition}

\begin{remark} Again, it is well known (a theorem of Fej\'er) that zeros of $x_n$ lie in $H$.
\end{remark}

\begin{proof} Define
\begin{equation}\lb{3.7}
\varphi_n (w) = \sum_{j=0}^n \, \ol{x_j(z_0)}\, x_j(w)
\end{equation}
We claim
\begin{equation}\lb{3.8}
\langle \varphi_n, (M_z-z_0)\varphi_n\rangle =-\f{\|X_{n+1}\|}{\|X_n\|}\,
\ol{x_n(z_0)}\, x_{n+1} (z_0)
\end{equation}
This is precisely an analog of \eqref{2.6}. Given this and Lemma~\ref{L3.1}, the proof
is identical to that of Proposition~\ref{P2.2}.

To prove \eqref{3.8}, we note the integral kernel of $K_n$ is
\begin{equation}\lb{3.9}
K_n(s,t) =\sum_{j=0}^n\, p_j(s) \, \ol{p_j(t)}
\end{equation}
and that \eqref{3.4} says
\begin{equation}\lb{3.10}
\int (s-w) K_n(s,w) K_n(w,t)\, d\mu(w) = \f{\|X_{n+1}\|}{\|X_n\|}\,
x_{n+1}(s)\, \ol{x_n(z_0)}
\end{equation}
\eqref{3.10} originally holds for a.e.\ $s,t$ in $\supp(d\mu)$, but since both sides are
polynomials in $s$ and $\bar t$, for all $s,t$. Setting $s=t=z_0$, \eqref{3.10} is just
\eqref{3.8}.
\end{proof}

Now we want to specialize to OPUC. The zeros in that case lie in $\bbD$. One defines
the balayage of the zeros measure, $d\nu_n$, on $\partial\bbD$ by
\begin{equation}\lb{3.11}
\calP(d\nu) = F(\theta)\, \f{d\theta}{2\pi}
\end{equation}
where
\begin{equation}\lb{3.12}
F(\theta) =\int \f{1-\abs{z}^2}{\abs{e^{i\theta}-z}^2}\, d\nu_n(z)
\end{equation}
It is the unique measure on $\partial\bbD$ with
\begin{equation}\lb{3.13}
\int z^k \calP(d\nu_n) = \int z^k \, d\nu_n(z)
\end{equation}
for $k\geq 0$ (see \cite[Prop.~8.2.2]{OPUC1}).

Since $\abs{z} >1\geq\abs{w}$ implies
\begin{equation}\lb{3.14x}
\log\abs{z-w}^{-1} =-\log\abs{z} + \Real\biggl(\, \sum_{j=1}^\infty \,
\f{1}{j}\biggl( \f{w}{z}\biggr)^j\biggr)
\end{equation}
by \eqref{3.13}, we have
\begin{equation}\lb{3.15}
\abs{z} >1 \Rightarrow \Phi_{\nu_n}(z) = \Phi_{\calP(d\nu_n)}(z)
\end{equation}

If $d\nu_n\to d\nu_\infty$, then $\calP(d\nu_n)\to\calP(d\nu_\infty)$, and this equals $d\nu_\infty$
if $d\nu_\infty$ is a measure on $\partial\bbD$. If $\supp(d\mu)\subsetneqq\partial\bbD$, then it is
known that the bulk of the zeros goes to $\partial\bbD$ (Widom's zero theorem; see \cite[Thm.~8.1.8]{OPUC1}),
so $d\nu_\infty$ is a measure on $\partial\bbD$. It is also known (see \cite[Thm.~8.2.7]{OPUC1})
that the zero counting measures for the paraorthogonal polynomials (POPUC) have the same weak limits
as $\calP(d\nu_n)$. The analogs of Theorems~\ref{T2.3} and \ref{T2.4} are thus:

\begin{theorem}\lb{T3.4} Let $d\mu$ be a measure on $\partial\bbD$, the unit circle. Let $n(j)$ be a
subsequence with $n(1) <n(2)< \dots$ so that $(\rho_1 \dots \rho_{n(j)})^{1/n(j)}$ has a nonzero
limit $A$ and so that there is a measure $d\nu_\infty$ on $\partial\bbD$ which is the weak limit
of $\calP(d\nu_{n(j)})$ {\rm{(}}equivalently, of $d\nu_{n(j)}$ if $\supp(d\mu)\neq\partial\bbD$;
equivalently, of the zero counting measures of POPUC{\rm{)}}. Then for any $\abs{z}>1$ or
$z\notin\partial\bbD \setminus\supp(d\mu)$,
\begin{equation}\lb{3.16}
\lim_{j\to\infty}\, \abs{\varphi_{n(j)}(z)}^{1/n(j)} = A^{-1} \exp (-\Phi_{\nu_\infty}(z))
\end{equation}
In particular,
\begin{equation}\lb{3.14}
\exp(-\Phi_{\nu_\infty}(z))\geq A
\end{equation}

It follows that if $E=\sigma_\ess(d\mu)$, then $A\leq C(E)$, and if $\mu$ is regular {\rm{(}}i.e.,
$(\rho_1 \dots \rho_n)^{1/n}\to C(E)${\rm{)}}, then every limit point of $\calP(d\nu_{n(j)})$ is the
equilibrium measure $d\nu_E$. So $\calP(d\nu_n)\to d\nu_E$ {\rm{(}}and if $E\neq\partial\bbD$,
$d\nu_n\to d\nu_E${\rm{)}}.
\end{theorem}

\begin{proof} Given the above discussion and results, this is identical to the proofs of
Theorems~\ref{T2.3} and \ref{T2.4}.
\end{proof}

By mimicking the proof we give for Theorem~\ref{T1.8} for OPRL, we obtain the same result for OPUC.

\section{Van Assche's Proof of Widom's Theorem} \lb{s4}

In this section, we will prove Theorem~\ref{T1.10} using part of Van Assche's approach \cite{vA}.
The basic idea is simple: By a combination of Chebyshev's inequality and the Borel--Cantelli
lemma, if $\|P_{n(j)}\|_{L^2(d\mu)}^{1/n(j)}\to A$, then for $d\mu$-a.e.\ $x$, we have
$\limsup_{j\to\infty} \abs{P_{n(j)}(x)}^{1/n(j)}\leq A$. By using some potential theory,
we will find that the density of zeros measure, $d\nu$, supported on $E$ obeys for q.e.\ $x$,
$\Phi_\nu(x)\geq\log(A^{-1})$ a.e.\ $d\mu$. Since $d\rho_E$ is a.c.\ with respect to $d\mu$,
this will imply $\int \Phi_{\rho_E}(x)\, d\nu(x)\geq\log (A^{-1})$. But by potential theory again,
$\Phi_{\rho_E}(x)\leq\log (C(E)^{-1})$, so we will have $A^{-1}\leq C(E)^{-1}$, that is,
$C(E)\leq A$.

\begin{lemma}\lb{L4.1} Let $d\mu$ be a probability measure on a measure space $X$. Let $f_{n(j)}$ be a
sequence of functions indexed by integers $1 \leq n(1) <n(2) <\dots$. Suppose for some $1\leq p<\infty$,
\begin{equation}\lb{4.1}
\limsup_{j\to\infty}\, \|f_{n(j)}\|_{L^p}^{1/n(j)} =A
\end{equation}
Then for $d\mu$-a.e.\ $x$,
\begin{equation}\lb{4.2}
\limsup \abs{f_{n(j)}(x)}^{1/n(j)}\leq A
\end{equation}
\end{lemma}

\begin{proof} Fix $B >A$. Then
\begin{equation}\lb{4.3}
\mu(S_j(B)) \equiv \mu(\{x\mid\abs{f_{n(j)}(x)} >B^{n(j)}\}) \leq
\f{\|f_{n(j)}\|_{L^p}^p}{B^{n(j)p}}
\end{equation}

By \eqref{4.1} and $B >A$, we see
\[
\sum_j \mu(S_j(B)) <\infty
\]
so for $\mu$-a.e.\ $x$, there is $J(x)$ with $x\notin S_j$ for all $j> J(x)$.
Thus, for $\mu$-a.e.\ $x$,
\[
\limsup \abs{f_{n(j)}(x)}^{1/n(j)}\leq B
\]
Since $B$ is arbitrary, we have \eqref{4.2}.
\end{proof}

\begin{proof}[Proof of Theorem~\ref{T1.10}] Let $A$ be a limit point of $\|P_{n(j)}\|_{L^2(d\mu)}^{1/n(j)}$.
By passing to a subsequence, we can suppose the zero counting measure $d\nu_{n(j)}$ has a limit
$d\nu_\infty$ which, by Proposition~\ref{P1.6A}, is supported on $E$.

By Lemma~\ref{L4.1} for a.e.\ $x(d\mu)$,
\begin{equation}\lb{4.3a}
\limsup \abs{P_{n(j)}(x)}^{1/n(j)} \leq A
\end{equation}

By \eqref{1.4} for such $x$,
\begin{equation}\lb{4.4}
\limsup \exp(-\Phi_{\nu_{n(j)}}(x))\leq A
\end{equation}
By the upper envelope theorem (Theorem~\ref{TA.4}) for q.e.\ $x\in\bbC$,
\begin{equation}\lb{4.5x}
\Phi_{\nu_\infty}(x) = \liminf \Phi_{\nu_{n(j)}}(x)
\end{equation}
Thus, there exist sets $S_1$ and $S_2$ so that $\mu(S_1)=0$ and $C(S_2)=0$, so that for
$x\in\bbC\setminus (S_1\cup S_2)$,
\begin{equation}\lb{4.5}
\Phi_{\nu_\infty}(x)\geq \log(A^{-1})
\end{equation}

We can now repeat the argument that led to \eqref{2.18}. By hypothesis, $d\rho_E$ is $d\mu$-a.c.
So $\rho_E(S_1)=0$ and, of course, since $\calE(\rho_E)<\infty$, $\rho_E(S_2)=0$. Thus,
\eqref{4.5} holds a.e.\ $d\rho_E$.

Therefore, by \eqref{A.2},
\begin{align*}
\log(A^{-1}) &\leq \int \Phi_{\nu_\infty} (x)\, d\rho_E(x) \\
&= \int \Phi_{\rho_E}(x)\, d\nu_\infty(x) \\
&\leq \log (C(E)^{-1})
\end{align*}
by \eqref{A.17}.

Thus, $A^{-1}\leq C(E)^{-1}$ or $C(E)\leq A$. Thus, $\liminf \|P_n\|^{1/n}\geq C(E)$. Since
(see \eqref{1.15}), $\limsup \|P_n\|^{1/n}\leq C(E)$, we have regularity.
\end{proof}

The above proof is basically a part of Van Assche's argument \cite{vA} which can be simplified
since he proves that $d\nu_\infty =d\rho_E$ by a direct argument using similar ideas, and we
can avoid that because of the general argument in Section~2.

This argument can also prove a related result---we will see examples of this phenomenon at
the end of the next section.

\begin{theorem}\lb{T4.2} Suppose $\mu$ is a measure of compact support on $\bbR$ so $E\subset
\supp(d\mu)$ for an essentially perfect compact set $E$ with $C(E) >0$. Suppose $d\rho_E$ is
a.c.\ with respect to $d\mu$, and for some $n(1) < n(2) < \dots$, we have
\begin{equation}\lb{4.7a}
\|P_{n(j)}\|_{L^2 (\bbR,d\mu)}^{1/n(j)} \to C(E)
\end{equation}
for the monic $P_n (x,d\mu)$. Let $d\nu_{n(j)}$ be the corresponding zero counting measure. Then
$d\nu_{n(j)}\overset{w}{\longrightarrow} d\rho_E$.
\end{theorem}

\begin{remarks} 1. We have in mind cases where $E$ is a proper subset of $\supp(d\mu)$. There will
be many subsets with the same capacity, but there can only be one that has $d\rho_E$ a.e.\
with respect to $d\mu$.

\smallskip
2. Since $d\mu\restriction E$ is regular (by Theorem~\ref{T1.10}) and $\|P_{n(j)}
(\dott,\mu)\|_{L^2 (d\mu)} \geq \|P_{n(j)} (\dott, \mu\restriction E)\|_{L^2 (d\mu
\restriction E)}$, we see that
\[
\liminf \|P_n\|_{L^2 (\bbR,d\mu)}^{1/n} \geq C(E)
\]
so \eqref{4.7a} is equivalent to a $\limsup$ assumption.
\end{remarks}

\begin{proof} Let $d\nu_\infty$ be a limit point of $d\nu_{n(j)}$. As in the proof of
Theorem~\ref{T1.10}, there exist sets $S_1$ with $\mu(S_1)=0$ and $S_2$ with $C(S_2)=0$,
so for $x\in\bbC\setminus (S_1\cup S_2)$, we have
\begin{equation}\lb{4.7b}
\Phi_{\nu_\infty}(x) \geq \log (C(E)^{-1})
\end{equation}

Since $\rho_E(S_1)=0$ by the assumption and $\rho_E(S_2)=0$ since $C(E)>0$, \eqref{4.7b}
holds for $\rho_E$-a.e.\ $x$. Moreover, since
\begin{equation} \lb{4.7c}
\Phi_{\rho_E}(z)\leq \log (C(E)^{-1})
\end{equation}
for all $z$,
\begin{align}
\log(C(E)^{-1}) &\leq \int \Phi_{\nu_\infty}(x)\, d\rho_E(x) \notag \\
&= \int \Phi_{\rho_E}(x)\, d\nu_\infty (x) \notag \\
&\leq\log (C(E)^{-1}) \lb{4.10a}
\end{align}

Thus, using \eqref{4.7c}, we see
\[
\Phi_{\rho_E}(x) = \log (C(E)^{-1})
\]
for $\nu_\infty$-a.e.\ $x$. But $\Phi_{\rho_E}(z)<\log(C(E)^{-1})$ for all $z\notin E$, so $\nu_\infty$
is supported on $E$. By \eqref{4.10a} and \eqref{4.7b}, $\Phi_{\nu_\infty}(x)=\log(C(E)^{-1})$ for
$d\rho_E$-a.e.\ $x$. By Theorem~\ref{TA.9}, $\nu_\infty =\rho_E$.
\end{proof}

There is an alternate way to prove \eqref{4.3a} without Lemma~\ref{L4.1} that links it to ideas
more familiar to spectral theorists. It is well known that for elliptic PDEs, there are polynomially
bounded eigenfunctions for a.e.\ energy with respect to spectral measures. This is called the BGK
expansion in \cite{Sxxi} after Berezanski\u\i \cite{Ber}, Browder \cite{Br}, G\aa rding \cite{Gar},
Gel'fand \cite{Gel}, and Kac \cite{Kac}. The translation to OPRL is discussed in Last--Simon \cite{S263}.
Since $\int\abs{p_n(x)}^2\, d\mu=1$, we have
\begin{equation} \lb{4.8}
\sum_{n=0}^\infty \, (n+1)^{-2} \int \abs{p_n(x)}^2\, d\mu <\infty
\end{equation}
and thus, for $d\mu$-a.e. $x$,
\begin{equation} \lb{4.9}
\sum_{n=0}^\infty \, (n+1)^{-2} \abs{p_n(x)}^2 <\infty
\end{equation}
so
\begin{equation} \lb{4.10}
\abs{P_n(x)} \leq C(x) (n+1) \|P_n\|_{L^2}
\end{equation}
which implies \eqref{4.3a}.

It is interesting to note that if $E$ is such that it is regular and $d\rho_E$ is purely
absolutely continuous on $E=\supp(d\rho_E)$, one can use these ideas to provide an alternate
proof (see Simon \cite{Sim-wk} for still another alternate proof in this case). For in that
case, the measure associated to the second kind polynomials, $q_n(x)$, also has a.c.\ weight
$\wti w(x)>0$ for a.e.\ $x$ in $E$, and thus
\begin{equation} \lb{4.11}
\abs{q_n(x)}\leq C(x) (n+1)
\end{equation}
which, by constancy of the Wronskian, implies
\begin{equation} \lb{4.12}
\abs{p_n(x)}^2 + \abs{p_{n+1}(x)}^2 \geq \wti C(x) (n+1)^{-1}
\end{equation}
If $d\nu_{n(j)}\to d\nu_\infty$, so does $d\nu_{n(j)+1}$ (by interlacing of zeros), and thus,
by \eqref{4.10} and \eqref{4.12}, if $\lim (a_1 \dots a_{n(j)})^{1/n(j)} \to A$, then
\begin{equation} \lb{4.13}
-\log (A) + \int \log(\abs{x-y}^{-1})\, d\nu_\infty (y)=0
\end{equation}
for $x\in E$ but with a set of Lebesgue measure zero and of capacity zero removed.
By Theorem~\ref{TA.9}, we conclude that $A =C(E)$ and $d\nu_\infty =d\rho_E$.

\begin{remark} We note that \eqref{4.13} holds a.e.\ on the a.c.\ spectrum and by the above
arguments, a.e.\ on that spectrum, $\f{1}{n} \log \|T_n(x)\|\to 0$, a deterministic
analog of the Pastur--Ishii theorem.
\end{remark}

\section{The Stahl--Totik Criterion} \lb{s5}

In this section, we will present an exposition of Stahl--Totik's proof \cite{StT} of their
result, our Theorem~\ref{T1.11}. As a warmup, we prove

\begin{theorem}\lb{T5.1} Let $d\mu$ be a measure on $\partial\bbD$ obeying
\begin{equation} \lb{5.1}
\inf_{\theta_0}\, \mu(\{e^{i\theta}\mid \abs{\theta-\theta_0} \leq \tfrac{1}{m}\})
\geq C_\delta e^{-\delta m}
\end{equation}
for all $\delta >0$. Then $\mu$ is regular.
\end{theorem}

\begin{proof} We will use Bernstein's inequality that for any polynomial, $P_n$, of degree $n$,
\begin{equation} \lb{5.2}
\sup_{z\in\bbD}\, \abs{P'_n(z)} \leq n\, \sup_{z\in\partial\bbD}\, \abs{P_n(z)}
\end{equation}
Szeg\H{o}'s simple half-page proof of this can be found, for example, in Theorem~2.2.5 of
\cite{OPUC1}.

Applying this to the monic polynomials $\Phi_n(z; d\mu)$, we see that if $\theta_n$ is chosen with
$\abs{\Phi_n(e^{i\theta_n}; d\mu)} =\|\Phi_n\|_{\partial\bbD}$, the sup norm, and $\abs{\theta-\theta_n}
\leq \f{1}{2n}$, then
\begin{equation} \lb{5.3}
\abs{\Phi_n(e^{i\theta}; d\mu)} \geq \tfrac12\, \|\Phi_n\|_{\partial\bbD}
\end{equation}
Thus, by \eqref{5.1} with $m=2n$,
\begin{equation} \lb{5.4}
\|\Phi_n (\dott; d\mu)\|_{L^2(d\mu)}^2 \geq (\tfrac14\, \|\Phi_n\|_{\partial\bbD}^2) C_\delta
e^{-2\delta n}
\end{equation}

Since $\Phi_n (\dott; d\mu)$ is monic,
\begin{equation} \lb{5.5}
\int \Phi_n (e^{i\theta}; d\mu) e^{-in\theta}\, \f{d\theta}{2\pi} =1
\end{equation}
so the $\sup$ norm obeys
\begin{equation} \lb{5.6}
\|\Phi_n\|_{\partial\bbD} \geq 1
\end{equation}
and so \eqref{5.4} implies
\[
\liminf \|\Phi_n (\dott; d\mu)\|_{L^2}^{1/n} \geq e^{-2\delta}
\]
Since $\delta$ is arbitrary, the $\liminf$ is larger than or equal to $1$. Since $C(E)=1$,
$\mu$ is regular.
\end{proof}

There are two issues with just using these ideas to prove Theorem~\ref{T1.11}. While \eqref{5.5}
is special for $\partial\bbD$, its consequence, \eqref{5.6}, is really only an expression
of $\|T_n\|_E \geq C(E)^n$ (see \eqref{B.7}), so it is not an issue.

However, \eqref{5.2} only holds because a circle has no ends. The analog for, say, $[-1,1]$ is
Bernstein's inequality
\begin{equation} \lb{5.7}
\abs{p'(x)} \leq \f{n}{\sqrt{1-x^2}}\, \|p\|_{[-1,1]}
\end{equation}
or (Markov's inequality)
\begin{equation} \lb{5.8}
\abs{p'(x)} \leq n^2 \|p\|_{[-1,1]}
\end{equation}
Either one can be used to obtain a theorem like Theorem~\ref{T5.1} on $[-1,1]$ but $e^{-\delta m}$
needs to be replaced by $e^{-\delta \sqrt{m}}$---interesting, but weaker than Theorem~\ref{T1.11}.

The other difficulty is that \eqref{5.1} is global, requiring a result uniform in $\theta_0$, and
\eqref{1.25} needs only a result for most $\theta_0$. The problem with using bounds on derivatives
is that they only get information on a single set of size $O(\f{1}{n})$ at best. They get
$\abs{p_n(x)}\geq \f12 \|p_n\|_E$ there, but that is overkill---we only need $\abs{p_n(x)} \geq
e^{-\delta' n} \|p_n\|_E$, and that actually holds on a set of size $O(1)$! The key will thus be
a variant of the Remez inequality in the following form:

\begin{proposition}\lb{P5.2} Fix $E$ a finite union of closed bounded intervals in $\bbR$. Then
there is $c(\delta) >0$ with $c(\delta)\to 0$ as $\delta\downarrow 0$, so that for any $F\subset E$
with $\abs{E\setminus F}<\delta$, we have
\begin{equation} \lb{5.9}
\|Q_n\|_E\leq e^{c(\delta) n} \|Q_n\|_F
\end{equation}
for any polynomial, $Q_n$, of degree $n$.
\end{proposition}

\begin{remarks} 1. This is a variant of an inequality of Remez \cite{Rem}; see the proof for his precise
result.

\smallskip
2. The relevance of Remez's inequality to regularity appeared already in Erd\"os--Tur\'an \cite{ET}
and was the key to the proof in Freud \cite{FrB} of the Erd\"os--Tur\'an theorem, Theorem~\ref{T1.9}.
Its use here is due to Stahl--Totik \cite{StT}.
\end{remarks}

\begin{proof} If $E= I_1\cup\cdots\cup I_\ell$ disjoint intervals and $\abs{E\setminus F}\leq\delta$,
then $\abs{I_j\setminus I_j\cap F}\leq\delta$ for all $j$, so it suffices to prove this result for
each single interval and then, by scaling, for $E=[-1,1]$.

In that case, Remez's inequality (due to Remez \cite{Rem}; see Borwein--Erd\'elyi \cite{BE} for a
proof and further discussion) says that if $F\subset [-1,1]$ and $\abs{[-1,1]\setminus F}\leq\delta$,
then with $T_n$ the classical first kind Chebyshev polynomials,
\begin{equation} \lb{5.10}
\|Q_n\|_E \leq T_n \biggl( \f{2+\delta}{2-\delta}\biggr) \|Q_n\|_F
\end{equation}
(This can be proven by showing the worst case occurs when $F=[-1,1-\delta]$ and $Q_n(x) =T_n
(\f{2x+\delta}{2-\delta})$.)

Since
\begin{equation} \lb{5.11}
T_n (\cosh (x)) = \cosh (nx) \leq e^{nx}
\end{equation}
and $\cosh (\veps) =1 + \f{\veps^2}{2} + O(\veps^4)$, we have
\begin{equation} \lb{5.12}
T_n \biggl( \f{2+\delta}{2-\delta}\biggr) \leq \exp \bigl( n [\sqrt{2\delta} + O(\delta^{3/2})]\bigr)
\end{equation}
so for $E=[-1,1]$, \eqref{5.9} holds with $c(\delta) =\sqrt{2\delta} + O(\delta^{3/2})$.
\end{proof}

\begin{lemma}\lb{L5.3} If $P_n$ is a real polynomial of degree $n$, and $a>0$, $S\equiv\{\lambda\in
\bbR\mid\abs{P_n(x)} > a\}$ is a union of most $(n+1)$ intervals.
\end{lemma}

\begin{proof} $\partial S$ is the finite set of points where $P_n(x) =\pm a$. If all the zeros of
$P_n\pm a$ are simple, these boundary points are distinct. Including $\pm\infty$ so each interval has
two ``endpoints," these intervals have at most $2n+2$ distinct endpoints (and exactly that number if
all roots of $P_n\pm a$ are real). If some root of $P_n\pm a$ is double, two intervals can share
an endpoint but that endpoint counts twice in the zeros.
\end{proof}

\begin{proof}[Proof of Theorem~\ref{T1.11}] By Proposition~\ref{PB.3}, if $P_n(x)=P_n(x; d\mu)$, then
\begin{equation} \lb{5.14}
\sup_{x\in E}\, \abs{P_n(x)} \geq c(E)^n
\end{equation}
Fix $\delta_1$ and let
\begin{equation} \lb{5.15}
F=\{x\mid\abs{P_n(x)} \leq c(E)^n e^{-2c(\delta_1) n}\}
\end{equation}
If $\abs{E\setminus F} <\delta_1$, then \eqref{5.9} would imply $\|P_n\|_E\leq c(E)^n e^{-c(\delta_1)n}$,
violating \eqref{5.14}. So
\begin{equation} \lb{5.15x}
\abs{E\setminus F} \geq\delta_1
\end{equation}

By Lemma~\ref{L5.3}, $\bbR\setminus F$ is a union of at most $n+1$ intervals, so if $E$ is a union of
$\ell$ intervals, $E\setminus F$ consists of at most $\ell (n+1)$ intervals (a very crude overestimate
that suffices for us!).

Some of these intervals may have size less than $\f{\delta_1}{4n\ell}$, but the total size of those
is at most $\f{\delta_1}{2}$, so we can find disjoint intervals $I_1^{(n)}, \dots, I_{k(n)}^{(n)}$ in
$E\setminus F$, so
\begin{equation} \lb{5.16}
\abs{I_j^{(n)}}\geq \f{\delta_1}{4 n\ell} \qquad \biggl| \bigcup_{j=1}^{k(n)} I_j^{(n)}\biggr| \geq
\f{\delta_1}{2}
\end{equation}

Let $\ti I_j^{(n)}$ be the interval of size $\f12 \abs{I_j^{(n)}}$ and the same center. Then with
$L^{(n)}(\delta_1)=\cup_{j=1}^{k(n)} \ti I_j^{(n)}$, we have
\begin{gather}
\abs{L^{(n)} (\delta_1)} \geq \f{\delta_1}{4} \lb{5.17} \\
\abs{P_n(y)}\geq c(E)^n e^{-2c(\delta_1)n} \quad\text{if}\quad \dist (y, L^{(n)}(\delta_1)) \leq
\f{\delta_1}{16 n\ell} \lb{5.18}
\end{gather}

Now define for any $\delta_2 >0$ and $m$,
\[
J(m,\delta_2) =\{x\mid \mu (x-\tfrac{1}{m}, x+\tfrac{1}{m})\geq e^{-\delta_2 m}\}
\]
By hypothesis, for any fixed $\delta_2$,
\[
\lim_{m\to\infty}\, \abs{E\setminus J(m,\delta_2)} =0
\]
and, in particular, for any fixed integer $M$\!, for all large $n$,
\begin{equation} \lb{5.19}
\abs{E\setminus J (Mn, \delta_2)} < \f{\delta_1}{4}
\end{equation}
so, in particular,
\begin{equation} \lb{5.20}
J(Mn,\delta_2)\cap L^{(n)}(\delta_1) \neq \emptyset
\end{equation}

Given $\delta_1$, pick $M$ so large that
\begin{equation} \lb{5.21}
M^{-1} \leq \f{\delta_1}{16\ell}
\end{equation}
If $x$ lies in the set on the left side of \eqref{5.20}, let $I=\{y\mid \abs{x-y} \leq \f{1}{Mn}\}$.
Then since $\f{1}{Mn} \leq \f{\delta_1}{16n\ell}$, for $y\in I$,
\begin{equation} \lb{5.22}
\abs{P_n(y)} \geq c(E)^n e^{-2c(\delta_1)n}
\end{equation}
since $x\in L^{(n)}(\delta_1)$ and \eqref{5.18} holds. By $x\in J(Mn,\delta_2)$,
\begin{equation} \lb{5.23}
\mu(I) \geq e^{-M\delta_2 n}
\end{equation}

Thus,
\begin{equation} \lb{5.24}
\|P_n\|_{L^2} \geq c(E)^n e^{-2c(\delta_1)n} e^{-MN\delta_2/2}
\end{equation}
so
\begin{equation} \lb{5.25}
\liminf \|P_n\|_{L^2}^{1/n} \geq c(E) e^{-2c(\delta_1)} e^{-M\delta_2/2}
\end{equation}

First pick $\delta_1$, then fix $M$ by \eqref{5.21} (recall $\ell$ is fixed as the number of intervals
in $E$) and let $\delta_2=\f{\delta_1}{M}$. Then take $\delta_1\downarrow 0$ and get
$\liminf \|P_n\|_{L^2}^{1/n} \geq c(E)$, proving regularity.
\end{proof}

Here is a typical application of the Stahl--Totik criterion. It illustrates the limitations of
regularity criteria like those of \cite{Ull,Wid} that only depend on what sets are carriers for
$\mu$. This result is a special case of a theorem of Wyneken \cite{Wyn}.

\begin{theorem}\lb{T5.4} Let $\mu$ be a measure whose support is $E$, a finite union of closed
intervals. Then there exists a measure $\eta$ equivalent to $\mu$ which is regular.
\end{theorem}

\begin{proof} For any $n$, define
\begin{equation} \lb{5.26}
\mu_n =\sum_{\{j\mid \mu((\f{j}{n}, \f{j+1}{n}]) >0\}}\,
\mu \biggl(\biggl( \f{j}{n}\, , \f{j+1}{n}\biggr]\biggr)^{-1}
\mu \restriction \biggl( \f{j}{n}\, , \f{j+1}{n}\biggr]
\end{equation}
Then $\mu_n$ has total mass at most $n(\abs{E}+\ell)$ where $\ell$ is the number of intervals. Let
\begin{equation} \lb{5.27}
\eta =\sum_{n=1}^\infty n^{-3} \mu_n
\end{equation}
which is easily seen to be equivalent to $\mu$.

Notice that if $\dist(x, \bbR\setminus E) >\f{1}{n}$, then $[x-\f{1}{n}, x+\f{1}{n}]$ contains an
interval of the form $(\f{j}{2n}, \f{j+1}{2n}]$, so $\mu_{2n} ([x-\f{1}{n}, x+\f{1}{n}])\geq 1$.
Thus
\begin{equation} \lb{5.28}
\bigl| \bigl\{x\bigm| \eta ([ x-\tfrac{1}{n}\, , x+\tfrac{1}{n}]) \leq
8n^{-3}\bigr\}\bigr| \leq \tfrac{2\ell}{n}
\end{equation}
and \eqref{1.25} holds.
\end{proof}

By using point measures, it is easy to construct nonregular measures, including ones that illustrate
how close \eqref{1.26} is to being ideal. The key is

\begin{theorem}\lb{T5.5} Let $\{x_j\}_{j=1}^\infty$ be a bounded sequence in $\bbR$ and $\{a_j\}_{j=1}^\infty$
an $\ell^1$ sequence of positive numbers. Let
\begin{equation} \lb{5.29}
\mu=\sum_{j=1}^\infty a_j \delta_{x_j}
\end{equation}
Let
\begin{equation} \lb{5.30}
d=\max_{j,k}\, \abs{x_j-x_k}
\end{equation}
Then
\begin{equation} \lb{5.31}
\|P_n(x,d\mu)\|_{L^2 (\bbR,d\mu)} \leq d^n \biggl( \, \sum_{j=n+1}^\infty a_j\biggr)^{1/2}
\end{equation}
\end{theorem}

\begin{proof} Let $Q_n(x) =\prod_{j=1}^n (x-x_j)$, which kills the contributions of the pure points at
$\{x_j\}_{j=1}^n$, so
\begin{equation} \lb{5.32}
\|Q_n\|^2 \leq\sum_{j=n+1}^\infty d^{2n} a_j
\end{equation}
by \eqref{5.30}. Since $\|P_n\|\leq\|Q_n\|$, \eqref{5.31} is immediate.
\end{proof}

\begin{corollary}\lb{C5.6} Let $\{x_j\}_{j=1}^\infty$ be an arbitrary bounded subset of $\bbR$. Then
there exists a pure point measure $d\mu$ with precisely this set as its set of pure points, so that
$\|P_n\|^{1/n}\to 0$. In particular, if $E$ is any compact set with $C(E)>0$, there is a measure $\mu$
with $\supp(\mu)=E$ and $\mu$ not regular.
\end{corollary}

\begin{proof} Pick $a_j=e^{-j^2}$ so $(\sum_{j=n+1}^\infty a_j)^{1/2n}\to 0$.
\end{proof}

\begin{example}\lb{E5.7} The following illuminates \eqref{1.25}. For $2^n\leq k<2^{n+1}$, let
$x_k =\f{k-2^n}{2^n}$ and let $0<y<1$. Define
\begin{equation} \lb{5.33}
d\mu =\sum_{k=1}^\infty y^k \delta_{x_k}
\end{equation}
The $x_k$ are not distinct, but that does not change the bound \eqref{5.31}. Thus
\begin{equation} \lb{5.34}
\limsup \|P_n\|^{1/n} \leq y
\end{equation}
Since $C([0,1])=\f14$, the measure is not regular if $y<\f14$. On the other hand, if $2^n\leq m\leq 2^{n+1}$
and $x_0\in [0,1]$, there is an $x_k$ with $\abs{x_k-x_0}\leq \f{1}{m}$ and $2^n\leq k <2^{n+1}$. Thus
\[
\mu ([ x_0-\tfrac{1}{m}\, , x_0 + \tfrac{1}{m}]) \geq y^k \geq y^{2^{n+1}} \geq y^{2m}
\]
so \eqref{1.25} holds for $\eta= -\log y^2$, that is, for some but not all $\eta$. This shows the
exponential rate in Theorem~\ref{T1.11} cannot be improved.
\qed
\end{example}

\begin{example}\lb{E5.8} We will give an example of a measure $d\mu$ on $[-2,2]$ which is a.c. on
$[-2,0]$ and so that among the limit points of the zero counting measures, $d\nu_n$ are both
$d\rho_{[-2,2]}$ and $d\rho_{[-2,0]}$, the equilibrium measure for $[-2,2]$ and for $[-2,0]$.
This will answer a question asked me by Yoram Last, in reaction to Remling \cite{Remppt},
whether a.c.\ spectra force the existence of a density  of states and also show that bounds
on limit points of $d\nu_n$ of Totik--Ullman \cite{TU} and Simon \cite{Sim-wk} cannot be
improved.

We define $d\mu$ by
\begin{equation} \lb{5.35}
d\mu\restriction [-2,0]=(-x(x+2))^{-1/2}\, dx
\end{equation}
picked so the OPRL for the restriction are multiples of the Chebyshev polynomials for $[-2,0]$
\begin{equation} \lb{5.36}
d\mu\restriction [0,2]=\sum_{n=1}^\infty a_n\, d\eta_n
\end{equation}
where $d\eta_n$ is concentrated uniformly at the dyadic rationals of the form $k/2^n$ not previously
``captured," that is,
\begin{equation} \lb{5.37}
d\eta_n =\sum_{j=0}^{2^n-1} \f{1}{2^n} \, \delta_{(2j+1)/2^{n-1}}
\end{equation}
The $a_n$'s are carefully picked as follows. Define $N_j$ inductively by
\begin{equation} \lb{5.38}
N_1 =1 \qquad N_{j+1} =2^{N_j^3}
\end{equation}
and
\[
a_n = \begin{cases}
\f{1}{n^2} & N_{2k-1} < N \leq N_{2k} \\
2^{-n^4} & N_{2k} < n \leq N_{2k+1}
\end{cases}
\]

Our goal will be to prove that
\begin{equation} \lb{5.39}
d\nu_{2^{N_{2k}^2}} \to d\rho_{[-2,0]} \qquad
d\nu_{2^{N_{2k+1}^2}} \to d\rho_{[-2,2]}
\end{equation}
Intuitively, for $m=2^{N_{2k+1}^2}$, the measures at level $1/m$ will be uniformly spaced out (on
an exponential scale), so by the Stahl--Totik theorem, the zeros will want to look like the equilibrium
measure for $[-2,2]$. But for $m=2^{N_{2k}^2}$, most intervals of size $1/m$ in $[0,2]$ will have tiny
measure, so the zeros will want to almost all lie on $[-2,0]$, where the best strategy for these
(to minimize $\int P_m^2\, d\mu$) will be to approximate the equilibrium measure for $[-2,0]$.

As a preliminary, we will show
\begin{alignat}{2}
&\limsup \|P_n\|^{1/n} =1 \qquad &&  \liminf \|P_n\|^{1/n} =\tfrac12 \lb{5.40} \\
&\lim_{k\to\infty} \|P_{2^{N_{2k}^2}}\|^{1/2^{N_{2k}^2}} = \tfrac12 \qquad  &&
\lim_{k\to\infty} \|P_{2^{N_{2k+1}^2}}\|_{1/2^{N_{2k+1}^2}}  =1 \lb{5.41}
\end{alignat}

We begin with
\begin{equation} \lb{5.42}
\limsup \|P_n\|^{1/n} \leq 1 \qquad
\liminf \|P_n\|^{1/n} \geq \tfrac12
\end{equation}
The first is immediate from \eqref{1.15} and $C([-2,2])=1$; the second from $\|P_n(x,d\mu)\|\geq
\|P_n(x,d\mu\restriction [-2,0])\|$ (by \eqref{1.12}), regularity of $d\mu\restriction [-2,0]$, and
$C([-2,0])=\f12$.

Next, we turn to
\begin{equation} \lb{5.43}
\limsup_{k\to\infty}\, \|P_{2^{N_{2k}^2}}\|^{1/2^{N_{2k}^2}} \leq \tfrac12
\end{equation}
Let $T_n(x; [-2,0])$ be the Chebyshev polynomials for $[-2,0]$ (which are just affinely related to the
classic Chebyshev polynomials of the first kind) and let
\begin{equation} \lb{5.44}
Q_{2^{N_{2k}^2}} (x) = T_{2^{N_{2k}^2 - N_{2k}}} (x;[-2,0]) \prod_{\ell=1}^{2^{N_{2k}}}
\biggl( x - \f{\ell}{2^{N_{2k}-1}}\biggr)
\end{equation}
so by \eqref{1.12},
\begin{equation}\lb{5.45}
\|P_{2^{N_{2k}^2}}\| \leq \|Q_{2^{N_{2k}^2}}\|\leq \boxed{1} + \boxed{2}
\end{equation}
where $\boxed{1}$ is the contribution of the integral from $[-2,0]$ and $\boxed{2}$ from $(0,2)$.

Since $\cos\ell x=2^{\ell-1}(\cos x)^\ell +$ lower order and the average of $\cos^2 x$ is $\f12$,
for any $[a,b]$,
\begin{equation} \lb{5.46}
\|T_m(x;[a,b])\| = \sqrt{2}\, C([a,b])^m
\end{equation}
where the norm is over $L^2 (\bbR,d\rho_{[a,b]})$. Since the product in \eqref{5.44} is bounded by
$4^{2^{N_{2k}}}$ on $[-2,2]$, we have
\begin{equation} \lb{5.47}
\boxed{1} \leq \sqrt{2} \, (\tfrac12)^{2^{N_{2k}^2 - N_{2k}}} 4^{2^N_{2k}}
\end{equation}

On the other hand, there is a constant $K$ so
\begin{equation} \lb{5.48}
\|T_m(x;[-2,0])\|_{L^\infty ([-2,2])} \leq K^m
\end{equation}
and the product in \eqref{5.44} kills all the pure points up to level $N_{2k}$:
\begin{align*}
\boxed{2} &\leq K^{2^{N_{2k}}} \sum_{n=N_{2k}}^\infty a_n \\
&\leq K^{2^{N_{2k}}} [N_{2k+1} 2^{-N_{2k}^4} + (N_{2k+1})^{-1}]
\end{align*}
is much smaller than the right side of \eqref{5.47} for $k$ large. Thus, by \eqref{5.47},
\begin{equation} \lb{5.59}
\limsup \left(\,\boxed{1} + \boxed{2}\,\right)^{1/2^{N_{2k}^2}} \leq \tfrac12
\end{equation}
proving \eqref{5.43}.

In verifying \eqref{5.40}, we finally prove that
\begin{equation} \lb{5.60}
\liminf_{k\to\infty}\, \|P_{2^{N_{2k+1}^2}}\|^{1/2^{N_{2k+1}^2}} \geq 1
\end{equation}
By the fact that the Chebyshev polynomials for $[-2,2]$ obey
\begin{equation} \lb{5.61}
T_n (2\cos x; [-2,2]) =2\cos nx
\end{equation}
has $\|T_n\|_{L^\infty ([-2,2])} =2$, we see
\begin{equation} \lb{5.62}
\|P_m\|_{L^\infty ([-2,2])} \geq 2
\end{equation}
By Markov's inequality \eqref{5.8}, we have
\begin{equation} \lb{5.63}
\|P'_m\|_{L^\infty ([-2,2])} \leq \tfrac{m^2}{2} \, \|P_m\|_{L^\infty ([-2,2])}
\end{equation}
so there is an interval of size $4/m^2$ where $P_m(x)\geq 1$, that is,
\begin{equation} \lb{5.64}
\|P_m(x)\|_{L^2(d\mu)}^2 \geq \inf_{y\in [-2,2]} \mu([y-\tfrac{2}{m^2}\, ,
y+\tfrac{2}{m^2}])
\end{equation}
which implies that
\begin{equation} \lb{5.65}
\|P_{2^{N_{2k+1}^2}}\|_{L^2(\mu)} \geq a_{2^{N_{2k+1}^2}} 2^{-N_{2k+1}^2}
\end{equation}
so it is bounded from below by a power of $2^{-N_{2k+1}^2}$. Since $m^{-\ell/m}\to 1$
for any fixed $\ell$, we obtain \eqref{5.60}.

Clearly, \eqref{5.42}, \eqref{5.43}, and \eqref{5.60} imply \eqref{5.40} and \eqref{5.41}.
We now only need to go from there to results on limits of $d\nu_n$. By Theorem~\ref{T2.4},
the second equality in \eqref{5.41} implies the second limit result in \eqref{5.39}.
By Theorem~\ref{T4.2}, the first equality in \eqref{5.41} implies the first limit
result in \eqref{5.39}.
\qed
\end{example}

\begin{example}\lb{E5.9} Here is an example of a measure $d\mu$ on $[0,1]$ where the density
of zeros has a limit singular relative to the equilibrium measure for $[0,1]$. Such examples
are discussed in \cite{StT} and go back to work of Ullman. Let $\Sigma$ be the classical
Cantor set and $d\rho_\Sigma$ its equilibrium measure. Let
\begin{equation} \lb{5.70}
d\mu = d\rho_\Sigma + \sum_{n=1}^\infty 2^{-n^4} \biggl(\, \sum_{j=0}^{2^{n-1}} \f{1}{2^n}\,
\delta_{(2j+1)/2^n}\biggr)
\end{equation}
As in the above construction, one shows $\|P_n\|^{1/n}\to C(\Sigma)$ and then Theorem~\ref{T4.2}
implies that $d\nu_n\to d\rho_\Sigma$ which is singular with respect to Lebesgue measure, and
so relative to $d\rho_{[0,1]}\equiv d\rho_{\supp(d\mu)}$.
\qed
\end{example}

\section{Structural Results} \lb{s6}

In this section, we will focus on the mutual regularity of related measures. There are three main
theorems, all from Stahl--Totik \cite{StT}:

\begin{theorem}\lb{T6.1} Let $\mu,\eta$ be two measures of compact support whose supports are equal up
to sets of capacity zero. If $\mu\geq\eta$ and $\eta$ is regular, then so is $\mu$.
\end{theorem}

\begin{theorem} \lb{T6.2} Let $\{E_n\}_{n=1}^\infty$ and $E_\infty$ be compact subsets of $\bbC$ so
that $E_\infty$ and $\cup_{n=1}^\infty E_n$ agree up to sets of capacity zero and $C(E_\infty) >0$.
Let $\mu$ be a measure with $\supp(d\mu)=E_\infty$ so that each $\mu\restriction E_j$ which is nonzero
is regular. Then $\mu$ is regular.
\end{theorem}

\begin{remark} By $\mu\restriction K$, we mean the measure
\begin{equation} \lb{6.1}
(\mu\restriction K)(S) =\mu(K\cap S)
\end{equation}
\end{remark}

To understand why the next theorem is so restrictive compared to Theorem~\ref{T6.2}, consider

\begin{example}\lb{E6.3} Let $E$ be the standard Cantor set in $[0,1]$. Let $\eta$ be a measure
on $E$ which is not regular (see Corollary~\ref{C5.6}) and let
\begin{equation} \lb{6.2}
d\mu = d\eta + dx\restriction [0,1]
\end{equation}
By Theorem~\ref{T6.1}, $d\mu$ is regular. But $d\mu\restriction E =d\eta$ is not regular.
\qed
\end{example}

\begin{theorem}\lb{T6.4} Let $I=[a,b]$ be a closed interval with $I\subset E\subset\bbR$ and $E$
compact. Let $\mu$ be a regular measure with support in $E$ so $C(\supp(\mu\restriction I)) >0$.
Then $\mu\restriction I$ is regular.
\end{theorem}

\begin{remarks} 1. We do not require that $\supp(d\mu)=E$ (nor that $I\subset\supp(d\mu)$) but only
that $\supp(d\mu)\subset E$ and that $\mu$ is regular in the sense that $C(\supp(d\mu)) >0$ and
$\|P_n (\dott, d\mu)\|_{L^2(d\mu)}^{1/n}\to C(\supp(d\mu))$.

\smallskip
2. The analog of the sets $I$ in \cite{StT} must have nonempty two-dimensional interior. Our
$I$ obviously has empty two-dimensional interior, but if $I=[a,b]\subset E\subset\bbR$ and
if $D$ is the disk $\{z\mid\abs{z-\f12(a+b)} \leq \f12\abs{b-a}\}$, then $\mu\restriction D
=\mu\restriction I$.
\end{remarks}

The proofs of Theorems~\ref{T6.1} and \ref{T6.2} will be easy, but Theorem~\ref{T6.4} will be
nontrivial. Here are some consequences of these results:

\begin{corollary}\lb{C6.5} Let $\mu,\nu$ be two regular measures {\rm{(}}with different supports
allowed{\rm{)}}. Then their max, $\mu\vee\nu$, and sum, $\mu +\nu$, are regular.
\end{corollary}

\begin{remark} See Doob \cite{Doob} for the definition of $\mu\vee\nu$.
\end{remark}

\begin{proof} $\mu+\nu$ and $\mu\vee\nu$ have the same support and $\mu+\nu\geq\mu\vee\nu$ so,
by Theorem~\ref{T6.1}, we only need the result for $\mu\vee\nu$. Let $E_1 =\supp(\mu)$ and
$E_2=\supp(\nu)$, so $\supp(\mu\vee\nu)= E_1 \cup E_2$. By definition, $(\mu\vee\nu)\restriction
E_1 \geq\mu$ and they have the same supports. So, by Theorem~\ref{T6.1}, $\mu\vee\nu\restriction
E_1$ is regular. Similarly, $\mu\vee\nu\restriction E_2$ is regular. By Theorem~\ref{T6.2},
$\mu\vee\nu$ is regular.
\end{proof}

\begin{corollary}\lb{C6.6} Let $E=I_1\cup\dots\cup I_\ell$ be a union of finitely many disjoint
closed intervals. Let $\mu$ be a measure on $E$. Then $\mu$ is regular if and only if each
$\mu\restriction I_j$ is regular.
\end{corollary}

\begin{proof} Immediate from Theorems~\ref{T6.2} and \ref{T6.4}.
\end{proof}

\begin{proof}[Proof of Theorem~\ref{T6.1}] Since \eqref{1.12} holds,
\begin{equation} \lb{6.3}
\|X_n\|_{L^2(d\eta)} \leq \|X_n\|_{L^2 (d\mu)}
\end{equation}
Given \eqref{1.15}, we have (with $E=\supp(d\mu)$)
\[
\lim \|X_n\|_{L^2(d\eta)}^{1/n} =C(E)\Rightarrow \lim \|X_n\|_{L^2 (d\mu)}^{1/n} =C(E)
\qedhere
\]
\end{proof}

\begin{proof}[Proof of Theorem~\ref{T6.2}] Let $\mu_j=\mu\restriction E_j$ and let $x_n(z)$ be the
$x_n$'s for $d\mu$. Then
\begin{equation} \lb{6.4}
\|x_n\|_{L^2 (d\mu_j)} \leq \|x_n\|_{L^2 (d\mu)} =1
\end{equation}
so
\begin{equation} \lb{6.5}
\abs{x_n(z)} \leq \f{\abs{x_n(z)}}{\|x_n\|_{L^2 (d\mu_j)}}
\end{equation}
By regularity and Theorem~\ref{T1.8}(vi), for q.e.\ $z\in E_j$ (using $G_{E_j}(z)=0$ for q.e.\
$z\in E_j$ by Theorem~\ref{TA.7}(b) and \eqref{A.26a}), we see for q.e.\ $x\in E_j$,
\begin{equation} \lb{6.6}
\limsup_{n\to\infty}\, \abs{x_n(z)}^{1/n} \leq 1
\end{equation}
Since $\cup_{j=1}^\infty E_j$ is q.e.\ $E$, we have \eqref{6.6} q.e.\ on all of $E$. By (iii)
$\Rightarrow$ (i) in Theorem~\ref{T1.8}, $\mu$ is regular.
\end{proof}

To prove Theorem~\ref{T6.4}, we first make a reduction:

\begin{proposition}\lb{P6.7} Suppose there is $I=[a,b]\subset E\subset\bbR$ {\rm{(}}with
$a<b${\rm{)}}, $\mu$ regular, $C(\supp(\mu\restriction I)) >0$ but
\begin{equation} \lb{6.7}
\liminf \|P_n (\dott,\mu\restriction I)\|_{L^2 (\mu\restriction I)}^{1/n} <
C(\supp(\mu\restriction I))
\end{equation}
Then there exists a $\mu$, perhaps distinct but also regular and supported on $E$, so that
\eqref{6.7} holds and
\begin{equation} \lb{6.8}
\liminf \|P_n (\dott, \mu\restriction I)\|_{L^2 (\mu\restriction I)}^{1/n} >0
\end{equation}
\end{proposition}

\begin{proof} If \eqref{6.8} holds for the initial $\mu$, we can stop. Otherwise, we will take
\begin{equation} \lb{6.9}
\ti\mu =\mu+\rho_F
\end{equation}
where $F=[x_0-\delta, x_0+\delta]\subset I$ with $x_0=\f12 (a+b)$ and $\delta$ sufficiently small
chosen later.

By Corollary~\ref{C6.5}, $\ti\mu$ is regular and, by \eqref{6.3},
\[
\|P_n(\dott,\ti\mu\restriction I)\|_{L^2 (\ti\mu\restriction I)}^{1/n} \geq
\|P_n(\dott, \rho_F)\|_{L^2 (\rho_F)}^{1/n}
\]
so \eqref{6.8} holds since
\[
\lim_{n\to\infty}\, \|P_n (\dott, \rho_F)\|_{L^2 (\rho_F)}^{1/n} =C(F) >0
\]

Thus, we need only prove that \eqref{6.7} holds for suitable $\delta$. Since we are supposing \eqref{6.8}
fails for $\mu$, pick $n(j)\to\infty$ so
\begin{equation} \lb{6.10}
\|P_{n(j)} (\dott, \mu\restriction I)\|_{L^2 (\mu\restriction I)}^{1/n(j)}\to 0
\end{equation}
Define
\begin{equation} \lb{6.11}
Q_{2n(j)}(x)=P_{n(j)} (x,\mu\restriction I) (x-x_0)^{n(j)}
\end{equation}

Let $d=\diam(E)$ and note that on $E$, since $x_0\in I\subset E$,
\begin{equation} \lb{6.12}
\abs{x-x_0}^{n(j)} \leq d^{n(j)}
\end{equation}
and since $P_{n(j)}$ has all its zeros in $\cvh(E)$, for $x\in E$,
\begin{equation} \lb{6.13}
\abs{P_{n(j)}(x)} \leq d^{n(j)}
\end{equation}

Thus,
\begin{align}
\|Q_{2n(j)}\|_{L^2(\ti\mu\restriction I)}^2
&= \|P_{n(j)}(\cdot -x_0)^{n(j)}\|_{L^2 (\mu\restriction I)}^2 +
\|P_{n(j)}(\cdot -x_0)^{n(j)}\|_{L^2 (d\rho_E)}^2 \notag \\
&\leq d^{2n(j)} \|P_{n(j)}\|_{L^2 (\mu\restriction I)}^2 + d^{2n(j)} \delta^{2n(j)} \lb{6.14}
\end{align}
Using \eqref{6.10}, we see
\[
\limsup\|Q_{2n(j)}\|_{L^2 (\ti\mu\restriction I)}^{1/2n(j)} \leq d^{1/2}\delta^{1/2} <
C(\supp(\mu\restriction I))\leq C(\supp(\ti\mu\restriction I))
\]
if we take $\delta$ small. Since \eqref{1.12},
\begin{equation} \lb{6.15}
\|P_{2n(j)} (\dott, \ti\mu\restriction I)\|_{L^2(\ti\mu\restriction I)} \leq
\|Q_{2n(j)}\|_{L^2( \ti\mu\restriction I)}
\end{equation}
we see $\ti\mu\restriction I$ is not regular.
\end{proof}

\begin{proof}[Proof of Theorem~\ref{T6.4}] By Proposition~\ref{P6.7}, we can find $\mu$ so $\mu$ is regular,
$\mu\restriction I$ is not regular but for some $a>0$ (with $\mu_I=\mu\restriction I$),
\begin{equation} \lb{6.16}
\int_I \, \abs{P_n(x,\mu_I)}^2\, d\mu\geq a^n
\end{equation}

Fix $x_0\in I_\intt$. Let $d=\diam (E)$ and for $\ell$ to be picked shortly, let
\begin{equation} \lb{6.17}
Q_{n(2\ell+1)}(x) = P_n(x,\mu_I) \biggl( 1-\f{(x-x_0)^2}{d^2}\biggr)^{\ell n}
\end{equation}
Since $P$ obeys \eqref{6.13}, if we define
\begin{equation} \lb{6.18}
\eta = \max_{x\in E\setminus I}\, \biggl( 1-\f{(x-x_0)^2}{d^2}\biggr) <1
\end{equation}
we have for $x\notin I$,
\begin{equation} \lb{6.19}
\abs{Q_{n(2\ell+2)}(x)}\leq \eta^{\ell n} d^n
\end{equation}
Choose $\ell$ so
\begin{equation} \lb{6.20}
(\eta^\ell d)^2 <a
\end{equation}
Then, by \eqref{6.17}, \eqref{6.19}, and \eqref{6.20},
\begin{equation} \lb{6.21}
\int_K \abs{Q_{n(2\ell+1)}(x)}^2\, d\mu \leq 2\int \abs{P_n(x,\mu_I)}^2\, d\mu
\end{equation}
so
\begin{equation} \lb{6.22}
\abs{p_n(x_0; d\mu_I)}^{1/n} \leq \f{2^{1/2n} \abs{Q_{n(2\ell +1)}(x_0)}^{1/n}}
{\|Q_{n(2\ell+1)}\|_{L^2 (K,d\mu)}^{1/n}}
\end{equation}
so, by \eqref{1.22} and regularity of $\mu$, for $\mu$-a.e.\ $x_0$ in $I^\intt\cap \supp(d\mu)$,
\[
\limsup \abs{p_n (x_0; d\mu_I)}^{1/n}\leq 1
\]

But then Theorem~\ref{T1.8} implies $d\mu_I$ is regular. This contradiction proves the theorem.
\end{proof}

\section{Ergodic Jacobi Matrices and Potential Theory} \lb{s7}

In this section, we will explore regularity ideas for ergodic half- and whole-line Jacobi operators
and see this is connected to Kotani theory (see \cite[Sect.~10.11]{OPUC2} and \cite{Dam07}
as well as the original papers \cite{Kot,S168,Ge92,Kot89}). A main goal is to prove Theorems~\ref{T1.13}
and \ref{T1.15A}.

Let $(\Omega,d\sigma)$ be a probability measure space. Let $T\colon\Omega\to\Omega$ be an invertible
ergodic transformation. Let $\ti A,\ti B$ be measurable functions from $\Omega$ to $\bbR$ with
$\ti B$ bounded, $\ti A$ positive, and both $\ti A$ and $\ti A^{-1}$ bounded. For $\omega\in\Omega$
and $n\in\bbZ$, define $a_n(\omega), b_n(\omega)$ by
\begin{equation}\lb{7.1}
a_n (\omega) =\ti A(T^n \omega) \qquad b_n(\omega) = \ti B(T^n \omega)
\end{equation}

By $J(\omega)$, we mean the Jacobi matrix with parameters $\{a_n(\omega), b_n(\omega)\}_{n=1}^\infty$.
By $\ti J(\omega)$, we mean the two-sided Jacobi matrix with parameters $\{a_n(\omega), b_n(\omega)
\}_{n=-\infty}^\infty$. Occasionally we will use $J_k^+(\omega)$ for the one-sided matrix with
parameters $\{a_{k+n}(\omega), b_{k+n}(\omega)\}_{n=1}^\infty$ and $J_k^-(\omega)$ for the
one-sided matrix with parameters $\{a_{k-n}(\omega), b_{k+1-n}(\omega)\}_{n=1}^\infty$.

Spectral measures for one-sided matrices (and vector $\delta_1$) are $d\mu_\omega, d\mu_\omega^{\pm k}$
and for $\ti J(\omega)$, we use $d\ti\mu_{\omega;k}$ for vector $\delta_k$.

For spectral theory, the transfer matrix is basic. Define for $n\in\bbZ$,
\begin{equation}\lb{7.2}
A_n (x,\omega) = \f{1}{a_{n+1}(\omega)}
\begin{pmatrix}
x-b_{n+1}(\omega) & -1 \\
a_{n+1}(\omega)^2 & 0
\end{pmatrix}
\end{equation}
then
\begin{equation}\lb{7.3}
a_{n+1} u_{n+1} + (b_{n+1}-x) u_n + a_n u_{n-1}=0
\end{equation}
is equivalent to
\begin{equation}\lb{7.4}
\begin{pmatrix} u_{n+1} \\ a_{n+1} u_n \end{pmatrix} =
A_n \begin{pmatrix} u_n \\ a_n u_{n-1} \end{pmatrix}
\end{equation}
We define for $n<m$,
\begin{equation}\lb{7.5}
T(m,n;x,\omega) = A_m (x,\omega) A_{m-1} (x,\omega) \dots A_{n+1} (x,\omega)
\end{equation}
and $T(n,n;x,\omega)=1$ and, for $m<n$, $T(m,n;x,\omega) = T(n,m;x,\omega)^{-1}$. Thus, solutions of \eqref{7.3}
obey
\begin{equation}\lb{7.6}
\begin{pmatrix} u_{m+1} \\ a_{m+1} u_m \end{pmatrix} =
T(m,n;x,\omega) \begin{pmatrix} u_{n+1} \\ a_{n+1} u_n \end{pmatrix}
\end{equation}
In particular, for $n\geq 1$,
\begin{equation}\lb{7.7}
\begin{pmatrix} p_{n+1} (x,\omega) \\ a_{n+1} p_n (x,\omega)\end{pmatrix}
=T(n,-1; x,\omega) \begin{pmatrix} 1 \\ 0 \end{pmatrix}
\end{equation}

The ergodic and subadditive ergodic theorems produce the following well-known facts:

\begin{theorem}\lb{T7.1} There exists $\Omega_0\subset\Omega$ of full $\sigma$ measure so that for
$\omega\in\Omega_0$,
\begin{SL}
\item[{\rm{(a)}}] $\sigma (\ti J(\omega))=E$, a fixed perfect subset of $\bbR$ independent of
$\omega$ {\rm{(}}in $\Omega_0${\rm{)}}. Moreover, for any $\omega\in\Omega_0$, each $J_k^\pm$ obeys
\begin{equation}\lb{7.8}
\sigma_\ess (J_k^\pm (\omega)) =E
\end{equation}

\item[{\rm{(b)}}] There is a measure $d\nu_\infty$ with
\begin{equation}\lb{7.9}
\supp(d\nu_\infty)=E
\end{equation}
If $d\nu_n^{k,\pm,\omega}$ is the zero counting measure for $J_k^\pm (\omega)$, then for any
$\omega\in\Omega_0$, as $n\to\infty$,
\begin{equation}\lb{7.10}
d\nu_n^{k,\pm,\omega} \overset{\omega}{\longrightarrow} d\nu_\infty
\end{equation}

\item[{\rm{(c)}}] Define the Lyapunov exponent $\gamma(z)$ for $z$ by
\begin{equation}\lb{7.11}
\gamma(z) =\lim_{n\to\infty}\, \bbE(\tfrac{1}{n}\log \|T(n-1, -1; z,\omega)\|)
\end{equation}
where
\begin{equation}\lb{7.11a}
\bbE(f) =\int f(\omega)\, d\sigma(\omega)
\end{equation}
and \eqref{7.11} includes that the limit exists. Moreover, for any $k\in\bbZ$,
\begin{align}
\gamma(z) &=\lim_{n\to\infty}\, \bbE(\tfrac{1}{n} \log\|T(n+k,k;z,\omega)\|) \lb{7.12} \\
\gamma(z) &=\lim_{n\to\infty}\, \bbE(\tfrac{1}{n} \log \|T(k-n,k;z,\omega)\|) \lb{7.13}
\end{align}

\item[{\rm{(d)}}] For any $\omega\in\Omega_0$ and $z\notin E$ and $k$ fixed,
\begin{align}
\lim_{n\to\infty}\, \tfrac{1}{n} \log \|T(n+k, k;z,\omega)\| &=\gamma(z) \lb{7.14} \\
\lim_{n\to\infty}\, \tfrac{1}{n} \log \|T(k-n, k;z,\omega)\| &=\gamma(z) \lb{7.15}
\end{align}

\item[{\rm{(e)}}] For any $z\in E$ and $\sigma$-a.e.\ $\omega\in\Omega_0$, \eqref{7.14} and
\eqref{7.15} hold.

\item[{\rm{(f)}}] For $\omega\in\Omega_0$, $\lim_{n\to\infty} (a_1\dots a_n)^{1/n}=A$ exists and
is $\omega$-independent, and one has the Thouless formula,
\begin{equation}\lb{7.16}
\gamma(z) =\log(A^{-1}) + \int \log(\abs{z-x})\, d\nu_\infty(x)
\end{equation}
Moreover, for all $z$,
\begin{equation}\lb{7.17}
\gamma(z) \geq 0
\end{equation}
\end{SL}
\end{theorem}

\begin{remarks} 1. For proofs, see \cite{CarLac,CFKS,PasFig,OPUC2}. \eqref{7.16} is due (in the physics
literature) to Herbert--Jones \cite{HJ} and Thouless \cite{Thou}. It is, of course, just \eqref{1.4} for
$z\notin\cvh(\Sigma)$. Almost everything else here is a simple consequence of the Birkhoff ergodic
theorem/the Kingman subadditive ergodic theorem and translation invariance which implies,
for example, that the expectation in \eqref{7.12} is $k$-independent for each $n$.

\smallskip
2. There are two subtleties to OP readers. First, \eqref{7.13} comes from $\|A^{-1}\|=\|A\|$ for
$2\times 2$ matrices $A$ with $\det(A)=1$. It implies that the Lyapunov exponent is the same in both
directions. $\det(T)=1$ also implies \eqref{7.17}.

\smallskip
3. The second subtlety concerns equality in \eqref{7.16} for all $z$, including those in $\Sigma$.
This was first proven by Avron--Simon \cite{AvS83}; the simplest proof is due to Craig--Simon
\cite{CrS83} who were motivated by work of Herman \cite{Her83}. The point is that, in general,
$\limsup \f{1}{n} \log \|T(n+k,k;z,\omega)\|$ (and $\limsup \f{1}{n} \log \abs{p_n (z,\omega)}$)
may not be upper semicontinuous but $\bbE (\f{1}{n}\log \|T(n+k,k;z,\omega)\|)$ is because of
translation invariance, H\"older's inequality, and
\begin{equation}\lb{7.18}
T(n+\ell+k,k ; z,\omega)= T(n+\ell+k, \ell+k; z,T^\ell \omega)
T(\ell+k, k;z,\omega)
\end{equation}
This implies that the expectation is subadditive so the limit is an $\inf$.
\end{remarks}

Two main examples are the Anderson model and almost periodic functions. For the former, $(a_n(\omega),
b_n(\omega))$ are independent $(0,\infty)\times\bbR$-valued (bounded with $a_n^{-1}$ also bounded)
identically distributed random variables. In the almost periodic case, $\Omega$ is a finite- or
infinite-dimensional torus with $d\sigma$ Haar measure and $\ti A,\ti B$ continuous functions.
A key observation (of Avron--Simon \cite{AvS83}) is that in this almost periodic case, the
density of states exists for all, not only a.e., $\omega\in \Omega$ so we can then take $\Omega_0
=\Omega$ in Theorem~\ref{T7.1}.

Here is the first consequence of potential theory ideas in this setting:

\begin{theorem}\lb{T7/2} $E$ has positive capacity; indeed,
\begin{equation}\lb{7.19}
C(E)\geq A
\end{equation}
Moreover, $E$ is always potentially perfect {\rm{(}}as defined in Appendix~A{\rm{)}}.
Each $d\mu_\omega$ {\rm{(}}$\omega\in\Omega_0${\rm{)}} is regular if and only if equality holds
in \eqref{7.19}.
\end{theorem}

\begin{proof} Use $\gamma\geq 0$ in \eqref{7.16}, integrating $d\nu_\infty$, to see that
\begin{equation}\lb{7.20}
\calE (\nu_\infty)\leq \log (A^{-1})
\end{equation}
so $\calE (\nu_\infty)<\infty$, implying $C(E) >0$. By \eqref{7.20}, we get \eqref{7.19}.

By \eqref{7.20}, $\nu_\infty$ has finite energy and so, by Proposition~\ref{PA.3}, $\nu_\infty$ gives
zero weight to any set of capacity zero. It follows that if $x\in\supp(d\nu_\infty)$, then $C((x-\delta,
x+\delta)) >0$ for all $\delta$. By \eqref{7.9}, $E$ is potentially perfect.

By definition of $A$, regularity for all $\omega\in\Omega_0$ is equivalent to $C(E)=A$.
\end{proof}

\begin{proof}[Proof of Theorem~\ref{T1.13}] We will prove that \eqref{1.34b} holds. By \eqref{7.20},
$\nu_\infty$ has finite Coulomb energy, so $\nu_\infty$ gives zero weight to sets of zero capacity.
Since equality holds in \eqref{A.17}, q.e.\ on $E$, we conclude that
\begin{alignat*}{2}
\log (C(E)^{-1}) &= \int d\nu_\infty (x)\, \Phi_{\rho_E}(x)  \\
&= \int d\rho_E (x)\, \Phi_{\nu_\infty}(x) \qquad&&\text{(by \eqref{A.2})} \\
&= \int d\rho_E (x)\, [\log(A^{-1}) -\gamma(x)] \qquad && \text{(by \eqref{7.16})}
\end{alignat*}
This is \eqref{1.34b}.

By \eqref{1.34b}, we have \eqref{1.27} $\Leftrightarrow \int \gamma(x)\, d\rho_E(x)=0$ which,
given that $\gamma(x)\geq 0$, holds if and only if $\gamma(x)=0$ for $\rho_E$-a.e.\ $x$.

If \eqref{1.27} holds, then each $d\mu_\omega$ is regular, so by Theorem~\ref{T1.6}, $d\nu_\infty =
d\rho_E$. The converse part follows from Theorem~\ref{T1.7}.
\end{proof}

{\it Note}: Remling remarked to me that Theorem~\ref{T1.13} has a deterministic analog with
essentially the same proof.

\smallskip
Kotani theory says something about when $\gamma(x)=0$ but we have not succeeded in making a
tight connection, so we will postpone the precise details until we discuss conjectures in the
next section. As a final topic, we want to prove Theorem~\ref{T1.15A} and a related result.

\begin{proof}[Proof of Theorem~\ref{T1.15A}] By \eqref{4.10} for $d\mu_\omega$-a.e.\ $x$, we have
\begin{equation}\lb{7.21}
\limsup \abs{p_n(x)}^{1/n} \leq 1
\end{equation}
On the other hand, by the upper envelope theorem and \eqref{1.4} for q.e.\ $x$,
\begin{align}
\lim\, \abs{p_n(x)}^{1/n} &= A^{-1} \exp (-\Phi_{\nu_\infty}(x)) \lb{7.22} \\
&= \exp (\gamma(x)) \lb{7.23}
\end{align}
by \eqref{7.16}. Let $Q_\omega$ be the capacity zero set where \eqref{7.23} fails.

On $S$, $\exp (\gamma(x)) >1$, so since \eqref{7.21} holds for a.e.\ $x$, we have $d\mu_\omega
(S\setminus Q_\omega)=0$ as claimed.
\end{proof}

\begin{remark} All we used was that $d\nu_\infty$ is the limit of $d\nu_n$, so this holds for all
$\omega\in\Omega_0$. In particular, in the almost periodic case, it holds for all $\omega$ in the
hull.
\end{remark}

One is also interested in the whole-line operator.

\begin{theorem}\lb{T7.3} Let $\ti J(\omega)$ be the whole-line Jacobi matrix associated with
$\{a_n(\omega), b_n(\omega)\}_{n=-\infty}^\infty$ and $d\mu_{\omega,k}$ its spectral measures.
Let $S\subset\bbR$ be the Borel set of $x$ with $\gamma(x) >0$. Then for each $\omega\in\Omega_0$,
there exists a set $\ti Q_\omega$ of capacity zero so that
\begin{equation}\lb{7.24}
\mu_{\omega,k}(S\setminus \ti Q_\omega)=0
\end{equation}
for all $k$.
\end{theorem}

\begin{proof} By \eqref{7.7}, the transfer matrix $T(n,-1;x,\omega)$ has matrix elements given by
$p_{n+1},p_n$ and the second kind polynomials $q_{n+1}, q_n$. As in the last proof, there is a set
$\ti Q_\omega^{(1)}$ of capacity zero so for $x\notin Q_\omega^{(1)}$,
\begin{equation}\lb{7.25}
\lim\, \abs{p_n(x)}^{1/n} = \exp (\gamma(x))
\end{equation}
and (zeros of $p_n$ and $q_n$ interlace, so the zero counting measure for $q_n$ also converges to
$d\nu_\infty)$
\begin{equation}\lb{7.26}
\lim\, \abs{q_n(x)}^{1/n} =\exp(\gamma(x))
\end{equation}

In particular, for $x\notin \ti Q_\omega^{(1)}$, we have
\begin{equation}\lb{7.27}
\lim_{n\to\infty}\, \tfrac{1}{n}\, \log \|T(n, -1; x,\omega)\| =\gamma(x)
\end{equation}
By the Ruelle--Osceledec theorem (see, e.g., \cite[Sect.~10.5]{OPUC2}), for any $w\neq 0
\in\bbC^2$, either
\begin{equation}\lb{7.28}
\|T(n, -1;x,\omega)w\|^{1/n} \to e^{\gamma(x)}
\end{equation}
or
\begin{equation}\lb{7.29}
\|T(n, -1; x,\omega)w\|^{1/n} \to e^{-\gamma(x)}
\end{equation}
Similarly, there is a set $\ti Q_\omega^{(2)}$ of capacity zero with similar behavior as $n\to\infty$.

This says that every solution of \eqref{7.3} for $x\notin \ti Q_\omega^{(1)}\cup\ti Q_\omega^{(2)}$
either grows exponentially at $\pm\infty$ or decays exponentially. Thus, polynomial boundedness
implies $\ell^2$ solutions. If $\ti Q_\omega^{(3)}$ is the set of eigenvalues of $\ti J(\omega)$
which is countable and so of capacity zero, and if $\ti Q_\omega = \ti Q_\omega^{(1)} \cup
\ti Q_\omega^{(2)} \cup \ti Q_\omega^{(3)}$, then
\[
\ti J(\omega)u=xu \text{ with $u$ polynomially bounded} \Rightarrow x\in \ti Q_\omega
\]
By the BGK expansion discussed in Section~\ref{s4}, this implies the spectral measures of $\ti J(\omega)$
are supported on $\ti Q_\omega$, that is, \eqref{7.24} holds.
\end{proof}

\begin{remarks} 1. The reader will recognize this proof as a slight variant of the Pastur--Ishii
argument \cite{Pas80,Ishii} that proves absence of a.c.\ spectrum on $S$.

\smallskip
2. As above, in the almost periodic case, this holds for all $\omega$ in the hull.

\smallskip
3. This is the first result on zero Hausdorff dimension in this generality. But for suitable
analytic quasi-periodic Jacobi matrices, the result is known; see Jitomirskaya--Last \cite{JL2}
and Jitomirskaya \cite{Jit07}.
\end{remarks}

\section{Examples, Open Problems, and Conjectures} \lb{s8}

Here we consider a number of illustrative examples and raise some open questions and conjectures.
The conjectures are sometimes mere guesses and could be wrong. Indeed, when I started writing
this paper, I had intended to make a conjecture for which a counterexample appears below as
Example~\ref{E8.11}. So the reader should regard the conjectures as an attempt to stimulate
work with my own guesses. I will try to explain my guesses, but they are not always compelling.

\begin{example}[Random and Decaying Random OPRL]\lb{E8.1} Let $\Omega=\bigtimes_{n=1}^\infty [(0,\infty)
\times\bbR]$ with $d\sigma (\{a_n, b_n\})=\otimes d\eta (a_n, b_n)$, where $\eta$ is a measure of compact
support on $(0,\infty)\times\bbR$. For each $\omega\in\Omega$, there is an associated Jacobi matrix,
and we want results on $J(\omega)$ that hold for $\sigma$-a.e.\ $\omega$. The traditional Anderson model
is the case where $a_n\equiv 1$ and $b_n$ is uniformly distributed on $[\alpha,\beta]$, that is,
$d\eta (a,b)=\delta_{a1} \f{1}{\beta-\alpha} \chi_{(\alpha,\beta)}(b)\, db$. The decaying random
model has two extra parameters, $\lambda\in (0,\infty)$ and $\gamma\in (0,1)$, takes $\ti a_n(\omega)
\equiv 1$, $\ti b_n(\omega)$ the Anderson model with $\beta=-\alpha =1$, and takes
\begin{align}
b_n (\omega) &=\lambda n^{-\gamma} \ti b_n(\omega) \lb{8.1} \\
a_n (\omega) &=1 \lb{8.2}
\end{align}
The Anderson model is ergodic; the decaying random model is not. The Anderson model goes back to
his famous work \cite{Ander} with the first mathematical results by Kunz--Souillard \cite{KS}
and the decaying model to Simon \cite{Sim157} (see also \cite{S265}).

For the Anderson model, it is known for a.e.\ $\omega$, $\sigma_\ess (J(\omega)) = [-2+\alpha,
2+\beta]$, while for the decaying random model, $\sigma_\ess (J(\omega))=[-2,2]$ by Weyl's theorem
(i.e., $J(\omega)$ is in Nevai class). Clearly, $(a_1\dots a_n)^{1/n}=1$. For the Anderson model,
\begin{equation}\lb{8.3}
C(\sigma_\ess (J(\omega))) =\tfrac14\, (4+(\beta-\alpha)) >1
\end{equation}
while for the decaying Anderson model,
\begin{equation}\lb{8.4}
C(\sigma_\ess (J(\omega))) =1
\end{equation}
so the former is not regular, while the latter is.

Of course, for the regular model, the density of zeros is the equilibrium measure where $\rho_E(x)
=\f{d\rho}{dx}=\f{1}{\pi} (4-x^2)^{-1/2}$ by \eqref{A.22}. For the Anderson model, on the other
hand, $\f{d\nu}{dx}$ is very different. It is $C^\infty$ even at the endpoints (by \cite{S188})
and decays exponentially fast to zero at the ends of the spectrum (Lifshitz tails; see
\cite{Kirsch2007}).

The Anderson model is known to have dense pure point spectrum and so is the decaying model
if $\gamma <\f12$. It is known for the Anderson model (see \cite{S250}) that for some
$\omega$-dependent labeling of the eigenvalues,
\begin{equation}\lb{8.5}
d\mu_\omega =\sum w_n (\omega) \delta_{e_n(\omega)}
\end{equation}
where for some $c>0$,
\begin{equation}\lb{8.6}
\abs{w_n(\omega)} \leq e^{-c\abs{n}}
\end{equation}

The same methods should allow one to prove for the decaying model on each $[-A,A]\subset (-2,2)$
that there is a labeling so that
\begin{equation}\lb{8.7}
\abs{w_n(\omega)} \leq e^{-c\abs{n}^{1-2\gamma}}
\end{equation}

One expects that there are lower bounds of the same form and that the labels are such that
the $e_n(\omega)$ are quasi-uniformly distributed (i.e., for $n \gg m$, the first $n$
$e_j(\omega)$ are at least within $\f1m$ of each point away from the edge of the spectrum).
If these expectations are met, this example nicely illustrates Theorems~\ref{T1.11} and
\ref{T1.12}.

In the not regular Anderson case, one expects $\mu_\omega ([\f{j}{m}, \f{j+1}{m}])\sim e^{-cm}$
for fixed $c$, while in the regular decaying random model, one expects $\mu_\omega ([\f{j}{m},
\f{j+1}{m}]) \sim e^{-cm^{1-2\gamma}} > C_\eta e^{-\eta m}$ for any $\eta$.
\qed
\end{example}

\begin{example}[Generic Regular Measures]\lb{E8.2} Fix $a_n\equiv 1$, $0<\gamma <\f12$,
and let $\calB=\{\{b_n\}\mid \lim n^\gamma b_n\to 0\}$ normed by $|||b|||=\sup_n
\abs{n^\gamma b_n}$. It is known (\cite{S234}; see also \cite{LS,CN}) that for a dense
$G_\delta$ in $\calB$, the associated Jacobi matrix has singular continuous measure.
We believe there is some suitable sense in which a generic regular measure is singular
continuous.
\qed
\end{example}

\begin{example}[Almost Mathieu Equation]\lb{E8.3} Perhaps the most studied model in spectral
theory is the whole-line Jacobi matrix with $a_n\equiv 1$ and
\begin{equation}\lb{8.8}
b_n=\lambda \cos(n\alpha+\theta)
\end{equation}
where $\lambda,\alpha,\theta$ are parameters with $\f{\alpha}{\pi}$ irrational. (See \cite{Jit07}
for a review on the state of knowledge.) We will use some of the most refined results and
comment on whether they are needed for the main potential theoretic conclusions.  We fix
$\alpha,\lambda$. $\theta\in\Omega = [0,2\pi)$ labels the hull of an almost periodic
family.

It is known since Avron--Simon \cite{AvS83} that for $\abs{\lambda}>2$, there is no a.c.\
spectrum for almost all $\theta$ (and by Kotani \cite{Kot97} and Last--Simon \cite{S263},
for all $\theta$) and by \cite{AvS83} for $\alpha$ which are Liouville numbers (irrational
but very well approximated by rationals) only singular continuous spectrum for all $\theta$.
Jitomirskaya \cite{Jito99} proved that for $\alpha$'s with good Diophantine properties
and $\abs{\lambda} >2$, there is dense pure point spectrum for a.e.\ $\theta$ (and there
is also singular continuous spectrum for a dense set of $\theta$'s \cite{S236}). On the other
hand, Last \cite{Last92} proved that for $\abs{\lambda} <2$ and all irrational $\alpha$
that the spectrum is a.c.\ for almost all $\theta$ (now known for all $\theta$ \cite{Kot97,
S263}). It is now known the spectrum in this region is purely a.c.\ (see \cite{AJ,AD,Av-prep1}).

At the special point $\lambda =2$, it is known  that for all irrational $\alpha$, the spectrum
has measure zero \cite{Last92,AK}, and therefore for all irrational $\alpha$ and a.e.\ $\theta$,
the spectrum is purely singular continuous \cite{GJLS}.

An important special feature for our purposes is Aubry duality (found by Aubry \cite{Aubry80};
proven by Avron--Simon \cite{AvS83}) that relates the Lyapunov exponent $\gamma(z)$ and
integrated density of states, $k(E)=\int_{-\infty}^E d\nu_\infty(x)$, for $\alpha$ fixed
(they are $\theta$-independent) at $\lambda$ and $\f{4}{\lambda}$. Making the $\lambda$-dependence
explicit,
\begin{equation}\lb{8.9}
k\biggl( E,\,\f{4}{\lambda}\biggr) = k\biggl( \f{2E}{\lambda}\, , \lambda\biggr) \qquad
\gamma\biggl( z,\, \f{4}{\lambda}\biggr) = \gamma \biggl(\f{2z}{\lambda}\, ,\lambda\biggr)
+ \log\biggl( \f{\lambda}{2}\biggr)
\end{equation}

Kotani theory implies in the a.c.\ region (i.e., $\lambda <2$) that $\gamma(E)=0$ for a.e.\
$E\in\spec(J)$, and Bourgain--Jitomirskaya \cite{BJ} proved continuity of $\gamma$. So using
\eqref{8.9},
\begin{equation}\lb{8.10}
E\in\spec(J)\Rightarrow \gamma(E) = \max\biggl( 0, \log \biggl( \f{\lambda}{2}\biggr)\biggr)
\end{equation}
and, in particular, $\gamma(E)=0$ on the spectrum if $\abs{\lambda}\leq 2$.
\qed
\end{example}

We thus have:

\begin{theorem}\lb{T8.4} The density of zeros for the almost Mathieu equation is the
equilibrium measure for its spectrum. For $\lambda\leq 2$, the measures are regular
{\rm{(}}for all $\omega${\rm{)}}. For $\lambda >2$, they are not regular since
\begin{equation}\lb{8.11}
\lim (a_1 \dots a_n)^{1/n} =1 < C(\spec(J_0(\omega))) =\log\biggl( \f{\lambda}{2}\biggr)
\end{equation}
\end{theorem}

\begin{proof} By Theorem~\ref{T1.12}, the measures are regular if $\lambda\leq 2$ since
$\gamma(E)=0$ on the spectrum. That the measure is the equilibrium measure even if $\lambda >2$
follows from \eqref{8.9} as does \eqref{8.11}.
\end{proof}

\begin{remarks} 1. Thus we see an example where the density of zeros is the equilibrium
measure even though $d\mu_\omega$ is not regular. Consistently with Theorem~\ref{T2.4A},
$d\mu_\omega$ lives on a set of capacity $0$ by Theorem~\ref{T7.3}.

\smallskip
2. If we knew a priori that $d\rho_{\sigma_\ess (J_\omega)}$ were absolutely continuous, Kotani
theory then would suffice for Theorem~\ref{T8.4}. But as it is, we need the continuity result
of \cite{BJ}.

\smallskip
3. It should be an exceptional situation that $J(\omega)$ has some singular spectrum but
the density of states is still $d\rho_E$. In particular, if there are separate regions in
$\sigma(J)$ of positive capacity where $\gamma(x)=0$ and where $\gamma(x) >0$, the density
of states cannot be $d\rho_E$ since, for it, $\gamma(x)$ is constant on $\supp(d\rho_E)$.
For examples with such coexistent spectrum (some only worked for the continuum case),
see \cite{Bour2002,Bour2005,FK2005,FK2006}.
\end{remarks}

\begin{example}[Rotation Invariant Anderson Model OPUC] \lb{E8.5} Let $d\sigma_0$ be a rotation
invariant measure on the disk, $\bbD$ (i.e., on $\ol{\bbD}$ with $\sigma_0 (\partial\bbD)=0$).
Let $\sigma$ on $\bigtimes_{j=0}^\infty \bbD$ be $\otimes_{j=0}^\infty d\sigma_0 (z_j)$. The
ergodic OPUC with Verblunsky coefficients $\alpha_j$ distributed by $\sigma$ is called the
rotation invariant Anderson model, and it is discussed in \cite[Sect.~12.6]{OPUC2} and
earlier in Teplyaev \cite{Tep91} and Golinskii--Nevai \cite{GN}.

If
\begin{equation}\lb{8.12}
\int -\log (1-\abs{z}) \, d\sigma_0(z) <\infty
\end{equation}
then
\begin{equation}\lb{8.13}
\lim_{n\to\infty}\, (\rho_0\dots\rho_{n-1})^{1/n} = \exp\biggl( \int \log(1-\abs{z}^2)\,
d\sigma_0(z)\biggr) >0
\end{equation}

If also
\begin{equation}\lb{8.13a}
\int -\log\abs{z}\, d\sigma_0(z) <\infty
\end{equation}
then, by a use of the ergodic theorem,
\[
\lim_{n\to\infty}\, \abs{\alpha_n}^{1/n} =1
\]
with probability $1$. By a theorem of Mhaskar--Saff \cite{MhS1} (see \cite[Thm.~8.1.1]{OPUC1}),
any limit point of the zero counting measure lives on $\partial\bbD$ so, by the ergodic theorem,
$\nu_n$ has a limit $\nu_\infty$ on $\partial\bbD$.

By the rotation invariance of $\sigma_0$, the distribution of $\{\alpha_j\}$ is invariant
under $\alpha_j\to e^{i(j+1)\theta}\alpha_j$. So the collection of measures is rotation
invariant and thus, by ergodicity, $d\nu_\infty$ is rotation invariant, that is, it is
$\f{d\theta}{2\pi}$. By the Thouless formula and \eqref{8.13},
\[
\gamma(e^{i\theta}) =-\int\log (1-\abs{z}^2)\, d\sigma_0(z) >0
\]
so long as $\sigma_0\neq \delta_{z=0}$. This is constant on $\partial\bbD$.

Thus, this family of measures is not regular, but the density of zeros is the equilibrium measure
for $\supp(d\mu_\omega)=\partial\bbD$. This is the simplest example of a nonregular measure for
which the density of zeros is the equilibrium measure. As is proven in Theorem~12.6.1 of
\cite{OPUC2}, the measure is a pure point measure, so $d\mu_w$ is for a.e.\ $\omega$ supported
on a countable set, so of zero capacity, consistent with Theorem~\ref{T2.4A}.
\qed
\end{example}

\begin{example}[Subshifts]\lb{E8.6} This is a rich class of ergodic Jacobi matrices (with
$a_n\equiv 1$), reviewed in \cite{Dam07b} (see also \cite[Sect.~12]{OPUC2}). For many of
them, it is known that $E\equiv\sigma(J)$ is a set of Lebesgue measure zero on which $\gamma(x)$
is everywhere $0$. By Theorem~\ref{T1.13}, $C(E)=1$ and a.e.\ $\omega$ has regular $d\mu_\omega$,
so, in particular, $d\nu_\infty =d\rho_E$.
\qed
\end{example}

Notice that, by Craig's argument (see Theorem~\ref{TA.8}), if $d\mu$ is any probability measure
whose support, $E$, has measure zero, then $G(z)=\int \f{d\mu(y)}{y-z}$ has the form
\begin{equation}\lb{8.13x}
G(z) =-\f{1}{\sqrt{(z-a)(z-b)}\,}\, \prod_{j=1}^\infty \,
\f{(z-\lambda_j)}{\sqrt{(z-\ell_j)(z-u_j)}\,}
\end{equation}
where the gaps in $E$ are $(\ell_j, u_j)$ and $a=\inf \ell_j$, $b=\sup \mu_j$. This is so regular
that we wildly make the following:

\begin{conjecture}\lb{Con8.7} Any ergodic matrix that has a spectrum of measure zero has vanishing
Lyapunov exponent on the spectrum; equivalently, $\gamma(x)>0$ for some $x\in\Sigma$ implies
$\abs{\Sigma} >0$. Such zero Lyapunov exponent examples would thus be regular.
\end{conjecture}

We note that for analytic functions on the circle with irrational rotation, this result
is known to be true \cite{Jito}, following from combining results from Bourgain \cite{Bour} and
Bourgain--Jitomirskaya \cite{BJ}. Of two experts I consulted, one thought it was false and
the other, ``likely true but too little support to make it a conjecture." Fools rush in
where experts fear to tread.

\begin{oq}[The Classical Cantor Set]\lb{OQ8.8} Of course, one of the simplest of measure zero
sets is the classical Cantor set. It would be a good first step to understand its ``isospectral
tori." Which whole-line Jacobi matrices have $\langle \delta_0, (J_0-z)^{-1}\delta_0\rangle =
\text{\eqref{8.13x}}$? Are they regular? As suggested by Deift--Simon \cite{S169}, are they
mainly mutually singular? Are any or all almost periodic?
\end{oq}

\begin{conjecture}[Last's Conjecture]\lb{Con8.8} A little more afield from potential theory, but
worth mentioning, is the conjecture of Last that any ergodic Jacobi matrix (whole- or
half-line) with some a.c.\ spectrum is almost periodic. Does it help to consider the
case where the spectrum is purely a.c.? We note that a result of Kotani \cite{Kot89} implies
Last's conjecture if $a_n,b_n$ take only finitely many values.
\end{conjecture}

And it links up to the next question:

\begin{oq}[Denisov--Rakhmanov Theorem]\lb{OQ8.9} Let $E$ be an essentially perfect set, that is,
for every $x\in E$ and $\delta >0$, $\abs{(x-\delta, x+\delta)\cap E} >0$. In \cite{DKS07}, $E$
was called a DR set if any half-line Jacobi matrix with $\sigma_\ess (J)=\Sigma_\ac(J)=E$ has
a set of right limit points which is uniformly compact (and so the limits are all almost periodic).
A classical theorem of Rakhmanov, as extended by Denisov (see \cite[Ch.~9]{OPUC2}), says that
$[-2,2]$ is a DR set. Damanik--Killip--Simon \cite{DKS07} proved a number of $E$'s, including
those associated with periodic problems, are DR sets. Remling \cite{Remppt} recently proved
any finite union of closed intervals is a DR set, and he remarks that it is possible to combine
his methods with those of Sodin--Yuditskii \cite{SY} to prove that any homogeneous set in the
sense of Carleson (see \cite{SY} for a definition) is a DR set.
\end{oq}

Following this section's trend to make (foolhardy?) conjectures:

\begin{conjecture}\lb{Con8.10} Any essentially perfect compact subset of $\bbR$ is a DR set.
\end{conjecture}

A counterexample would also be very interesting. This is relevant to this paper because,
as we have explained, Widom's theorem (Theorem~\ref{T1.10}) is a kind of poor man's DR condition.

Related to this: it would be interesting to find a proof of the almost periodicity of every
reflectionless two-sided Jacobi matrix with spectrum a finite union of intervals that did not
rely on the theory of meromorphic functions on a Riemann surface.

\begin{example} \lb{E8.11} Remling \cite{Remppt} has some interesting Jacobi matrices that are regular
on $[-2,2]$, not in Nevai class, and have $\Sigma_\ac =[0,2]$. Related examples for Schr\"odinger
operators appeared earlier in Molchanov \cite{Mol}.
\qed
\end{example}

\section{Continuum Schr\"odinger Operators} \lb{s9}

The theory presented earlier was developed by the OP community dealing with discrete
(i.e., difference) equations. The spectral theory community knows there are usually close
analogies between difference and differential equations, so it is natural to ask about
regularity ideas for continuum Schr\"odinger operators---a subject that does not seem to
have been addressed before. We begin this exploration here. This is more a description
of a research project than a final report. We will be discursive without proofs.

The first problem that one needs to address is that there is no natural potential theory
for infinite unbounded sets. $\log\abs{x-y}^{-1}$ is unbounded above and below so
Coulomb energies can go to $-\infty$. Moreover, the natural measures are no longer
probability measures. There is no reasonable notion of capacity, even of renormalized
capacity. But at least sometimes there is a natural notion of equilibrium measure and
equilibrium potential.

Consider $E=[0,\infty)$. We may not know the precise right question but we know the
right answer: For $V=0$, the solutions of $-u'' +Vu=\lambda u$ with $u(0)=0$ are $u(x) =
C\sinh (x\sqrt{-\lambda})$, and so
\begin{equation} \lb{9.1}
\lim_{x\to\infty}\, \f{\log\abs{u(x)}}{x}=\sqrt{-\lambda}
\end{equation}
which must be the correct analog of the potential theorist Green's function. And there is a
huge literature on continuum density of states, which for this case is
\begin{equation} \lb{9.2}
d\rho(\lambda) = \chi_{[0,\infty)} (\lambda)(\lambda)^{-1/2}(2\pi)^{-1}\, d\lambda
\end{equation}
This comes from noting the eigenvalues on $[0,1]$ with $u(0)=u(L)=0$ boundary conditions
are $(\f{\pi n}{L})^2$, $n=1,2,\dots$. Here is a first attempt to find the right question.

It is the derivative of $\int \log\abs{x-y}^{-1}\, d\mu(x)$ that is a Herglotz function, so we
make

\begin{definition} We say $d\nu$ is an {\it equilibrium measure\/} associated to a set
$E\subset [a,\infty)$ for some $a$, if and only if there is a Herglotz function, $F_E(z)$,
on $\bbC$ so that
\begin{SL}
\item[(i)] $\Ima F(\lambda +i0)$ is supported on $\lambda\in E$.
\item[(ii)] $\Real F(\lambda +i0)=0$ for a.e.\ $\lambda\in E$.
\item[(iii)] $F(\lambda)\to 0$ as $\lambda\to -\infty$.
\item[(iv)] $\pi^{-1} F(\lambda +i\veps)\, d\lambda \overset{w}{\longrightarrow} d\nu (\lambda)$
\item[(v)] For any bounded connected component $(a,b)$ of $\bbR\setminus E$, we have
\begin{equation} \lb{9.3}
\int_a^b F(\lambda)\, d\lambda =0
\end{equation}
\end{SL}
We will say $d\nu$ is {\it normalized\/} if
\begin{equation} \lb{9.4}
F(\lambda)\sim \tfrac12\, (-\lambda)^{-1/2} (1+ o(1))
\end{equation}
near $-\infty$.
\end{definition}

The reason for choosing \eqref{9.3} and \eqref{9.4} will be made clear shortly. Once we have
$F$, we define the equilibrium potential of $E$ by
\begin{equation} \lb{9.5}
\Phi_E(z)=\Real \biggl( \int_{x_0}^z F(\omega)\, d\omega\biggr)
\end{equation}
where $x_0\in E$ and the integral is in a path in $\bbC\setminus [a,\infty)$ with $a=\inf
\{y\in\bbR\mid y\in E\}$. That $\Real F=0$ on $E$ and that \eqref{9.3} holds show $\Phi_E$ is
independent of $x_0$. \eqref{9.3} also implies $\Phi_E(z)=0$ on $E$. For this reason, we
need to take $E=\sigma_\ess (-\f{d^2}{dx^2}+V)$, not $\sigma(-\f{d^2}{dx^2}+V)$.

With \eqref{9.4}, we have
\begin{equation} \lb{9.6}
\Phi_E(z) =\Real (\sqrt{-z}\,)(1+o(1))
\end{equation}
near $-\infty$. We can explain why we normalize as we do. For regular situations, we expect
that the absolute value of the eigenfunction, $\psi_z(x)$, analogous to OPs (see below)
are asymptotic to
\[
\exp(x\Phi_E(z))
\]
as $x\to\infty$. This, in turn, is related to integrals of the negative of the real part of
\begin{equation} \lb{9.7}
m(z,x) = \f{\eta'_z(x)}{\eta_z(x)}
\end{equation}
where $\eta$ is the solution of $L^2$ at infinity.

It is a result of Atkinson \cite{Atk} (see also \cite{S272}) that in great generality that as
$\abs{z}\to\infty$, $-\f{d^2}{dx^2}+V$ is bounded from below in sectors about $(-\infty, a)$
and, in general, in sectors $\abs{\arg z}\in (c,\pi-\veps)$,
\begin{equation} \lb{9.8}
m(z,\lambda) =-\sqrt{-z} +o(1)
\end{equation}
$\psi$ should grow as the inverse of $\eta$, so $\Phi\sim -m$ as $z\to-\infty$.

This is stronger than \eqref{9.6} (if one can interchange limits $x\to\infty$ and
$z\to\infty)$ since the error in \eqref{9.6} is $o(1)\sqrt{-z}$, while in \eqref{9.7} it
is $o(1)$. The lack of a constant term is an issue to be understood.

If we take $E=[0,\infty)$ since $F'>0$ on $(-\infty, 0$), we have $F>0$ on $(-\infty, 0)$,
and so $\log F(x+i0)$ has boundary values $0$ on $(-\infty, 0)$ and $\f12$ on $(0,\infty)$.
This plus $\log F(z) =o(z)$ at $-\infty$ uniquely determine $\log F$, and so $F$, up to
an overall constant which is fixed by the normalization yielding
\begin{equation} \lb{9.9}
F(z)=\f{1}{2\sqrt{-z}\,}
\end{equation}
so there is a unique ``potential" for $[0,\infty)$ that gives the right $\Phi(z)=\sqrt{-z}$.

Similarly, for a finite number of gaps removed from $[0,\infty)$, one gets a unique $F$.
Craig's argument yields $F$ up to positions of zeros in the gap, which are then fixed by
\eqref{9.1}.

\begin{op} Develop a formal theory of equilibrium measures and equilibrium potentials for
unbounded sets that are ``close" to $[0,\infty)$ (e.g., one might require that $E\setminus
[0,\infty)$ has finite Lebesgue measure). Can one understand the $o(1)$ in \eqref{9.8}
from this theory?
\end{op}

With potentials in hand, we can define regularity. We recall first that given any $V$ on
$[0,\infty)$ which is locally in $L^1$, one can define the regular solution, $\psi(x,z)$,
obeying
\begin{gather}
-\psi'' (x,z) +V(x) \psi(x,z)= z\psi (x,z) \lb{9.10} \\
\psi (0,z) =0 \qquad \psi' (0,z) =1 \lb{9.11}
\end{gather}
Here $\psi$ is $C^1$ (and so, locally bounded), its second distributional derivative is $L^1$,
and obeys \eqref{9.10} as a distribution. For fixed $x$, $\psi$ is an entire function of $x$
of order $\f12$. If $\psi$ is not $L^2$ at infinity for (one and hence all) $z\in\bbC_+$, $V$
is called limit point at infinity and then there is a unique selfadjoint operator
$H$ which is formally $-\f{d^2}{dx^2}+V(x)$ with $u(0)=0$ boundary conditions. We only want
to consider the case where $H$ is bounded below (which never happens if $V$ is not limit point).
$\eta_z(x)$ is then the solution $L^2$ at $\infty$ determined up to a constant, so
\begin{equation} \lb{9.12}
m(x,z) = \f{\eta'_z(x)}{\eta_z(x)}
\end{equation}
is determined by $V$\!.

\begin{definition} Let $E=\sigma_\ess(H)$. We say $H$ is {\it regular\/} if and only if for all
$z\notin\sigma(H)$,
\begin{equation} \lb{9.13}
\limsup_{x\to\infty}\, \f{1}{x}\, \log\abs{\psi(x,z)} =\Phi_E(z)
\end{equation}
\end{definition}

Of course, for this to make sense, $E$ has to be a set for which there is a potential. This will
eliminate a case like $V(x)=x^2$ where $\sigma_\ess(H)$ is empty). We expect the following
should be easy to prove:

\begin{mt} \lb{MT9.2}
\begin{SL}
\item[{\rm{(a)}}] If $H$ is regular, for $z\notin\sigma(H)$, $\limsup$ in \eqref{9.13} can be
replaced by $\lim$.
\item[{\rm{(b)}}] $H$ is regular if and only if \eqref{9.13} holds q.e.\ on $E$ {\rm{(}}where
$\Phi_E(z)$ is q.e.\ $=0${\rm{)}}.
\item[{\rm{(c)}}] If $H$ is regular, the density of states exists and equals the equilibrium measure
for $E$.
\item[{\rm{(d)}}] Conversely, if the density of states exists and equals the equilibrium measure
for $E_0$, either $H$ is regular or else the spectral measure for $H$ is supported on a set
of capacity zero.
\item[{\rm{(e)}}] If $H$ is regular and
\[
\lim_{n\to\infty} \int_n^{n+1} \abs{(\delta V)(x)}\, dx =0
\]
then $H+\delta V$ is also regular.
\end{SL}
\end{mt}

\begin{remarks} 1. Here capacity zero and q.e.\ are defined in the usual way, that is, any probability
measure of compact support contained in $E$ has infinite Coulomb energy.

\smallskip
2. By density of states, we mean the following (see \cite{S147,BP,IshM,Jo88,KKL,Nab1,Pak87}).
Take $H_L$ to be the operator $-\f{d^2}{dx^2}+V$ with $u(0)=0$ boundary conditions on $L^2 ([0,L],dx)$.
This has infinite but discrete spectrum $E_{1,L} < E_{2,L} < E_{3,L} < \dots$ (the solutions of
$\psi(L,z)=0$). Let $d\nu_L$ be the infinite measure that gives weight $\f{1}{L}$ to each $E_{j,L}$.
If $\wlim d\nu_L$ (as functions on continuous functions of compact support) exist, we say the density
of states exists and the limit is called the density of states.

\smallskip
3. (e) should follow from a standard use of an iterated DuHamel's formula.
\end{remarks}

\begin{op} Verify Metatheorem~\ref{MT9.2} and explore, in particular, analogs of
\begin{SL}
\item[(a)] Widom's theorem, Theorem~\ref{T1.10}
\item[(b)] The Stahl--Totik criterion, Theorem~\ref{T1.11}
\item[(c)] For ergodic continuum Schr\"odinger operators, the analog of Theorem~\ref{T1.15A}.
\end{SL}
\end{op}

\section*{Appendix A: A Child's Garden of Potential Theory in the Complex Plane} \lb{AppA}
\renewcommand{\theequation}{A.\arabic{equation}}
\renewcommand{\thetheorem}{A.\arabic{theorem}}
\setcounter{theorem}{0}
\setcounter{equation}{0}

We summarize the elements of potential theory relevant to this paper. For lucid accounts of
the elementary parts of the theory, see the appendix of Stahl--Totik \cite{StT}, Martinez-Finkelshtein
\cite{M-F}, and especially Ransford \cite{Ran}. More comprehensive are Helms \cite{Helms},
Tsuji \cite{Tsu}, and especially Landkof \cite{Land}. We will try to sketch some of the most
important notions in remarks but refer to the texts, especially for the more technical aspects.

The two-dimensional Coulomb potential is $\log\abs{x-y}^{-1}$ which has two lacks compared to
the more familiar $\abs{x-y}^{-1}$ of three dimensions: It is neither positive nor positive
definite. We will deal with lack of positivity by only considering measures of compact
support, and conditional positive definitiveness can replace positive definitiveness in some
situations.

If $\mu$ is a positive measure of compact support on $\bbC$, its {\it potential\/} is defined by
\begin{equation} \lb{A.1}
\Phi_\mu(x) =\int \log \abs{x-y}^{-1} \, d\mu(y)
\end{equation}
Because $\mu$ has compact support, $\log\abs{x-y}^{-1}$ is bounded below for $x$ fixed, so if
we allow the value $+\infty$, $\Phi_\mu$ is always well defined and Fubini's theorem is
applicable and implies that for another positive measure, $\nu$, also of compact support, we have
\begin{equation} \lb{A.2}
\int \Phi_\mu(x)\, d\nu(x) =\int \Phi_\nu(x)\, d\mu(x)
\end{equation}

Sometimes it is useful to fix $M>0$ and define the cutoff
\begin{equation} \lb{A.3}
\Phi_\mu^M(x)=\int \log [\min(M,\abs{x-y}^{-1})]\, d\mu(y)
\end{equation}
$\Phi_\mu^M$ is continuous and $\Phi_\mu^M$ is an increasing sequence in $M$\!, so

\begin{proposition}\lb{PA.1} $\Phi_\mu(x)$ is harmonic on $\bbC\setminus\supp(d\mu)$, lower
semicontinuous on $\bbC$, and superharmonic there.
\end{proposition}

One might naively think that $\Phi_\mu(x)$ only fails to be continuous because it can go
to infinity and that it is continuous in the extended sense---but that is wrong!

\begin{example}\lb{AE.1A} Let $x_n=-n^{-1}$ and let
\begin{equation}\lb{A.4a}
d\mu=\sum_{n=1}^\infty n^{-2} \delta_{x_n}
\end{equation}
Then $\Phi_\mu(x_n)=\infty$ and $x_n\to 0$, but
\begin{equation}\lb{A.4b}
\Phi_\mu(0)=\sum_{n=1}^\infty n^{-2} \log n <\infty
\end{equation}
Notice that this is consistent with lower semicontinuity, that is, $\Phi_\mu(\lim x_n) \leq
\liminf \Phi_\mu(x_n)$. Also notice, given Hydrogen atom spectra, that this example is
relevant to spectral theory.

Lest you think this kind of behavior is only consistent with unbounded $\Phi_\mu$, one
can replace $\delta_{x_n}$ by a smeared out probability measure, $\eta_n$ (using
equilibrium measures on a small interval, $I_n$, about $x_n$), so $\Phi_{\eta_n}=
\lambda n^2$ on $I_n$ and have with $\mu=\sum n^{-2}\eta_n$, then $\Phi_\mu$ is bounded,
$\Phi_\mu(x_n) \geq\lambda$ while $\Phi_\mu(0)\leq 2\sum_{n=1}^\infty n^{-2} \log n$.
Hence one loses continuity for $\lambda$ large.
\qed
\end{example}

The following is sometimes useful:

\begin{proposition}\lb{PA.1B} If $\Phi_\mu(x)$ restricted to $\supp(\mu)$ is continuous,
then $\Phi_\mu$ is continuous on $\bbC$.
\end{proposition}

\begin{remarks} 1. The general case can be found in \cite[Theorem~1.7]{Land}. Here
we will sketch the case where $\supp(\mu)\subset\bbR$ which is most relevant to OPRL.

\smallskip
2. By lower semicontinuity, if $\Phi_\mu$ fails to be continuous on $\bbC$, there exists
$z_n\to z_\infty$, so $\Phi_\mu (z_n)\to a> \Phi_\mu (z_\infty)$. Continuity off $\supp(\mu)$
is easy, so we must have $z_\infty\in\bbR$ (since we are supposing $\supp(\mu)\subset\bbR$).

\smallskip
3. If $w,x,y\in\bbR$, then $\abs{w-x-iy}^{-1}\leq \abs{w-x}^{-1}$, so
\[
\Phi_\mu (x+iy) \leq \Phi_\mu(x)
\]
and thus $\liminf \Phi_\mu (\Real z_n) \geq a > \Phi_\mu (\Real z_\infty)$, so without loss, we
can suppose $z_n$ are real.

\smallskip
4. If $(\alpha,\beta)\subset\bbR\setminus\supp(\mu)$ with $\alpha,\beta\in\supp(\mu)$, it is easy
to see that $\Phi_\mu(x)$ is continuous when restricted to $[\alpha,\beta]$ (using monotone
convergence at the endpoints) and convex on $[\alpha,\beta]$ since $\log\abs{x}^{-1}$ is convex.
Thus, $\max_{[\alpha,\beta]} \Phi_\mu(x)=\max (\Phi_\mu(\alpha), \Phi_\mu(\beta))$. From this,
it is easy to see that if such a $z_n\in\bbR$ exists, one can take $z_n\in\supp(\mu)$ and so
get a contradiction to the assumed continuity of $\Phi_\mu$ restricted to $\supp(\mu)$.
\end{remarks}

The {\it energy\/} or {\it Coulomb energy\/} of $\mu$ is defined by
\begin{equation} \lb{A.4}
\calE(\mu) =\int \Phi_\mu(x)\, d\mu(x) =\int \log\abs{x-y}^{-1}\, d\mu(x) \, d\mu(y)
\end{equation}
where, again, the value $+\infty$ is allowed. If $E\subset\bbC$ is compact, we say it has
{\it capacity zero\/} if $\calE(\mu)=\infty$ for all $\mu\in\calM_{+,1}(E)$, the probability
measures on $E$. If $E$ does not have capacity zero, then the {\it capacity\/}, $C(E)$,
of $E$ is defined by
\begin{equation} \lb{A.5}
C(E)=\exp(-\min(\calE(\rho)\mid\rho\in\calM_{+,1}(E)))
\end{equation}
One indication that this strange-looking definition is sensible is seen by, as we will
show below (see Example~\ref{EA.10}),
\begin{equation} \lb{A.5a}
C([a,b])=\tfrac14\, (b-a)
\end{equation}

It is useful to define the capacity of any Borel set. For bounded open sets, $U$,
\begin{equation} \lb{A.6}
C(U)=\sup(C(K)\mid K\subset U,\, K \text{ compact})
\end{equation}
and then for arbitrary bounded Borel $X$,
\begin{equation} \lb{A.7}
C(X) =\inf (C(U)\mid X\subset U,\, U \text{ open})
\end{equation}
It can then be proven (see \cite[Thm.~2.8]{Land}) that
\begin{equation} \lb{A.8}
C(X)=\sup(C(K)\mid K\subset X,\, K \text{ compact})
\end{equation}
for any Borel sets and that \eqref{A.7} holds for compact $X$. In particular, $C(X)=0$ if and
only if $\calE(\mu)=\infty$ for any measure $\mu$ with $\supp(\mu)\subset X$.

The key technical fact behind Theorem~\ref{T1.15A} is the following:

\begin{proposition}\lb{PA.2} If $C(X) >0$ for a Borel set $X$, there exists a probability measure,
$\mu$, supported in $X$ so that $\Phi_\mu(x)$ is continuous on $\bbC$.
\end{proposition}

\begin{remarks} 1. Let $\mu$ have finite energy so $\int \Phi_\mu(x)\, d\mu(x) <\infty$. By Lusin's
theorem (see, e.g., the remark after Theorem~6 of Appendix~A of Lax \cite{Lax} for the truly
simple proof), we can find compact sets $K\subset\supp(d\mu)$ so $\mu(K)>0$ and $\Phi_\mu \restriction
K$ is continuous.

\smallskip
2. Let $\nu=\mu\restriction K$, that is, $\nu(S)=\mu(S\cap K)$. Since $\mu(K) >0$, $\nu$ is
a nonzero measure. By general principles, both $\Phi_\nu$ and $\Phi_{\mu-\nu}$ are lower
semicontinuous on $K$, so since $\Phi_\mu$ is continuous,
\begin{equation}\lb{A.9a}
\Phi_\nu = \Phi_\mu - \Phi_{\mu-\nu}
\end{equation}
is upper semicontinuous on $K$. Thus, $\Phi_\nu$ is continuous on $K$ (since $\Phi_\mu$ is
continuous on $K$, it is bounded there, so $\Phi_\nu$ and $\Phi_{\mu-\nu}$ are both bounded
there, so there are no $\infty$--$\infty$ cancellations in \eqref{A.9a}).

\smallskip
3. By Proposition~\ref{PA.1B}, $\Phi_\nu$ is continuous on $\bbC$.
\end{remarks}

Now suppose $\mu$ is an arbitrary measure of compact support and that $C(\{x\mid \Phi_\mu(x)=
\infty\}) >0$. Then, by the above proposition, there is an $\eta$ supported on that set with
$\Phi_\eta$ continuous and so bounded above on $\supp(d\mu)$. Thus,
\begin{equation}\lb{A.9b}
\int \Phi_\eta(x)\, d\mu(x) <\infty
\end{equation}
On the other hand, $\Phi_\mu(x)=\infty$ on $\supp (d\eta)$, so
\begin{equation}\lb{A.9c}
\int \Phi_\mu(x)\, d\eta(x)=\infty
\end{equation}
This contradicts \eqref{A.2}. We thus see that the last proposition implies:

\begin{corollary}\lb{CA.2A} For any measure of compact support, $\mu$, $\{x\mid\Phi_\mu(x)=\infty\}$
has capacity zero.
\end{corollary}

A main reason for defining capacity for any Borel set is that it lets us single out sets of
capacity zero (also called {\it polar sets}), which are very thin sets (e.g., of Hausdorff
dimension zero; see Theorem~\ref{TA.16A}). We say an event (i.e., a Borel set) occurs
{\it quasi-everywhere} (q.e.) if and only if it fails on a set of capacity zero.
``Nearly everywhere" is also used. A countable union of capacity zero sets is capacity
zero. Note that if $\mu$ is any measure of compact support, with $\calE(\mu)<\infty$, then
$\calE(\mu\restriction E)<\infty$ for any compact $E$ (because $\log\abs{x-y}^{-1}$ is bounded
below) and thus, $\mu(E)=0$ if $C(E)=0$. It follows (using \eqref{A.8}) that

\begin{proposition}\lb{PA.3} If $\calE(\mu)<\infty$, then $\mu(X)=0$ for any $X$ with
$C(X)=0$.
\end{proposition}

Here is an important result showing the importance of sets of zero capacity. It is the key
to Van Assche's proof in Section~\ref{s4} and the proof of our new Theorem~\ref{T1.15A}
in Section~\ref{s7}.

\begin{theorem}\lb{TA.4} Let $\nu_n,\nu$ be measures with supports contained in
a fixed compact set $K$ and $\sup_n \nu_n(K)<\infty$. If $\nu_n\to\nu$ weakly, then
\begin{equation} \lb{A.9}
\liminf_{n\to\infty}\, \Phi_{\nu_n}(x) \geq\Phi_\nu(x)
\end{equation}
for all $x\in\bbC$ and equality holds q.e.
\end{theorem}

\begin{remarks} 1. \eqref{A.9} is called the ``Principle of Descent" and the equality q.e.\
is the ``Upper Envelope Theorem."

\smallskip
2. Suppose $\nu_n$ has a point mass of weight $\f{1}{2^n}$ at $\{\f{j}{2^n}\}_{j=0}^{2^n-1}$.
Then $d\nu_n\to dx\equiv d\nu$, Lebesgue measure. $\Phi_{\nu_n} (\f{j}{2^n})=\infty$ so
$\liminf \Phi_{\nu_n}(x)=\infty$ at any dyadic rational, while $\Phi_\nu(x)<\infty$ for
all $x$. This shows equality may not hold everywhere. This example is very relevant to
spectral theory. For the Anderson model, we expect $\limsup \abs{p_n(x)}^{1/n} =
e^{\gamma(x)}$ for almost all $x$ and $\limsup \abs{p_n(x)}^{1/n}=e^{-\gamma(x)}$ at
the eigenvalues. Thus, with $\nu_n$ the zero counting measure for $p_n$, so
$\Phi_{\nu_n}(x) =-\log \abs{p_n(x)}^{1/n}$, we have $\liminf \Phi_{\nu_n}(x)=
-\gamma(x)$ for almost all $x$ and $\gamma(x)$ at the eigenvalue consistent with
\eqref{A.9}, and with \eqref{A.9} failing on a capacity zero set, including the
countable set of eigenvalues.

\smallskip
3. \eqref{A.9} is easy. For $\Phi_\nu^M$ is the convolution with a continuous function
so $\lim_{n\to\infty} \Phi_{\nu_n}^M(x)=\Phi_\nu^M(x)$. Since $\Phi_{\nu_n}(x) \geq
\Phi_{\nu_n}^M(x)$, we see $\liminf_{n\to\infty} \Phi_{\nu_n}(x)\geq\Phi_\nu^M(x)$.
Taking $M\to\infty$ yields \eqref{A.9}.

\smallskip
4. Let $X$ be the set of $x$ for which the inequality in \eqref{A.9} is strict. Suppose $C(X)>0$.
Then, by Proposition~\ref{A.2}, there is $\eta$ supported on $X$ with $\Phi_\eta (x)$ continuous
so
\begin{equation} \lb{A.10}
\lim_{n\to\infty}\int \Phi_\eta(x)\, d\nu_n =\int \Phi_\eta(x)\, d\nu
\end{equation}
By \eqref{A.2} and Fatou's lemma ($\Phi_{\nu_n}(x)$ is uniformly bounded below),
\begin{align}
\lim_{n\to\infty} \int \Phi_\eta(x)\, d\nu_n
&= \lim\int \Phi_{\nu_n}(x)\, d\eta \notag\\
&\geq \int \liminf \Phi_{\nu_n}(x)\, d\eta \notag \\
&> \int \Phi_\nu(x)\, d\eta \lb{A.11} \\
&= \int \Phi_\eta(x)\, d\nu \notag
\end{align}
where \eqref{A.11} comes from the assumptions $\supp(d\eta)\subset X$ and \eqref{A.9} is
strict on $X$. This contradiction to \eqref{A.10} shows $C(X)=0$, that is, equality holds in
\eqref{A.9} q.e.
\end{remarks}

If $\calE_M(\mu) =\int \Phi_\mu^M\, d\mu(x)$, then it is easy to prove $\calE_M$ is weakly
continuous and conditionally positive definite in that
\begin{equation} \lb{A.12}
\mu(\bbC)=\nu(\bbC)\Rightarrow \calE_M(\mu-\nu)\geq 0
\end{equation}
where boundedness of $\log(\min(\abs{x-y}^{-1},M))$ implies $\calE_M$ makes sense for any signed
measure. By taking $M$ to infinity, one obtains

\begin{theorem}\lb{TA.5} The map $\mu\mapsto\calE(\mu)$ is weakly lower semicontinuous on
$\calM_{+,1}(E)$ for any compact $E\subset\bbC$. Moreover, it is conditionally positive
definite in the sense that for $\mu,\nu\in\calM_{+,1}(E)$, $\calE(\mu)<\infty$ and
$\calE(\nu)<\infty$ imply
\begin{equation} \lb{A.13}
\int \Phi_\nu (x)\, d\mu(x) \leq \tfrac12\, \calE(\mu) + \tfrac12\, \calE(\nu)
\end{equation}
with strict inequality if $\mu\neq\nu$.
\end{theorem}

\begin{remark} The strict inequality requires an extra argument. One can prove that if
$\mu,\nu\in\calM_{+,1}(E)$ with finite energy, then $\widehat\mu(k)-\widehat\nu(k)$ is
analytic in $k$ vanishing at $k=0$ and
\begin{equation} \lb{A.14}
\calE(\mu) +\calE(\nu)-2 \int \Phi_\nu(x)\, d\mu(x) =\f{1}{2\pi}\int
\biggl|\f{\widehat\mu(k) -\widehat\nu(k)}{k}\biggr|^2\, d^2 k
\end{equation}
\end{remark}

Since the inequality in \eqref{A.13} is strict and
\begin{equation} \lb{A.15}
\calE(\tfrac12\,\mu + \tfrac12\, \nu) =\tfrac14\, \calE(\mu) + \tfrac14\, \calE(\nu) +
\tfrac12 \int \Phi_\nu(x)\, d\mu(x)
\end{equation}
we see that $\calE(\mu)$ is strictly convex on $\calM_{+,1}(E)$, and thus

\begin{theorem}\lb{TA.6} Let $E$ be a compact subset of $\bbC$ with $C(E)>0$. Then there
exists a unique probability measure, $d\rho_E$, called the equilibrium measure for $E$,
that has \begin{equation} \lb{A.16}
\calE(\rho_E)=\log(C(E)^{-1})
\end{equation}
\end{theorem}

The properties of $\rho_E$ are summarized in

\begin{theorem}\lb{TA.7} Let $E\subset\bbC$ be compact. Let $\Omega$ be the unbounded
component of $\bbC\setminus E$ and $\wti\Omega=\bbC\setminus (\Omega\cup E)$ the union
of the bounded components of $\bbC\setminus E$. Suppose
$C(E)>0$ and $d\rho_E$ is its equilibrium measure. Then
\begin{SL}
\item[{\rm{(a)}}] For all $x\in\bbC$,
\begin{equation} \lb{A.17}
\Phi_{\rho_E}(x) \leq \log(C(E)^{-1})
\end{equation}
\item[{\rm{(b)}}] Equality holds in \eqref{A.17} q.e.\ on $E$ and on $\wti\Omega$.
\item[{\rm{(c)}}] Strict inequality holds in \eqref{A.17} on $\Omega$.
\item[{\rm{(d)}}] $\rho_E$ is supported on $\partial\Omega$, the boundary viewed as a set in $\bbC$.
\item[{\rm{(e)}}] $\Phi_{\rho_E}$ is continuous on $\bbC$ if and only if it is continuous when
restricted to $\supp(d\rho_E)$ if and only if equality holds in \eqref{A.17} on $\supp(d\rho_E)$.
\item[{\rm{(f)}}] If $I\subset E\subset\bbR$ with $I=(a,b)$, then $d\rho_E\restriction I$
is absolutely continuous with respect to Lebesgue measure, $\f{d\rho_E}{dx}\restriction I$
is real analytic, and equality holds in \eqref{A.17} on $I$.
\end{SL}
\end{theorem}

\begin{remarks} 1. For example, if $E=\partial\bbD$, $\Omega =\bbC\setminus\ol{\bbD}$
and $\wti\Omega=\bbD$.

2. See \cite{Land,Ran} for complete proofs.

\smallskip
3. (a)$+$(b) is called Frostman's theorem.

\smallskip
4. Equality in \eqref{A.17} may not hold everywhere on $E$; for example, if $E=[-1,1]\cup
\{2\}$, the equilibrium measure gives zero weight to $\{2\}$, so is the same as the
equilibrium measure for $[-1,1]$ and that $d\rho_E$ has inequality on $\bbC\setminus [-1,1]$
by (c).

\smallskip
5. If $f$ is supported on $\supp (d\rho_E)$ and $f$ bounded and Borel, and $\int f\,
d\rho_E=0$, then $(1+\veps f) d\rho_E$ is a probability measure for $\veps$ small
with $\calE((1+\veps f)\, d\rho_E)<\infty$. Since $\f{d}{d\veps}\calE((1+\veps f)
d\rho_E)= 2\int f(x) \Phi_{\rho_E}(x)\, d\rho_E(x)$, we see $\Phi_{\rho_E}(x)$ is a
constant for $d\rho_E$-a.e.\ $x$. Since $\calE(\rho_E)=\int d\rho_E\, \Phi_{\rho_E}(x)$,
the constant must be $\calE(\rho_E)=\log (C(E)^{-1})$. By lower semicontinuity,
\eqref{A.17} holds on $\supp(d\rho_E)$. Since $\Phi_{\rho_E}$ is harmonic on $\bbC
\setminus\supp(d\rho_E)$ and goes to $-\infty$ as $\abs{x}\to\infty$, \eqref{A.17}
holds by the maximum principle.

\smallskip
6. Let $\eta$ be a probability measure on $E$ with $\calE(\eta) <\infty$. Then
\begin{equation} \lb{A.18}
\calE((1-t)d\rho_E +td\eta) =\calE(d\rho_E) + t\biggl( \int \Phi_{\rho_E}(x)
[d\eta -d\rho_E]\biggr) + O(t^2)
\end{equation}
Since $\int d\rho_E \Phi_{\rho_E}(x)=\calE(d\rho_E)=\log (C(E)^{-1})$, if $\eta$ is
supported on a set where strict inequality holds in \eqref{A.17}, $\calE((1-t)d\rho_E
+td\eta) <\calE(d\rho_E)$ for small $t$, violating minimality. Thus the set where
\eqref{A.17} has inequality cannot support a measure of finite energy, that is,
it has zero capacity, proving (b).

\smallskip
7. Since $\Phi_{\rho_E}$ is harmonic on $\Omega$ and goes to $-\infty$ at $\infty$,
the maximum principle implies $\Phi_{\rho_E}(x)$ cannot take its maximum
(which is $\log (C(E)^{-1})$) on $\Omega$. (e) follows from Proposition~\ref{PA.1B}.
(d) is left to the references; see \cite{Land,Ran}.

\smallskip
8. If $I\subset E\subset\bbR$, one first shows equality holds in \eqref{A.17} on $I$
and that $\Phi$ is continuous there. (This uses the theory of ``barriers"; see
\cite{Land,Ran}. One can also prove this using periodic Jacobi matrices and approximations;
see \cite{Rice}). Then one can apply the reflection principle to see that $\Phi_{\rho_E}$
has a harmonic continuation across $I$. Indeed, $\Phi_{\rho_E}$ is then the real part of
a function analytic on $I$ with zero derivative there. That derivative for $\Ima z >0$ is
the real part of
\begin{equation} \lb{A.19}
F(z) =\int \f{d\rho_E(x)}{x-z}
\end{equation}
so, by the standard theory of boundary values of Herglotz functions (see
\cite[Sect.~1.3]{OPUC1}), we have that $d\rho_E\restriction I$ is absolutely
continuous and
\begin{equation} \lb{A.20}
\f{d\rho_E}{dx} = \f{1}{\pi}\, \Ima F(x+i0)
\end{equation}
proving real analyticity of this derivative.

\smallskip
9. The same argument as in Remark~8 applies if $I$ is replaced by an analytic arc with
a neighborhood $N$ obeying $N\cap E=I$. In particular, if $I$ is an ``interval" in
$\partial\bbD$ and $I \subset E \subset\partial\bbD$, we have absolute continuity and
analyticity on $I$.
\end{remarks}

Here is an interesting consequence of \eqref{A.2}:

\begin{theorem}\lb{TA.7A} Let $\nu$ be a measure of compact support, $E$, so that
$C(E) >0$. Then
\begin{equation} \lb{A.22x}
\Phi_\nu(x) <\infty \qquad\text{for $d\rho_E$ a.e.\ $x$}
\end{equation}
\end{theorem}

\begin{remarks} 1. This can happen even if $\calE(\nu)=\infty$ so $\int \Phi_\nu(x)\,
d\nu(x)=\infty$.

\smallskip
2. For \eqref{A.17} implies
\[
\int \Phi_{\rho_E}(x)\, d\nu \leq \log (C(E)^{-1})\nu(E) <\infty
\]
so \eqref{A.2} implies
\begin{equation} \lb{A.22y}
\int \Phi_\nu(x)\, d\rho_E(x) <\infty
\end{equation}
\end{remarks}

The following illustrates the connection between potential theory and polynomials:

\begin{theorem}[Bernstein--Walsh Lemma]\lb{TA.7B} Let $E$ be a compact set in $\bbC$
with $C(E) >0$ and let $\Omega$ be the unbounded component of $\bbC\setminus E$. Let
$p_n$ be a polynomial of degree $n$ and let
\begin{equation} \lb{A.21a}
\|p_n\|_E=\sup_{z\in E}\, \abs{p_n(z)}
\end{equation}
Then for all $z\in\Omega$,
\begin{equation} \lb{A.21b}
\abs{p_n(z)} \leq C(E)^{-n} \|p_n\|_E [\exp (-n\Phi_{\rho_E}(z))]
\end{equation}
\end{theorem}

\begin{remarks} 1. This is named after Bernstein and Walsh \cite{Walsh}, although
the result appears essentially in Szeg\H{o} \cite{Sz24}.

\smallskip
2. Let $\{z_j\}_{j=1}^n$ be the zeros of $p_n$. Define
\begin{equation} \lb{A.21c}
g(z)=\log\abs{p_n(z)} + n\Phi_{\rho_E}(z) + n\log (C(E))
\end{equation}
on $\Omega\cup\{\infty\}\setminus\{z_j\}_{j=1}^n=\Omega'$. $g$ is harmonic on $\Omega'$
including at $\infty$ since both $\log\abs{p_n(z)}$ and $-n\Phi_{\rho_E} (z)$ are
$n\log\abs{z}$ plus harmonic near $\infty$. Since $g_n(z)\to-\infty$ at the $z_j\in\Omega$,
we see
\begin{equation} \lb{A.21d}
\sup_{z\in\Omega'}\, \abs{g(z)} \leq\lim_{\delta\downarrow 0}\,
\biggl[\sup_{\substack{\dist (w,E)=\delta \\ w\in\Omega }} \abs{g(w)}\biggr]
\end{equation}
But, by \eqref{A.17}, $g(z)\leq \log \abs{p_n(z)}$, so
\[
g(z)\leq \log\|p_n\|_E
\]
which is \eqref{A.21b} on $\Omega'\setminus\{\infty\}$. \eqref{A.21b} holds trivially
at the $z_j$, completing the proof.
\end{remarks}

Following ideas of Craig \cite{Craig}, one can say much more about $\f{d\rho_E}{dx}$
when $E$ contains an isolated closed interval:

\begin{theorem}\lb{TA.8} Let $E\subset\bbR$ be compact and $a<c<d<b$ so $E\cap (a,b)=[c,d]$.
Then there exists $g$ real, real analytic, and strictly positive on $[c,d]$ so that
\begin{equation} \lb{A.21}
d\rho_E \restriction [c,d] = g(x) [(d-x)(x-c)]^{-1/2}\, dx
\end{equation}
If $E=\cup_{j=1}^{\ell+1} [a_j,b_j]$ with $a_1 <b_1 < a_2 < \cdots < b_{\ell+1}$,
there are $x_j\in (b_j, a_{j+1})$ for $j=1,2,\dots, \ell$ so that
\begin{equation} \lb{A.22}
d\rho_E(x) = \f{1}{\pi} \biggl[\, \prod_{j=1}^\ell \f{x-x_j}{\sqrt{(x-b_j)(x-a_{j+1})}}\biggr]
\f{1}{\sqrt{(x-a_1)(b_{\ell+1}-x)}}\, dx
\end{equation}
\end{theorem}

\begin{remarks} 1. \eqref{A.22} is from Craig \cite{Craig}.

\smallskip
2. The idea behind the proof is simple. One lets $F(z)=\int \f{d\rho_E(x)}{x-z}$. By the
arguments above, $F$ is pure imaginary on $[c,d]$ as the derivative of $\Phi_{\rho_E}(x)$.
Thus, $\arg F(x+i0)$ is $\f{\pi}{2}$ on $[c,d]$, and by a simple argument, $0$ on
$[c-\delta, c)$ and $\pi$ on $(d,d+\delta]$. A Herglotz representation for $\log F(x+i0)$
yields \eqref{A.21} and \eqref{A.22}.

\smallskip
3. The $x_j$'s are uniquely determined by
\begin{equation} \lb{A.23}
\int_{b_j}^{a_{j+1}} F(x)\, dx =0
\end{equation}
\end{remarks}

Recall a set $S$ is called {\it perfect\/} if it is closed and has no isolated points.
A standard argument shows that any compact $E$ has a unique decomposition into disjoint
sets, $D\cup S$ where $D$ is a countable set and $S$ is perfect (similarly, any compact
$E\subset\bbR$ can be written $Z\cup F$ where $Z$ has Lebesgue measure zero and $F$ is
{\it essentially perfect}, that is, $\abs{F\cap (x-\delta,x+\delta)} >0$
for any $x\in F$ and $\delta >0$).

Similarly, we call a set $P$ {\it potentially perfect} (the terminology is new) if $P$
is closed and $C(P\cap\{x\mid\abs{x-x_0} <\delta\}) >0$ for all $x_0\in P$ and $\delta >0$.
It is easy to see that any compact $E\subset\bbC$ can be uniquely written as a disjoint
union $E=Q\cup P$ where $C(Q)=0$ and $P$ is potentially perfect.

These notions are related to equilibrium measures. If $\capa(E)>0$ and $E=Q\cup P$ is this
decomposition, then
\begin{equation} \lb{A.28}
P=\supp(d\rho_E)
\end{equation}
In particular, $\supp(d\rho_E)=E$ if and only if $E$ is potentially perfect.

Just as one writes $\sigma(d\mu)=\sigma_\disc (d\mu)\cup\sigma_\ess (d\mu)$, we
single out the potentially perfect part of $\sigma(d\mu)$ and call it $\sigma_\capa (d\mu)$.

Next, we want to state a kind of converse to Frostman's theorem.

\begin{theorem}\lb{TA.9} Let $E\subset\bbC$ be compact. Suppose $E$ is potentially perfect. Let
$\eta\in\calM_{+,1}(E)$ be a probability measure on $E$ with $\supp(d\eta)\subseteq  E$ so that
for some constant, $\alpha$,
\begin{equation} \lb{A.24}
\Phi_\eta(x)=\alpha \qquad d\rho_E\text{-a.e. } x
\end{equation}
Then $\eta=\rho_E$, the equilibrium measure, and $\alpha =\log(C(E)^{-1})$.
\end{theorem}

\begin{remark} By lower semicontinuity, $\Phi_\eta(x)\leq\alpha$ on $\supp(d\rho_E)=E$ by
hypothesis. Thus,
\begin{equation} \lb{A.25}
\calE(\eta) =\int \Phi_\eta (x)\, d\eta(x) \leq \alpha <\infty
\end{equation}
so $\eta$ must give zero weight to zero capacity sets. Thus, $\Phi_{\rho_E}(x)=\log(C(E)^{-1})$
for $d\eta$-a.e.\ $x$ and thus,
\begin{equation} \lb{A.26}
\int \Phi_{\rho_E}(x)\, d\eta(x) =\log(C(E)^{-1})
\end{equation}
By \eqref{A.2} and \eqref{A.24},
\[
\text{LHS of \eqref{A.26}}=\int \Phi_\eta(x)\, d\rho_E(x)=\alpha
\]
Thus, $\alpha =\log(C(E)^{-1})$, and by \eqref{A.25} and uniqueness of minimizers, $\eta=\rho_E$.
\end{remark}

Next, we note that the {\it Green's function\/} for a compact $E\subset\bbC$ is defined by
\begin{equation}\lb{A.26a}
G_E(z)=-\Phi_{\rho_E}(z) + \log(C(E)^{-1})
\end{equation}
It is harmonic on $\bbC\setminus E$, $G_E(z)-\log\abs{z}$ is harmonic at infinity, and $G_E(z)$
has zero boundary values q.e.\ on $E$. Notice that $G_E(z)\geq 0$ on $\bbC$. If
\begin{equation} \lb{A.27}
\lim_{z\to E}\, G_E(z)=0
\end{equation}
in the sense that
\begin{equation} \lb{A.28x}
\lim_{\delta\downarrow 0}\, \sup_{\dist(z,E)<\delta} G_E(z)=0
\end{equation}
we say $E$ is {\it regular for the Dirichlet problem\/} (just called regular). By Theorem~\ref{TA.7}(c),
this is true if and only if $\Phi_{\rho_E}(x)=\log (C(E)^{-1})$ for all $x\in E$. By
Theorem~\ref{TA.8}, this is true for finite unions of disjoint closed intervals.

Notice that the Bernstein--Walsh lemma \eqref{A.21b} can be rewritten
\begin{equation} \lb{A.29}
\abs{p_n(z)}\leq \|p_n\|_E \exp (nG_E(z))
\end{equation}

Closely related are comparison theorems and limit theorems. We will state them for subsets
of $\bbR$:

\begin{theorem}\lb{TA.10A} Let $E_1\subset E_2\subset\bbR$ be compact sets. Then
\begin{SL}
\item[{\rm{(i)}}] $C(E_1)\leq C(E_2)$
\item[{\rm{(ii)}}]
\begin{equation}\lb{A.35x}
G_{E_2}(z)\leq G_{E_1}(z)
\end{equation}
for all $z\in\bbC$
\item[{\rm{(iii)}}]
\begin{equation}\lb{A.35a}
d\rho_{E_2}\restriction E_1 \leq d\rho_{E_1}
\end{equation}
\item[{\rm{(iv)}}] If $I=(a,b)\subset E_1$, then on $I$,
\begin{equation}\lb{A.35b}
\f{d\rho_{E_2}}{dx} \leq \f{d\rho_{E_1}}{dx}
\end{equation}
for all $x\in I$.
\end{SL}
\end{theorem}

\begin{theorem}\lb{TA.10B} Let $E_1 \supset E_2 \supset \dots$ be compact subsets of $\bbR$. Let
$E_\infty =\cap_{j=1}^\infty E_j$. Then
\begin{SL}
\item[{\rm{(i)}}]
\begin{equation}\lb{A.35c}
\lim_{n\to\infty}\, C(E_n) =C(E_\infty)
\end{equation}
\item[{\rm{(ii)}}] $\rho_{E_n}\to \rho_{E_\infty}$ weakly
\item[{\rm{(iii)}}] For $z\in\bbC\setminus E_\infty$,
\begin{equation}\lb{A.35d}
\lim_{n\to\infty}\, G_{E_n}(z) =G_{E_\infty} (z)
\end{equation}
and \eqref{A.35d} holds q.e.\ on $E_\infty$.
\item[{\rm{(iv)}}] If $I=(a,b)\subset E_\infty$, then uniformly on compact subsets of $I$,
\begin{equation}\lb{A.35g}
\lim_{n\to\infty}\, \f{d\rho_{E_n}}{dx} = \f{d\rho_{E_\infty}}{dx}
\end{equation}
\end{SL}
\end{theorem}

\begin{remarks} 1. \eqref{A.35a} has the pleasing physical interpretation that if one conductor is
connected to another, charge leaks out in a way that there is less charge everywhere in the original
conductor.

\smallskip
2. Part (i) of each theorem is easy. For Theorem~\ref{TA.10A}, it follows from the minimum energy
definition. For Theorem~\ref{TA.10B}(i), we note that if $U$ is open with $E_\infty \subset U$, then
eventually $E_n\subset U$, so \eqref{A.7} implies \eqref{A.35c}.

\smallskip
3. One proves (ii)--(iv) of Theorem~\ref{TA.10A} first for $E$, a finite union of closed intervals,
then proves Theorem~\ref{TA.10B}, and then for general compact $E_\infty \subset\bbR$ defines
$E_n=\{x\mid\dist(x, E_\infty)\leq\f{1}{n}\}$ and proves $\cap_n E_n=E_\infty$ and
each $E_n$ is a finite union of closed intervals. Theorem~\ref{TA.10B} then yields
Theorem~\ref{TA.10A} for general $E$'s (see \cite{Rice}).

\smallskip
4. For $E_1, E_2$ finite union of closed intervals and $z\notin E_2$, one gets \eqref{A.35x} by noting
the difference $G_{E_1}(z)-G_{E_2}(z)$ is harmonic on $\bbC\setminus E_2$, zero on $E_1$, and
positive on $E_2\setminus E_1$, where $G_{E_1} >0$ and $G_{E_2}=0$. The inequality for $z\in E_2$
then follows from the fact that any subharmonic function $h$ obeys
\begin{equation}\lb{A.35e}
h(z_0)=\lim_{r\downarrow 0}\, \f{1}{2\pi r} \int_0^{2\pi} h(z_0 +re^{i\theta})\, d\theta
\end{equation}

\smallskip
5. For this case, one gets \eqref{A.35a}/\eqref{A.35b} by noting that \eqref{A.20} can be rewritten
\begin{equation}\lb{A.35f}
\f{d\rho_E(x_0)}{dx} = \f{1}{\pi}\, \lim_{\veps\downarrow 0}\, \veps^{-1} G_E (x_0+i\veps)
\end{equation}
for $x_0\in E$ and using \eqref{A.35x} for $x\in E_1$.

\smallskip
6. If $\{E_n\}_{n=1}^\infty$, $E_\infty$ are as in Theorem~\ref{TA.10B} and $d\eta$ is a weak limit
point of $d\rho_{E_n}$, then $\eta$ is supported on $E_\infty$, and by lower semicontinuity of the
Coulomb energy $\calE$,
\begin{align*}
\calE(\eta) &\leq \lim \calE(\rho_{E_n}) \\
&= \lim\log (C(E_n)^{-1}) \\
&= \log (C(E_\infty)^{-1})
\end{align*}
by \eqref{A.35c}, so $\eta=\rho_{E_\infty}$, that is, $\rho_{E_n}\to \rho_{E_\infty}$ weakly.
\eqref{A.35d} then follows for $z\notin E_\infty$ from \eqref{A.35c} and continuity of
$\Phi_\nu(z)$ in $\nu$ for $z\notin\supp(d\nu)$. \eqref{A.35e} implies convergence for
$z\in E_\infty$.

\smallskip
7. \eqref{A.35g} follows from $\rho_{E_n}\to \rho_{E_\infty}$ and uniform bounds on derivatives of
$\f{d\rho}{dx}$ on $I$, which in turn follow from the proof of \eqref{A.21}.
\end{remarks}

\begin{example}\lb{EA.10} Harmonic functions are conformally invariant, which means (since
Green's functions are normalized by $G_E(z)=\log\abs{z} + O(1)$ near infinity and boundary
values of $0$ on $E$), if $Q$ is an analytic bijection of $\bbC\setminus\ol{\bbD}\cup\{\infty\}$
to $\Omega\cup\{\infty\}$ with $Q(z)=Cz+O(1)$ near infinity, then, since $\log\abs{z}$ is the
Green's function for $E=\partial\bbD$,  $\log\abs{Q^{-1}(z)}$ is the Green's function for $E$
and $C$ its capacity. In particular, with $Q(z)=z+\f{1}{z}$, we see
\begin{equation} \lb{A.30}
C([-2,2])=1
\end{equation}
and consistent with \eqref{A.22}
\begin{equation} \lb{A.31}
d\rho_{[-2,2]}(x) =\f{1}{\pi}\, \f{dx}{\sqrt{4-x^2}\,}
\end{equation}
Notice that, by scaling, if $\lambda >0$ and $\lambda E=\{\lambda z\mid z\in E\}$ and
$\mu\in\calM_{+,1}(E)$ is mapped to $\mu_\lambda$ in $\calM_{+,1}(\lambda E)$ by scaling,
then
\begin{equation} \lb{A.32}
\calE(\mu_\lambda) =-\log(\lambda) + \calE(\mu)
\end{equation}
so
\begin{equation} \lb{A.33}
C(\lambda E) =\lambda C(E)
\end{equation}
This plus translation invariance shows
\[
C([a,b]) =\tfrac14\, (b-a)
\]

Now, let $E\subset\bbR$, $E_- =E\cap (-\infty, 0)$, $E_+ =E\setminus E_-$, and for $a>0$,
let $E^{(a)}=E_-\cup (E_+ +a)$. Let $d\rho_E$ be equilibrium measure for $E$. Since $\log
\abs{x+a-y}^{-1} < \log\abs{x-y}^{-1}$ for $x>0$, $y<0$, $a>0$, we see
\begin{equation} \lb{A.34}
\calE(\rho\restriction E_- + T_a (\rho\restriction E_+)) <\calE(\rho)
\end{equation}
(where $T_a \mu$ is the translate of $\mu$). Thus
\[
C(E^{(a)}) \geq C(E)
\]
This is an expression of the repulsive nature of the Coulomb force! Thus, by joining together
all the pieces of $E$ (via a limiting argument), one sees that if $\abs{E}$ is the Lebesgue
measure of $E\subset\bbR$, then
\begin{equation} \lb{A.35}
C(E)\geq\tfrac14\, \abs{E}
\end{equation}
and sets of capacity zero have Lebesgue measure zero.
\qed
\end{example}

\begin{example}\lb{EA.11} Let $d\mu$ be the conventional Cantor measure on $[0,1]$ which can be
thought of as writing $x=\sum\f{a_n(x)}{3^n}$ with $a_n =0,1$ or $2$ and taking $d\mu$ as the
infinite product of measures given weight $\f12$ to $a_n=0$ or $2$. Looking at $a_1$, we get the
usual two pieces of mass $\f12$ with minimum distance $\f13$ between them. Look at $a_1,\dots,a_k$
and we have $2^k$ pieces of mass $2^{-k}$ and minimum distance $3^{-k}$. Given $x,y$ in the
Cantor set, $\dist\abs{x-y} <3^{-k}$ if and only if they are in the same pieces, that is,
\[
\mu (\{x\mid\abs{x-y} <3^{-k}\}) =2^{-k}
\]
Thus
\begin{align*}
\int \log\abs{x-y}^{-1}\, d\mu(x)\, d\mu(y)
&\leq \sum_k \, (k+1) (\log 3)\, \mu (\{x,y\mid \abs{x-y} <3^{-k}\}) \\
&= \sum_k \, (k+1)(\log 3)\, 2^{-k} <\infty
\end{align*}
This shows the Cantor set has positive capacity. Generalizing, we get sets of any Hausdorff
dimension $\alpha >0$ with positive capacity. In fact, as we will see shortly, any set of
positive Hausdorff dimension has positive capacity.
\qed
\end{example}

\begin{example}\lb{EA.12} Fix $a>0$ and let $E=(-\f{a}{2} -\Delta, -\f{a}{2})\cup (\f{a}{2},
\f{a}{2} +\Delta)$ where $\Delta =\f{4}{a}$. When $a$ is very large, the equilibrium measure
is very close to the average of the equilibrium measure for the two individual intervals.
This measure has energy approximately
\[
2\, \f{1}{4} \, \log \biggl( \f{4}{\Delta}\biggr) + 2\, \f14\, \log(a) =0
\]
so the asymptotic capacity is $1$. This phenomenon of distant pieces of individually small
capacity having total capacity bounded away from zero is a two-dimensional phenomenon.
\qed
\end{example}

Sets of capacity zero not only have zero Lebesgue measure, but they also have zero
$\alpha$-dimensional Hausdorff measure for any $\alpha >0$:

\begin{theorem}\lb{TA.16A} Any compact set $E$ of capacity zero has zero Hausdorff dimension.
\end{theorem}

\begin{remarks} 1. We will sketch a proof where $E\subset\bbR$. What one needs to do, for any
$\alpha >0$, $\veps >0$, is to find a cover of $E$ by intervals $I_1,\dots, I_n\dots$ of length
$\abs{I_j}$ so that
\begin{equation}\lb{A.52a}
\sum\, \abs{I_j}^\alpha <\veps
\end{equation}

\smallskip
2. We begin by noting that there is a measure $\mu$ (not necessarily supported by $E$) so that
$\Phi_\mu(x)=\infty$ for all $x\in E$ (we do not care that $E$ is exactly the set where
$\Phi_\mu=\infty$ but note that by combining the ideas here with Corollary~\ref{CA.2A},
one can show $E$ is the set where some potential is infinite if and only if $E$ is a
$G_\delta$-set of zero capacity). Here is how to construct $\mu$. Let $E_m =\{x\in\bbR
\mid\dist(x,E)\leq \f{1}{m}\}$. $E_m$ is a finite union of closed intervals and, by
\eqref{A.7}, $C(E_m)\downarrow 0$. Pass to a subsequence $\ti E_m$, so $C(\ti E_m)\leq
\exp (-\f{1}{m^2})$ so $\Phi_{\rho_{\ti E_m}}(x)\geq m^2$ on $\ti E_m$ and so on $E$.
Let $\mu =\sum_m m^{-2} \rho_{\ti E_m}$. $\mu$ is a finite measure with $\Phi_\mu=\infty$
on $E$.

\smallskip
3. Let $x\in E$. Suppose for some $\alpha >0$ and $c>0$, we have with $I_r^x =(x-r, x+r)$,
\begin{equation}\lb{A.52b}
\mu (I_r^x) \leq c(2r)^\alpha
\end{equation}
Then picking $r=2^{-n}$, we see (with $n_0$ large and negative so $\supp(d\mu)\subset
I_{2^{-n}}^x$)
\begin{align*}
\int \log \abs{y-x}^{-1}\, d\mu(y) &\leq \sum_{n=n_0}^\infty [(n+1)\log 2]
\mu (I_{2^{-n}}^x) \\
&<\infty
\end{align*}
We thus conclude \eqref{A.52b} always fails, that is, for any $x\in E$ and any $\alpha$,
\begin{equation}\lb{A.52c}
\limsup_{r\downarrow 0}\, (2r)^{-\alpha} \mu (I_r^x) =\infty
\end{equation}

\smallskip
4. Given $\alpha >0$, $\delta >0$ fixed, by \eqref{A.52c}, we can find for each $x\in E$,
so $I_{r_x}^x$ with
\begin{equation}\lb{A.52d}
\mu (I_{r_x}^x) \geq \delta^{-1} (2r_x)^\alpha
\end{equation}

\smallskip
5. There is standard covering lemma used in the proof of the Hardy--Littlewood maximal
theorem (see \cite{Katz}, p.~74, the proof of the lemma) that we can find a suitable
sequence $x_j$ with
\begin{equation}\lb{A.52e}
I_{r_{x_j}}^{x_j} \cap I_{r_{x_k}}^{x_k}=\emptyset
\end{equation}
and
\begin{equation}\lb{A.52f}
\bigcup_x I_{r_x}^x \subset \bigcup_{j=1}^\infty I_{4r_{x_j}}^{x_j}
\end{equation}

\smallskip
6. Thus, $\{I_{4r_{x_j}}^{x_j}\}$ cover $E$ and, by \eqref{A.52d},
\begin{align*}
\sum_j\, \abs{I_{4r_{x_j}}^{x_j}}^\alpha
&\leq 4^\alpha \sum_j \, \abs{I_{r_{x_j}}^{x_j}}^\alpha \\
&\leq 4^\alpha \delta \sum_j \mu (I_{r_{x_j}}^{x_j}) \\
&\leq 4^\alpha \delta \mu(\bbR)
\end{align*}
by \eqref{A.52e}. Since $\delta$ is arbitrary, we have the required covers to see $\dim (E)=0$.
\end{remarks}

A final comparison result will be needed in Appendix~B:

\begin{theorem}\lb{TA.11} Let $\mu,\nu$ be two probability measures on $\bbR$ so that for all
$z$ near infinity,
\begin{equation} \lb{A.64}
\Phi_\mu(z)\geq\Phi_\nu(z)
\end{equation}
Then $\mu=\nu$. In particular, if either $\Phi_\mu(z)\geq \Phi_{\rho_E}(z)$ or $\Phi_\mu(z)
\leq \Phi_{\rho_E}(z)$ for all $z$ near infinity, then $\mu=\rho_E$.
\end{theorem}

\begin{remark}
\begin{align*}
\Phi_\mu(z) + \log\abs{z} &= -\Real \int \log (1-\tfrac{w}{z})\, d\mu(w) \\
&= \Real \biggl[\, \sum_{n=1}^\infty z^{-n} \int w^n\, d\mu(w)\biggr]
\end{align*}
Thus $\ti\Phi_\mu(z)\equiv \Phi_\mu(z) +\log\abs{z}$ is harmonic near infinity with
$\ti\Phi_\mu(\infty) =0$. Thus, $\Phi_\mu - \Phi_\nu =\ti\Phi_\mu -\ti\Phi_\nu$ is
harmonic and vanishing at $\infty$. The only way it can have a definite sign near
infinity is if it is identically $0$. By harmonicity off $\bbR$, $\Phi_\mu=\Phi_\nu$
on $\bbC\setminus\bbR$ and then, by \eqref{A.35e}, on $\bbR$. Thus, $\Phi_\mu=
\Phi_\nu$ as distributions. Since $-\Delta\Phi_\mu =2\pi\mu$, we see $\mu=\nu$.
\end{remark}

\section*{Appendix B: Chebyshev Polynomials, Fekete Sets, and Capacity} \lb{AppB}
\renewcommand{\theequation}{B.\arabic{equation}}
\renewcommand{\thetheorem}{B.\arabic{theorem}}
\setcounter{theorem}{0}
\setcounter{equation}{0}

For further discussion of the issues in this appendix, see Andrievskii--Blatt \cite{AB},
Goluzin \cite{Go}, and Saff--Totik \cite{ST} whose discussion overlaps ours here. Let
$E\subset\bbC$ be compact and infinite. The Chebyshev polynomials, $T_n(x)$, are defined
as those monic polynomials of degree $n$ which minimize
\begin{equation} \lb{B.1}
\|T_n\|_E =\sup_{z\in E}\, \abs{T_n(z)}
\end{equation}
By $T_n^R$, the restricted Chebyshev polynomials, we mean the monic polynomials, all of
whose zeros lie in $E$, which minimize $\|\cdot\|_E$ among all such polynomials. They can
be distinct: for example, if $E=\partial\bbD$, $T_n(z)=z^n$ while $T_n^R(z)=1+z^n$ (not
unique). It can be proven (see \cite[Thm.\ III.23]{Tsu}) that Chebyshev (but not restricted
Chebyshev) polynomials are unique.

Clearly,
\begin{equation} \lb{B.2}
\|T_n\|_E \leq \|T_n^R\|_E
\end{equation}
and since $T_nT_m$ is a monic polynomial of degree $n+m$,
\begin{equation} \lb{B.3x}
\|T_{n+m}\|_E \leq \|T_n\|_E \|T_m\|_E
\end{equation}
so $\lim_{n\to\infty} \|T_n\|_E^{1/n}$ exists, and similarly, so does $\lim_{n\to\infty} \|T_n^R\|_E^{1/n}$.

An {\it $n$ point Fekete set\/} is a set $\{z_j\}_{j=1}^n\subset E$ that maximizes
\begin{equation} \lb{B.3}
q_n(z_1,\dots, z_n) = \prod_{i\neq j} \abs{z_i-z_j}
\end{equation}
There are $n(n-1)$ terms in the product and the Fekete constant is defined by
\begin{equation} \lb{B.4x}
\zeta_n(E)=q_n (z_1, \dots, z_n)^{1/n(n-1)}
\end{equation}
for the maximizing set. The set may not be unique: for example, if $E=\partial\bbD$ and $\omega_n$
is an $n$th root of unity, $\{z_k=z_0 \omega_n^k\}$ is a minimizer for any $z_0\in\partial\bbD$.

Let $z_1, \dots, z_{n+1}$ be an $n+1$-point Fekete set. For each $j$,
\begin{equation} \lb{B.5}
\prod_{\substack{k,\ell\neq j\\ \ell\neq k}}\, \abs{z_k-z_\ell} \leq \zeta_n^{n(n-1)}
\end{equation}
Thus, taking the product over the $n+1$ values of $j$ and noting that each $z_k-z_\ell$ occurs
$n-1$ times,
\begin{equation} \lb{B.6}
[\zeta_{n+1}^{(n+1)n}]^{n-1} \leq [\zeta_n^{n(n-1)}]^{n+1}
\end{equation}
so $\zeta_n$ is monotone decreasing. Thus $\zeta_n$ has a limit, called the {\it transfinite diameter}.
The main theorem relating these notions and capacity is

\begin{theorem}\lb{TB.1} For any compact set $E\subset\bbC$, we have
\begin{equation} \lb{B.7}
C(E) \leq \|T_n\|_E^{1/n} \leq\|T_n^R\|_E^{1/n} \leq \zeta_{n+1}(E)
\end{equation}
Moreover,
\begin{equation} \lb{B.8}
\lim_{n\to\infty}\, \zeta_n (E) \leq C(E)
\end{equation}
{\rm{(}}so all limits equal $C(E)${\rm{)}}. Finally, if $C(E) >0$,
\begin{SL}
\item[{\rm{(i)}}] The normalized density of Fekete sets converges to $d\rho_E$, the equilibrium
measure for $E$.
\item[{\rm{(ii)}}] If $E\subset\bbR$, the normalized zero counting measure for $T_n$ and
for $T_n^R$ converges to $d\rho_E$.
\end{SL}
\end{theorem}

\begin{remarks} 1. Normalized densities and zero counting measure are the point measures that give
weight $k/n$ to a point in the set of multiplicity $k$ (for Fekete sets, $k=1$, but for polynomials
there can be zeros of multiplicity $k >1$).

\smallskip
2. If $E=\partial\bbD$, $T_n(z) =z^n$, so (ii) fails for $T_n$. If $E=\bbD$, $T_n^R(z)=z^n$ and
(ii) fails for $T_n^R$ also. It can be shown that if $E\subset\partial\bbD$, $E\neq\partial\bbD$,
(ii) also holds.

\smallskip
3. Fekete sets have the interpretation of sets minimizing the point Coulomb energy $\sum_{j\neq k}
\log\abs{z_j-z_k}^{-1}$. Parts of this theorem can be interpreted as saying the point minimizer and
associated energy without self-energies converge to the minimizing continuum distribution and
energy, which is physically pleasing!

\smallskip
4. The equality of $\lim\zeta_n$ and $\lim\|T_n\|^{1/n}$ is due to Fekete \cite{Fek}. The
rest is due to Szeg\H{o} \cite{Sz24}, whose proof we partly follow.

\smallskip
5. Stieltjes \cite{Stie} considered what we call Fekete sets for $E=[-1,1]$, proving that, in that
case, the set is unique and consists of $1$, $-1$, and the $n-2$ zeros of a suitable Jacobi polynomial
(see \cite{Szb}). The general set up is due to Fekete \cite{Fek}.

\smallskip
6. When $E\subset\partial\bbD$, there are two other sets of polynomials related to minimizing
$\|P_n\|_{\infty,E}$. We can restrict to either
\begin{SL}
\item[(a)] ``Quasi-real" monic polynomials, that is, degree $n$ polynomials, so for some $\varphi$,
$e^{-i\varphi} e^{-in\theta/2} P_n (e^{i\theta})$ is real for $\theta$ real (these are exactly polynomials
for which $P_n^*(z)=e^{-2i\varphi} P_n(z)$ where $\,^*\,$ is the Szeg\H{o} dual). Equivalently, zeros
are symmetric about $\partial\bbD$.
\item[(b)] Monic $P_n$ all of whose zeros lie on $\partial\bbD$. These Chebyshev-like polynomials
are used in \cite{Sim-wk}.
\end{SL}

Since there are classes of polynomials between all monic and monic with zeros on $E$, the $n$th
roots of the norms also converge to $C(E)$.
\end{remarks}

$\|T_n\|_E^{1/n}\leq \|T_n^R\|_E^{1/n}$ is \eqref{B.2}. Here is the last inequality in \eqref{B.7}:

\begin{proposition}\lb{PB.2}
\begin{equation} \lb{B.9}
\|T_n^R\|_E \leq \zeta_{n+1} (E)^n
\end{equation}
\end{proposition}

\begin{proof} Let $\{z_j\}_{j=1}^{n+1}$ be an $(n+1)$-Fekete set. Let
\begin{equation} \lb{B.10}
P_k(z) =\prod_{j\neq k} (z-z_j)
\end{equation}
called a Fekete polynomial. ({\it Note}: There is a different set of polynomials occurring
in a different context also called Fekete polynomials.) By the maximizing property of Fekete sets,
\begin{equation} \lb{B.11}
\|P_k\|_E =\prod_{j\neq k} \, \abs{z_k-z_j}
\end{equation}
since if $z'_j=z_j$ ($j\neq k$), $z'_k=z$, then $\prod_{\ell\neq k} \abs{z'_\ell -z'_k}
\leq \prod_{\ell\neq k} \abs{z_\ell-z_k}$. Since $\|T_n^R\|_E \leq \|P_k\|_E$ (by the
minimizing property of $\|T_n^R\|_E$), taking the $n+1$ choices of $k$,
\[
\|T_n^R\|_E^{n+1} \leq \prod_{k=1}^{n+1} \|P_k\|_E =\prod_{\text{all } j\neq k} \abs{z_k-z_j}
=\zeta_{n+1}^{n(n+1)}
\]
which is \eqref{B.9}.
\end{proof}

The following completes the proof of \eqref{B.7}:

\begin{proposition}\lb{PB.3} For any monic polynomial $P_n(z)$,
\begin{equation} \lb{B.12}
\|P_n\|_E \geq C(E)^n
\end{equation}
\end{proposition}

\begin{proof} There is nothing to prove if $C(E)=0$, so suppose $C(E) >0$. By the
Bernstein--Walsh lemma \eqref{A.21b},
\begin{equation}\lb{B.14}
\abs{P_n(z)}\leq \|P_n\|_E C(E)^{-n} \exp (-n \Phi_{\rho_E}(z))
\end{equation}
Divide by $\abs{z}^n$ and take $z\to\infty$. The left side of \eqref{B.14} goes to $1$. Since
$\Phi_{\rho_E}(z)=-\log\abs{z} + o(1)$, the right side goes to $\|P_n\|_E C(E)^{-n}$.
\end{proof}

Next we turn to the convergence of Fekete set counting measures to $d\rho_E$.

\begin{proposition} \lb{PB.4} Let $d\nu_n$ be finite point probability measures supported
at $\{z_j^{(n)}\}_{j=1}^{N_n}$ with weight $\nu_{n,j} =\nu (\{z_j^{(n)}\})$. Suppose $d\nu_n
\to d\eta$ weakly for some measure $\eta$. Suppose there is a compact $K\subset\bbC$ containing
all the $\{z_j^{(n)}\}$ and that as $n\to\infty$,
\begin{equation} \lb{B.22}
\sum_j \nu_{n,j}^2 \to 0
\end{equation}
Then
\begin{equation} \lb{B.23}
\limsup_{n\to\infty} \prod_{j\neq k} \abs{z_j^{(n)} -z_k^{(n)}}^{\nu_{n,j}\nu_{n,k}} \leq
\exp\biggl( \int d\eta(z) d\eta(w) \log\abs{z-w}\biggr)
\end{equation}
\end{proposition}

\begin{remark} Since $\sum_j \nu_{n,j}=1$,
\begin{equation} \lb{B.24}
(\max_j \nu_{n,j})^2 \leq\sum_j \nu_{n,j}^2 \leq \max \nu_{n,j}
\end{equation}
\eqref{B.22} is equivalent to
\[
\max_j \nu_{n,j} \to 0
\]
\end{remark}

\begin{proof} Fix $m\geq 0$ and let
\begin{equation} \lb{B.25}
g_m (z,w) \equiv \log(\max (m,\abs{z-w}))
\end{equation}
Then
\begin{equation} \lb{B.26}
m^{\sum_j \nu_{n,j}^2} \prod_{j\neq k} \abs{z_j^{(n)} -z_k^{(n)}}^{\nu_{n,j}\nu_{n,k}}
\leq\exp\biggl( \int d\nu_n (z)d\nu_n(w) g_m(z,w)\biggr)
\end{equation}

Now take $n\to\infty$. By \eqref{B.22}, $m^{\sum \nu_{n,j}^2}\to 1$, and by continuity of
$g_m (z,w)$ and the weak convergence,
\begin{equation} \lb{B.27}
\int d\nu_n (z) d\nu_n(w) g_m (z,w)\to \int d\eta (z) d\eta(w) g_m(z,w)
\end{equation}
we have
\begin{equation} \lb{B.28}
\text{LHS of \eqref{B.23}} \leq \exp\biggl( \int d\eta(x) d\eta(y) g_m(z,w)\biggr)
\end{equation}

Now take $m\to 0$ using monotone convergence to get \eqref{B.23}.
\end{proof}

\begin{lemma}\lb{BL.5} Let $E\subset\bbR$. Let $(a,b)\cap E=\emptyset$. Then $T_n(z)$ has at
most one zero in $(a,b)$ which is simple. If $(a,b)\cap \cvh(E)=\emptyset$, $T_n$ has no
zero in $(a,b)$ {\rm{(}}where $\cvh(E)$ is the convex hull of $E${\rm{)}}. In particular,
if $d\eta$ is a limit point of the normalized zero counting measure for $T_n$, then
$\supp (d\eta)\subset E$.
\end{lemma}

\begin{proof} Suppose $x_1, x_2$ are two zeros in $(a,b)$ with $x_1 <x_2$. Then
\[
(z-(x_1-\delta))(z-(x_2+\delta)) = (z-x_1)(z-x_2)-\delta (x_2-x_1)-\delta^2
\]
so uniformly on $E$ where $(z-x_1)(z-x_2)>0$,
\[
\abs{(z-(x_1-\delta))(z-(x_2+\delta))} < \abs{(z-x_1)(z-x_2)}
\]
for $\delta$ small. Thus, $\|T_n(z)\|_E$ is decreased by changing those two zeros.
Similarly, if $x$ is a zero below $\cvh(E)$, $T_n$ is decreased by moving the zero
up slightly. If $x_j$ is a complex zero, $\|T_n(z)\|_E$ is decreased by replacing
$x_j$ by $\Real x_j$.

The final statement is immediate if we note that if $f$ is a continuous function supported
in $(a,b)$, then $\int f\, d\eta =0$.
\end{proof}

\begin{proof}[Proof of Theorem~\ref{TB.1}] We have already proved \eqref{B.7}. The Fekete
points are distinct, so $\nu_{n,j}=1/n$, in the language of Proposition~\ref{PB.4}. So
if we pass to a subsequence for which $d\nu_{n(j)}$ has a weak limit $\eta$, we see
(using $\lim\zeta_n$ exists)
\begin{align}
\lim\zeta_n = \lim_{j\to\infty}\, \zeta_{n(j)}^{(n(j)-1)/n(j)}
&\leq \exp (-\calE (d\eta)) \notag \\
&\leq \exp \bigl( -\inf_{\text{all } d\rho}\, \calE(d\rho)\bigr) \lb{B.27x} \\
&= C(E) \notag
\end{align}

By \eqref{B.7}, $\lim\zeta_n \geq C(E)$, so we have equality in \eqref{B.27x} and $d\eta
=d\rho_E$. Thus, any limit point is $d\rho_E$. By compactness, we have (i).

That leaves the proof of (ii). By the Berstein--Walsh lemma \eqref{A.21b}, for $z\in
\bbC\setminus E$,
\begin{equation} \lb{B.28x}
\f{1}{n}\, \log \abs{T_n(z)} \leq \log \biggl( \f{\|T_n\|_E^{1/n}}{C(E)} \biggr) -
\Phi_{\rho_E}(z)
\end{equation}
and similarly for $T_n^R$.

Now let $d\eta$ be a limit point of the normalized density of zeros of $T_n(z)$. By the last
lemma, $d\eta$ is supported on $E$, so \eqref{B.28x} plus $\lim\|T_n\|_\infty^{1/n}=C(E)$ implies
\begin{equation} \lb{B.30}
\Phi_\eta(z) \geq \Phi_{\rho_E}(z)
\end{equation}
for $z\in\bbC\setminus E$. By Theorem~\ref{TA.11}, this implies $d\eta=d\rho_E$.  Thus,
$d\rho_E$ is the only limit point of the zeros, and so the limit is $d\rho_E$.
\end{proof}

\medskip
{\it Note added in proof.} Since completion of this manuscript, I have found a result (to appear in
``Regularity and the Ces\`aro--Nevai class", in prep.) relevant to the subject
of the current review.  In the simplest case, it states that if a measure
has $[-2,2]$ as its essential support and is regular, then \eqref{1.24a} holds.


\end{document}